\documentclass[11pt]{amsart}
\usepackage{amscd}
\usepackage[all]{xy}
\usepackage{graphicx}
\usepackage{amsmath}
\usepackage{amsfonts}
\usepackage{amssymb}
\usepackage{amsthm}
\usepackage{latexsym}
\usepackage{dsfont}
\usepackage{slashed}
\usepackage{enumerate}
\usepackage{bbm}

\usepackage{soul}
\usepackage{comment}
\usepackage{color}
\usepackage{rotating}
\usepackage{url}

\usepackage{mathrsfs}
\usepackage{verbatim}
\usepackage{mathtools}

\usepackage{longtable}
\usepackage[table]{xcolor}

\newcommand{\ce}{\coloneqq}

\newcommand{\gr}{\rowcolor[gray]{.89}}
\newcommand{\f}[1]{{$\scriptstyle{#1}$}}
\newcommand{\y}[1]{{$\scriptstyle{#1}$}}
\newcommand{\ttt}{{\mathfrak t}}
\newcommand{\uuu}{{\mathfrak u}}
\renewcommand{\sp}{\mathfrak{\mathop{sp}}}

\definecolor{brown}{RGB}{150,100,0}

\definecolor{purple}{RGB}{128,0,128}

\definecolor{grey}{RGB}{128,128,128}

\numberwithin{equation}{subsection}
\theoremstyle{plain}
\newtheorem{lemma}[equation]{Lemma}
\newtheorem{proposition}[equation]{Proposition}
\newtheorem{theorem}[equation]{Theorem}
\newtheorem{corollary}[equation]{Corollary}

\theoremstyle{definition}
\newtheorem{definition}[equation]{Definition}
\newtheorem{remark}[equation]{Remark}
\newtheorem{example}[equation]{Example}

\newtheorem{subsec}[equation]{}
\newtheorem{construction}[equation]{Construction}
\newtheorem*{proposition*}{Proposition}
\newtheorem*{lemma*}{Lemma}

\newcommand{\R}{{\mathds R}}
\newcommand{\C}{{\mathds C}}
\newcommand{\F}{{\mathds F}}
\newcommand{\Z}{{\mathds Z}}

\newcommand{\Q}{{\mathds Q}}

\newcommand{\K}{{\mathbbm k}}

\renewcommand{\P}{{\mathbb P}}

\newcommand{\Aa}{{\mathcal A}}

\newcommand{\Cc}{{\mathcal C}}    %configuration space

\newcommand{\Oo}{{\mathcal O}}

\newcommand{\Tt}{{\mathcal T}}

\newcommand{\Zz}{{\mathcal Z}}

\newcommand{\ad}{{\rm ad \,}}
\newcommand{\Ad}{{\rm Ad }}

\newcommand{\eps}{{\varepsilon}}

\renewcommand{\a}{{\mathfrak a}}

\newcommand{\z}{{\mathfrak z}}

\newcommand{\g}{{\mathfrak g}}
\newcommand{\gl}{{\mathfrak {gl}}}
\newcommand{\h}{{\mathfrak h}}

\renewcommand{\c}{{\mathfrak c}}

\renewcommand{\u}{{\mathfrak u}}
\newcommand{\su}{\mathfrak{su}}

\newcommand{\Cg}{{\mathfrak C}}

\newcommand{\ssl}{{\mathfrak {sl}}}

\newcommand{\GL}{{\rm GL}}
\newcommand{\SL}{{\rm SL}}

\newcommand{\hs}{\kern 0.8pt}
\newcommand{\hsss}{\kern 1.2pt}
\newcommand{\hm}{\kern -0.8pt}
\newcommand{\hmm}{\kern -1.2pt}

\newcommand{\hssh}{\kern 1.2pt}
\newcommand{\hshs}{\kern 1.6pt}
\newcommand{\hssss}{\kern 2.0pt}

\newcommand{\into}{\hookrightarrow}
\newcommand{\isoto}{\overset{\sim}{\to}}

\newcommand{\labelto}[1]{\xrightarrow{\makebox[1.5em]{\scriptsize ${#1}$}}}

\newcommand{\upalpha}{\hs^\alpha\kern-0.5pt}

\newcommand{\CC}{{\bf C}}
\newcommand{\GG}{{\bf G}}
\newcommand{\HH}{{\bf H}}

\newcommand{\MM}{{\bf M}}

\newcommand{\YY}{{\bf Y}}
\newcommand{\ZZ}{{\bf Z}}

\DeclareMathOperator{\Lie}{Lie}

\newcommand{\sT}{{\mathcal T}}

\newcommand{\Gtil}{{\widetilde{G}}
}

\newcommand{\Gal}{{\rm Gal}}
\newcommand{\Zm}{{\mathcal Z}}

\newcommand{\SmallMatrix}[1]{\text{{\tiny\arraycolsep=0.4\arraycolsep\ensuremath
    {\begin{pmatrix}#1\end{pmatrix}}}}}

\newcommand{\ov}{\overline}

\newcommand{\Ho}{{\mathrm{H}\kern 0.3pt}}
\newcommand{\Zl}{{\mathrm{Z}\kern 0.2pt}}
\newcommand{\Bd}{{\mathrm{B}\kern 0.2pt}}

\newcommand{\id}{{\rm id}}

\newcommand{\Cl}{{\rm Cl}}

\newcommand{\Stab}{{\rm Stab}}
\newcommand{\Aut}{{\rm Aut}}

\newcommand{\diag}{{\rm diag}}

\newcommand{\Nm}{{\mathcal N}}

\newcommand{\upgam}{{\hs^\gamma\hm}}

\newcommand{\Am}{{\mathcal A}}
\newcommand{\Fm}{{\mathcal F}}
\newcommand{\OOm}{{\mathcal O}}

\newcommand{\pbar}{{\ov p}}
\newcommand{\gbar}{{\ov g}}

\newcommand{\ve}{\varepsilon}
\newcommand{\vk}{\varkappa}

\newcommand{\vet}{{\tilde\ve}}

\newcommand{\EEE}{{\sf E}}
\newcommand{\AAA}{{\sf A}}

\newcommand{\PPi}{{\scriptscriptstyle\Pi}}
\newcommand{\rR}{{\scriptscriptstyle\R}}
\newcommand{\cC}{{\scriptscriptstyle{\C}}}
\newcommand{\kK}{{\scriptstyle\K}}

\begin{document}

\title{Classification of real trivectors in dimension nine}

\author{Mikhail  Borovoi}
\address{Borovoi: Raymond and Beverly Sackler School of Mathematical Sciences,
Tel Aviv University, 6997801 Tel Aviv, Israel}
\email{borovoi@tauex.tau.ac.il}

\author{Willem A. de Graaf}
\address{De Graaf: Department of Mathematics, University of Trento, Povo (Trento), Italy}
\email{degraaf@science.unitn.it}

\author{H\^ong V\^an L\^e}
\address{L\^e: Institute of Mathematics, Czech  Academy of Sciences,
Zitna 25, 11567 Praha 1, Czech Republic}
\email{hvle@math.cas.cz}

\thanks{ Borovoi was partially supported
by the Israel Science Foundation (grant 870/16).
De Graaf was partially supported by an Australian Research Council
grant, identifier DP190100317.
Research  of L\^e was supported  by  GA\v CR-project 18-01953J and	 RVO: 67985840.}

\date{\today}



\begin{comment}
% Abstract for arXiv:

We classify real trivectors in dimension 9.
The corresponding classification over the field C of complex numbers was
obtained by Vinberg and Elashvili in 1978.
One of the main tools used for their classification was the construction of the representation
of SL(9,C) on the space of complex trivectors of C^9 as a theta-representation
corresponding to a Z/3Z-grading of the simple complex Lie algebra of type E_8.
This divides the trivectors into three groups: nilpotent, semisimple, and mixed trivectors.
Our classification follows the same pattern.
We use Galois cohomology, first and second, to obtain the classification over R.

\end{comment}

\begin{abstract}
In this paper we classify real trivectors in dimension 9.
The corresponding classification over the field $\C$ of complex numbers was
obtained by Vinberg and Elashvili in 1978.
One of the main tools used for their classification was the construction of the representation
of $\SL(9,\C)$ on the space of complex trivectors of $\C^9$ as a theta-representation
corresponding to a $\Z/3\Z$-grading of the simple complex Lie algebra of type $\EEE_8$.
This divides the trivectors into three groups: nilpotent, semisimple, and mixed trivectors.
Our classification follows the same pattern.
We use Galois cohomology, first and second, to obtain the classification over $\R$.
\end{abstract}

\maketitle
\tableofcontents

%Main category: Representation theory.
%The other categories:
%Group Theory,
%Rings and Algebras,
%Differential Geometry

%%%%%%%%%%%%%%%%%%%%%%%%%%%%%%%%%%%%%%%intro.tex%%%%%%%%%%%%%%%%%%%%%%%%%%

\section{Introduction}

Let $V$ be an $n$-dimensional vector space over a field
$\K$. The group $\GL(V)$ naturally acts on the spaces $\bigwedge^m V$.
The elements of $\bigwedge^2 V$ are called bivectors. The
elements of $\bigwedge^3 V$ are called trivectors. If the ground field $\K$ is
$\R$ or $\C$, then the orbits of $\GL(V)$
on the space of bivectors can be listed for all $n$ (see Gurevich \cite{Gurevich1964}, \S 34).
The situation for trivectors is much more complicated and a lot of
effort has gone into finding orbit classifications for particular $n$.
For $n$ up to $5$ it is straightforward to obtain the orbits of
$\GL(V)$ on the space of trivectors (see \cite{Gurevich1964}, \S 35).
The classification for higher $n$ and $\K=\C$ started with the thesis of Reichel
\cite{Reichel}, who classified the orbits for $n=6$ and with a few
omissions also for $n=7$. In 1931 Schouten \cite{Schouten31} published
a classification for $n=7$. In 1935
Gurevich \cite{Gurevich1935a} obtained a classification for $n=8$
(see also \cite{Gurevich1964}, \S 35). In these cases the number of orbits is
finite. This ceases to be the case for $n\geq 9$. Vinberg and Elashvili \cite{VE1978}
classified the orbits of trivectors for $n=9$ under the group
$\SL(V)$. In this classification there are several parametrized
families of orbits. The maximum number of parameters of such a family is 4.

More recently classifications have appeared for different fields.
For $n=6$,
Revoy \cite{Revoy} gave a classification for arbitrary field $\K$.
For $n=7$, Westwick \cite{Westwick} classified the trivectors for $\K=\R$,
Cohen and Helminck \cite{CH1988} treated the case of a perfect field $\K$
of cohomological dimension $\le 1$ (which includes finite fields),
and Noui and Revoy \cite{NR1994} treated the cases of a
perfect field of cohomological dimension $\le 1$ and of a $p$-adic field.
For $n=8$,
Djokovi\'c \cite{Djokovic1983} treated the case $\K=\R$,
Noui \cite{Noui1997} treated the case of an algebraically closed field $\K$
of arbitrary characteristic,
and Midoune and Noui \cite{MN2013} treated the case of a finite field.
For $n=9$,
Hora and Pudl\'ak \cite{HP2020} treated the case of the finite field $\F_2$ of two elements.
In the present paper we give a classification of trivectors under
the action of $\SL(V)$ for $n=9$ and $\K=\R$.

For their classification, Vinberg and Elashvili used a particular construction of
the action of $\SL(V)$ on $\bigwedge^3 V$ ($\dim V=9$, $\K=\C$). They
considered the complex Lie algebra $\g^\cC$ of type $\EEE_8$ and the corresponding
adjoint group $G$ (here $G$ is equal to the automorphism
group of $\g^\cC$). This Lie algebra has a $\Z_3$-grading $\g^\cC = \g_{-1}^\cC
\oplus \g_0^\cC\oplus \g_1^\cC$ such that $\g_0^\cC \cong \ssl(9,\C)$.
Let $G_0$ be the connected algebraic subgroup of $G$ with Lie algebra
$\g_0^\cC$. Since the Lie algebra $\g_0^\cC$ preserves $\g_1^\cC$ when acting on $\g^\cC$,
so does the group $G_0$.
An isomorphism $\psi^\cC\colon \ssl(9,\C)\to \g_0^\cC$ lifts to a surjective morphism of
algebraic groups $\Psi^\cC\colon \SL(9,\C)\to G_0$. This defines an action of
$\SL(9,\C)$ on $\g_1^\cC$, and it turns out that $\g_1^\cC \cong \bigwedge^3 \C^9$
as $\SL(9,\C)$-modules. This construction pertains to Vinberg's
theory of $\theta$-groups (\cite{Vinberg1975}, \cite{Vinberg1976},
\cite{Vinberg1979}). We use this theory to study
the orbits of trivectors when $n=9$. Among the technical tools that this
makes available, we mention the following:
\begin{itemize}
\item The elements of $\g_1^\cC$ have a {\em Jordan decomposition}, that is,
  every $x\in \g_1^\cC$ can be uniquely written as $x=s+n$ where $s\in \g_1^\cC$
  is {\em semisimple} (i.e., its $G_0$-orbit is closed), $n\in \g_1^\cC$ is
  {\em nilpotent} (i.e., the closure of its orbit contains 0) and $[s,n]=0$.
  This naturally splits the orbits into three types according to whether the
  elements of the orbit are nilpotent, semisimple, or mixed (that is,
  neither semisimple nor nilpotent). Thus the classification problem splits
  into three subproblems. We remark that an element $y\in\g_1^\cC$ is semisimple
  (respectively nilpotent) if and only if the linear map $\ad y \colon \g^\cC\to
  \g^\cC$ is semisimple (respectively nilpotent).
\item  Any nilpotent $e\in \g_1^\cC$ lies in a {\em homogeneous $\ssl_2$-triple},
  meaning that there are $h\in \g_0^\cC$, $f\in \g_{-1}^\cC$ with
  $[h,e]=2e$, $[h,f]=-2f$, $[e,f]=h$. The $G_0$-orbits of $e$ and of the
  triple $(h,e,f)$ determine each other.
  Furthermore, $e$ lies in a {\em carrier algebra} and the theory of
  these algebras can be used to classify the nilpotent orbits
  (see Vinberg \cite{Vinberg1979}).
\item  A {\em Cartan subspace} is a maximal subspace $\Cg$  of $\g_1^\cC$
 consisting of commuting semisimple elements.
Any other Cartan subspace is conjugate to $\Cg$ under the action of $G_0$.
 Each semisimple orbit has a
  point in $\Cg$. Furthermore, two elements of $\Cg$ are $G_0$-conjugate
  if and only if they are conjugate under the finite group
  $W\ce \Nm_{G_0}(\Cg)/\Zm_{G_0}(\Cg)$ called the {\em Weyl
  group of the graded Lie algebra} $\g^\cC$, or the {\em little Weyl group.}
\item Fix a semisimple element $s\in\g_1^\cC$, let $\Zm_s$ denote the stabilizer
  of $s$ in $\SL(9,\C)$, and let $\z_{\g^\cC}(s)$ denote the centralizer of $s$ in
  $\g^\cC$. Then the grading of $\g^\cC$ induces a grading
  of its Lie subalgebra  $\z_{\g^\cC}(s)$,
  and  the classification of the orbits of mixed elements $s+e$,
  where $e$ is nilpotent, reduces to the classification of the {\em nilpotent}
  $\Zm_s$-orbits in the graded Lie algebra $\z_{\g^\cC}(s)$.
\end{itemize}

Over the real numbers we use a similar construction of the action of $\SL(V)$
on $\bigwedge^3 V$. There is a Chevalley
basis of $\g^\cC$ such that the $\Z_3$-grading is defined over $\R$ (that is,
the spaces $\g_i^\cC$ for $i=-1,0,1$ all have bases whose elements are
$\R$-linear combinations of the given Chevalley basis of $\g^\cC$).  For
$i=-1,0,1$  we denote by $\g_i$ the real spans of these bases. Then $\g_0$ is isomorphic
to $\ssl(9,\R)$. We can define $\psi^\cC$ in such a way that it restricts to
an isomorphism $\psi \colon \ssl(9,\R)\to \g_0$. Then $\Psi^\cC$ is defined over
$\R$ and restricts to a surjective homomorphism $\Psi \colon \SL(9,\R) \to
G_0(\R)$. In this way $\g_1$ becomes an $\SL(9,\R)$-module isomorphic to
$\bigwedge^3 \R^9$. In Section \ref{sec:gradedLie} we give details of this
construction as well as a summary of a number of results that are useful for the
classification of nilpotent and semisimple elements.

As in the earlier works on classification of trivectors
\cite{Revoy} and \cite{Djokovic1983},
our main workhorse for classifying the $\SL(9,\R)$-orbits on
$\bigwedge^3 \R^9$ is Galois cohomology. Very briefly this amounts to the
following. Let $\Oo$ be an $\SL(9,\C)$-orbit in $\bigwedge^3 \C^9$. We are
interested in determining the $\SL(9,\R)$-orbits that are contained in
$\Oo\cap\bigwedge^3 \R^9$. It can happen that this intersection is empty
(we say that $\Oo$ has no real points). In that case we discard
$\Oo$. On the other hand, if we know a real point $x$ of $\Oo$,
then we consider the stabilizer $\Zm_x$ of $x$ in $\SL(9,\C)$. The
$\SL(9,\R)$-orbits contained in $\Oo$ are in bijection with the first Galois
cohomology set $\Ho^1 \Zm_x\ce \Ho^1(\R, \Zm_x)$.
Moreover, from an explicit cocycle representing  a cohomology class in
$\Ho^1 \Zm_x$\hs, we can effectively compute
a representative of the corresponding real orbit.
Section \ref{sec:galcohom} has further details on Galois
cohomology over $\R$.

In Section \ref{sec:methods} we describe the methods that we have developed
to obtain the classification over $\R$.
For the nilpotent $\SL(9,\R)$-orbits we calculate as described above.
From the explicit representatives in \cite{VE1978}, it follows
immediately that every nilpotent $\SL(9,\C)$-orbit in
$\bigwedge^3 \C^9$ has a real point.
However, instead of a real {\em point} $e$ in the orbit, we work with a real homogeneous
{\em $\ssl_2$-triple} $t=(h,e,f)$ containing it. Since nilpotent orbits correspond
bijectively to orbits of homogeneous $\ssl_2$-triples, it suffices to
compute the stabilizer $\Zm_t$ of the triple $t$ and to determine the set
$\Ho^1 \Zm_t$.

In \cite{VE1978} an explicit Cartan subspace $\Cg$ of $\g_1^\cC$ is
given. In $\Cg$ seven {\em canonical sets} are defined
with the property that every semisimple orbit has a point in one of the
canonical sets. Moreover, two elements of the same canonical set have
the same stabilizer in $\SL(9,\C)$. Let $\Fm$ be such a canonical set.
Let $\Nm$ be its normalizer in $\SL(9,\C)$ ($\Nm$ is the set of $g\in
\SL(9,\C)$ with $gp\in \Fm$ for all $p\in \Fm$) and let $\Zm$ be its
centralizer in $\SL(9,\C)$ ($\Zm$ is the set of all $g\in \SL(9,\C)$ with
$gp=p$ for all $p\in \Fm$). Write $\Aa=\Nm/\Zm$. In Section \ref{sec:semsim}
we show how to divide the complex orbits of the elements of $\Fm$ having real
points into several classes that are in bijection with the Galois cohomology
set $\Ho^1 \Aa$.
For each class we can explicitly find a real representative
of each orbit in this class.
Let $p$ be such a real representative, and let $\Zm_p$ be its centralizer
in $\SL(9,\C)$. Then the real orbits contained in the $\SL(9,\C)$-orbit of
$p$ are in bijection with $\Ho^1 \Zm_p$.

Section \ref{subs:mixed} is devoted to the elements of mixed type.
To classify the orbits consisting of
mixed elements, we fix a real semisimple element $p$ and
consider the problem of listing the $\SL(9,\R)$-orbits of mixed elements
having a representative of the form $p+e$ where $e$ is nilpotent and
$[p,e]=0$. Let $\Zm_p(\R)$ denote the stabilizer of $p$ in $\SL(9,\R)$.
Let $\a=\z_{\g}(p)$ be the centralizer of $p$ in $\g$. Note that the grading
of $\g$ induces a grading on $\a$. Then our problem amounts to classifying the
nilpotent $\Zm_p(\R)$-orbits in $\a_1$. In principle this can be done with
Galois cohomology in the same way as for the nilpotent $\SL(9,\R)$-orbits
in $\g_1$. In this case, however, there is one extra problem. If the reductive
subalgebra $\a$ is not split over $\R$, then it can happen that certain complex
nilpotent orbits in $\a_1^\cC$ (where $\a^\cC=\z_{\g^\cC}(p)$\hs) do not have
real points. In Section \ref{subs:mixed} we develop a method for checking
whether a complex nilpotent orbit has real points and for finding one in
the affirmative case. Among other things, this method uses calculations
with {\em second} (abelian) Galois cohomology.

The main results of this paper are the tables in Section \ref{sec:tables}
containing representatives of the orbits of $\SL(9,\R)$ in $\bigwedge^3 \R^9$.

In order to obtain the results of this paper, essential use of the computer
has been made. In \cite[Section 5]{BGL2021} we give details of the
computational methods that we used. Here we just mention that
we have used the computer algebra system {\sf GAP}4 \cite{gap4}
and its package {\sf SLA}  \cite{sla}. For the computations involving Gr\"obner
bases we have used the computer algebra system {\sf Singular} \cite{DGPS}.

In order to write a paper of reasonable length, we have omitted the details
of our computations, focussing instead on the theoretical background and the
methods that we have developed. The computations are described in detail
in our paper \cite{BGL2021}. That paper contains also much more material on
Galois cohomology and graded Lie algebras.

Finally we would like to  add   a few   more  motivations
for the problem of classification   of trivectors of $\R^9$
 and  put our method  in a broader  perspective.
Special  geometries  defined  by   differential  forms   are of   central  significance
in Riemannian geometry and in physics; see
\cite{HL1982},   \cite{Hitchin2000}, \cite{Joyce2007};  see also \cite{LV2020} for a survey.
One  of  important problems  in geometry  defined  by  differential forms  is to  understand  the orbit space
of the standard $\GL (n, \R)$-action on the vector space $\bigwedge^k (\R^n)^*$ of alternating $k$-forms on $\R^n$,
which is  in  a bijection   with the orbit space
of  the standard $\GL (n, \R)$-action on the space of $k$-vectors  $\bigwedge^k \R^{n}$.
For general  $k$ and $n$, this classification problem is intractable.
In the  case of trivectors of $\R^8$ and  $\R^9$,
this classification problem is  tractable  thanks  to  its formulation
as an  equivalent problem  for   graded  semisimple Lie algebras,
discovered  first  by Vinberg  in his  works \cite{Vinberg1975}, \cite{Vinberg1976}.
Semisimple graded Lie algebras and their  adjoint orbits
 play an important role  in geometry  and supersymmetries; see
 \cite{WG1968I}, \cite{WG1968II}, \cite{Nahm1978}, \cite{Kac1995}, \cite{CS2009}, \cite{Swann1991}.
\bigskip

{\bf Acknowledgements.}
We  are indebted   to Alexander Elashvili for his discussions with us
on his work  with Vinberg \cite{VE1978} and related  problems
in graded   semisimple Lie algebras over years;
it was he who brought  us together  for working on this paper. We  are  grateful  to
Domenico Fiorenza for his helpful comments on an early version  of this paper.

\subsection{Notation  and conventions}

Here we briefly describe some notation and conventions that we use in the
paper.

By $\Z$, $\Q$, $\R$, $\C$ we denote, respectively, the ring of integers
and the fields of rational, real, and complex numbers.
By $i\in \C$ we denote an imaginary unit.

By $\mu_n$ we denote the cyclic group consisting of the $n$-th roots of
unity in $\C$. Occasionally we denote  by $\mu_n$ the set of matrices consisting
of scalar matrices with an $n$-th root of unity on the diagonal. From the
context it will be clear what is meant.

If $G$ is an algebraic group and  $A$  is a subset in the Lie algebra $\g$
of $G$ then by $\Nm_G(A)$, $\Zm_G(A)$ we denote the normalizer
and centralizer in $G$ of $A$.
If $v$ is an element of a $G$-module, then
$\Zm_G(v)$ denotes the stabilizer of $v$ in $G$.
Furthermore, $\z_\g(A)$ denotes the centralizer of $A$ in $\g$.

We write $\Gamma$ for the group $\Gal(\C/\R)=\{1,\gamma\}$, where $\gamma$ is
the complex conjugation. By a $\Gamma$-group we mean a  group $A$ on which
$\Gamma$ acts by automorphisms. Clearly, a $\Gamma$-group is a pair
$(A,\sigma)$,where $A$ is a group and  $\sigma\colon A\to A$ is an
automorphism of $A$ such that $\sigma^2={\rm id}_A$\hs.
If $a\in A$, we write $\upgam a$  for $\sigma(a)$.

By a {\em linear algebraic group over $\R$} (for brevity $\R$-group) we mean a pair $\MM=(M,\sigma)$,
where $M$ is a linear algebraic group over $\C$, and $\sigma$ is an anti-regular involution of $M$;
see Subsection \ref{sss:real-group}.
Then $\Gamma$ naturally acts on $M$ (namely, the complex conjugation $\gamma$
acts as $\sigma$). In other words, $M$ naturally is a $\Gamma$-group.

In this paper the spaces $\bigwedge^3 \C^9$ and $\bigwedge^3 \R^9$ play
a major role. We denote by  $e_1,\ldots,e_9$ the standard basis of
$\C^9$ (or of $\R^9$). Then by $e_{ijk}$ we denote the basis vector
$e_i\wedge e_j\wedge e_k$ of $\bigwedge^3 \C^9$ (or of $\bigwedge^3 \R^9$).

%%%%%%%%%%%%%%%%%%%%%%%%%%%%%%%%%%%%%tables.tex%%%%%%%%%%%%%%%%%%%%%%%%%%%%%%

\section{The tables}\label{sec:tables}

In this section we list the representatives of the orbits of $\SL(9,\R)$ in
$\bigwedge^3 \R^9$. Except for the semisimple orbits, we list the representatives
in tables. We also give various centralizers in the real Lie algebra
$\ssl(9,\R)$. These centralizers are always reductive. For the semisimple
subalgebras we use standard notation for the real forms of the complex
simple Lie algebras.
By $\ttt$ we denote a 1-dimensional subalgebra of $\ssl(9,\R)$
spanned by a semisimple  matrix all of whose eigenvalues are real. By $\uuu$ we denote a
1-dimensional subalgebra of $\ssl(9,\R)$ spanned by a semisimple  matrix all of whose
eigenvalues are imaginary.

In the sequel we occasionally write about elements without specifying the
set to which they belong and also about orbits without further specification.
By an element we usually mean an element of $\bigwedge^3 \C^9$.
By a  real element we mean an element of $\bigwedge^3 \R^9$.
By a complex orbit we mean an orbit of  $\SL(9,\C)$ in $\bigwedge^3 \C^9$,
and by a real orbit we mean an orbit  of $\SL(9,\R)$ in $\bigwedge^3 \R^9$.
We say that an orbit (real or complex) is nilpotent or semisimple or
of mixed type, if its elements are of the given type.

\subsection{The nilpotent orbits}

Table \ref{tab:orbitreps} contains representatives of the nilpotent orbits of
$\SL(9,\R)$ on
$\bigwedge^3 \R^9$. In the first column we list the number of the complex orbit
as contained in Table 6 of \cite{VE1978}. The second column has the
characteristic of the complex orbit (see Remark \ref{rem:char}).
In the third column we give the dimension $d$ of the
complex orbit. The fourth column lists the
rank $d_s$ of a trivector $e$ in the complex orbit, that is,
the minimal dimension of a subspace $U$ of $\C^9$ such that
$e\in\bigwedge^3 U$. Let $t=(e,h,f)$ be a homogeneous $\ssl_2$-triple
containing a representative $e$ of the complex orbit. Let $\Zm_{\SL(9,\C)}(t)$ denote
the stabilizer in $\SL(9,\C)$ of $t$. The fifth and sixth columns have,
respectively, the size of the component group $K$ of $\Zm_{\SL(9,\C)}(t)$ and
a description of the structure of $K$. We omit the latter
when the component group  has order 1 or prime order. The seventh column has representatives
of all $\SL(9,\R)$-orbits contained in the complex orbit. The last column
has a description of the structure of the centralizer in $\ssl(9,\R)$ of a real
homogeneous $\ssl_2$-triple $t=(e,h,f)$ containing the real representative $e$
in the same line.

\smallskip

\begin{longtable}{|r|l|l|l|l|l|l|l|}
\caption{Real nilpotent orbits}\label{tab:orbitreps}
\endfirsthead
\hline
\endhead
\hline
\endfoot
\endlastfoot

\hline
No. &\quad Char.  & \f{d}  & \f{d_s} &\f{|K|} &\ \ \f{K} & \qquad\quad Representatives & $\mathfrak{z}(t)$ \\
\hline

\gr
1 & {\tiny 6 6 6 6 6 6 6 12 }& {\tiny 80} & \y9 & \y3   &
  &${\scriptstyle e_{126}+e_{135}-e_{234}+e_{279}+e_{369}+e_{459}+e_{478}+e_{568}}$
& ${\scriptstyle 0}$  \\
2 & {\tiny 6 6 6 0 6 6 6 6 }& {\tiny 79} & \y9 & \y3   &
  &${\scriptstyle e_{126}+e_{145}-e_{234}+e_{279}+e_{369}-e_{378}+e_{478}+e_{568}}$
& ${\scriptstyle 0}$  \\
\gr
3 & {\tiny 6 6 6 0 6 0 6 6 }& {\tiny 78} & \y9 & \y3   &
  &${\scriptstyle e_{126}+e_{145}-e_{235}+e_{279}+e_{369}-e_{378}+e_{478}+e_{568}}$
&  ${\scriptstyle 0}$  \\
4 & {\tiny 6 0 6 0 6 6 0 6 }& {\tiny 78} & \y9 & \y6   & \f{\mu_6}
  &${\scriptstyle e_{127}+e_{145}-e_{234}+e_{279}+e_{369}-e_{378}+e_{478}+e_{568}}$
  &  ${\scriptstyle 0}$ \\
  & & &    & &
  &${\scriptstyle e_{127}-e_{135}-e_{146}+e_{234}-e_{369}+e_{378}+e_{459}+e_{568}}$
& ${\scriptstyle 0}$  \\
\gr
5 & {\tiny 0 6 0 6 0 6 0 6 }& {\tiny 77} & \y9 & \y6   & \f{\mu_6}
  &${\scriptstyle e_{127}+e_{145}-e_{235}+e_{279}+e_{369}-e_{468}+e_{478}+e_{568}}$
&  ${\scriptstyle 0}$  \\
\gr
  & & &    & &
  &${\scriptstyle 2e_{126}+2e_{137}+2e_{235}+2e_{279}-2e_{369}+2e_{459}+2e_{468}+e_{578}}$
& ${\scriptstyle 0}$  \\
6 & {\tiny 6 1 5 1 5 6 1 5 }& {\tiny 77} & \y9 & \y1   &
  &${\scriptstyle e_{136}+e_{145}-e_{234}+e_{279}-e_{378}+e_{469}+e_{568}}$
&  ${\scriptstyle \ttt }$ \\
\gr
7& {\tiny 0 6 0 0 6 0 6 0 }& {\tiny 76} & \y9 & \y6   & \f{\mu_6}
  &${\scriptstyle e_{127}+e_{145}-e_{235}+e_{368}-e_{378}-e_{468}+e_{469}+e_{579}}$
& ${\scriptstyle 0}$  \\
\gr
  & & &    & &
  &${\scriptstyle 2e_{127}+2e_{134}-2e_{245}+2e_{368}+\frac12 e_{379}+e_{479}+e_{569}-e_{578}}$
& ${\scriptstyle 0}$  \\
8 & {\tiny 6 0 6 0 6 0 0 6 }& {\tiny 76} & \y9 & \y6   & \f{\mu_6}
  &${\scriptstyle e_{136}+e_{145}-e_{235}+e_{279}-e_{378}+e_{469}+e_{478}+e_{568}}$
  &   ${\scriptstyle 0}$ \\
  & & &    & &
  &${\scriptstyle -e_{135}+e_{146}-e_{179}+e_{236}-e_{245}+e_{369}+e_{378}+e_{459}+e_{568}}$
& ${\scriptstyle 0}$  \\
\gr
9 & {\tiny 0 0 6 0 0 6 0 6 }& {\tiny 76} & \y9 & \y{18}   & \f{\mu_3\times S_3}
  &${\scriptstyle e_{127}+e_{145}-e_{234}+e_{279}+e_{369}+e_{568}-e_{578}+e_{678}}$
& ${\scriptstyle 0}$    \\
\gr
  & & &    & &
  &${\scriptstyle e_{125}+e_{136}-e_{147}+2e_{234}-2e_{279}-2e_{459}+2e_{568}-2e_{578}}$
&   ${\scriptstyle 0}$  \\
10 & {\tiny 2 2 2 2 2 4 2 2 }& {\tiny 75} & \y9 & \y9   & \f{\mu_3\times \mu_3}
  &${\scriptstyle e_{127}+e_{136}+e_{145}-e_{234}+e_{379}+e_{469}+e_{478}+e_{568}}$
& ${\scriptstyle 0}$  \\
\gr
11 & {\tiny 6 1 5 1 5 0 1 5 }& {\tiny 75} & \y9 & \y1   &
  &${\scriptstyle e_{136}+e_{145}-e_{236}+e_{279}-e_{378}+e_{469}+e_{568}}$
&  ${\scriptstyle \ttt }$ \\
12 & {\tiny 0 1 5 0 1 5 1 5 }& {\tiny 75} & \y9 & \y1   &
  &${\scriptstyle e_{127}+e_{145}-e_{234}+e_{369}+e_{479}+e_{568}-e_{578}}$
&  ${\scriptstyle \ttt }$ \\
\gr
13 & {\tiny 1 5 1 1 4 1 5 1 }& {\tiny 75} & \y9 & \y1   &
  &${\scriptstyle e_{126}+e_{145}-e_{235}+e_{379}+e_{469}+e_{478}+e_{568}}$
&  ${\scriptstyle \ttt }$ \\
14 & {\tiny 2 2 2 2 2 2 2 2 }& {\tiny 74} & \y9 & \y9   & \f{\mu_9}
  &${\scriptstyle e_{127}+e_{136}+e_{145}-e_{235}+e_{379}+e_{469}+e_{478}+e_{568}}$
& ${\scriptstyle 0}$  \\
\gr
15 & {\tiny 6 0 0 0 6 0 0 6 }& {\tiny 74}   & \y9 & \y6   & \f{\mu_6}
  &${\scriptstyle e_{136}+e_{145}-e_{235}+e_{279}+e_{469}+e_{478}-e_{578}}$
  & ${\scriptstyle \ttt }$  \\
\gr
  & & &   & &
  &${\scriptstyle e_{134}-e_{156}-e_{179}+e_{235}-e_{246}-e_{369}+e_{378}+e_{459}}$
  & ${\scriptstyle \uuu }$  \\
\gr
 & & &   & &
 &${\scriptstyle e_{134}+e_{156}-e_{179}+e_{235}-e_{246}+e_{369}+e_{378}+e_{459}}$
& ${\scriptstyle \uuu }$  \\
16  & {\tiny 1 1 4 1 1 5 1 4 }& {\tiny 74} & \y9 & \y3   &
  &${\scriptstyle e_{127}+e_{136}+e_{145}-e_{234}+e_{379}+e_{469}+e_{568}}$
&  ${\scriptstyle \ttt }$ \\
\gr
17 & {\tiny 0 6 0 0 0 6 0 0 }& {\tiny 73} & \y9 & \y{18 } & \f{\mu_3\times S_3}
  &${\scriptstyle e_{137}+e_{146}-e_{245}-e_{268}+e_{278}+e_{368}+e_{479}+e_{569}}$
  &  ${\scriptstyle 0}$ \\
\gr
  & & &    & &
  &${\scriptstyle -e_{126}-e_{147}+e_{279}+2e_{345}-e_{368}-e_{469}+4e_{569}+e_{578}}$
& ${\scriptstyle 0}$  \\
18  & {\tiny 2 2 0 2 2 2 2 2 }& {\tiny 73} & \y9 & \y9   & \f{\mu_3\times\mu_3}
  &${\scriptstyle e_{127}+e_{136}+e_{145}-e_{235}+e_{379}+e_{469}+e_{478}-e_{578}}$
&  ${\scriptstyle 0}$ \\
\gr
19 & {\tiny 6 0 1 0 5 0 1 5 }& {\tiny 73} & \y9 & \y2   &
  &${\scriptstyle e_{136}+e_{145}-e_{236}+e_{279}-e_{378}+e_{478}+e_{569}}$
&  ${\scriptstyle \ttt }$ \\
\gr
  & & &   & &
  &${\scriptstyle 2e_{135}-2e_{146}-e_{236}-e_{245}+e_{279}+2e_{378}+2e_{569}}$
& ${\scriptstyle \ttt }$  \\
20 & {\tiny 3 0 3 3 0 6 0 3 }& {\tiny 73} & \y9 & \y3   &
  &${\scriptstyle e_{127}+e_{145}-e_{234}+e_{379}+e_{469}+e_{478}+e_{568}}$
& ${\scriptstyle \ssl(2,\R)}$  \\
\gr
21 & {\tiny 1 4 1 1 1 3 1 1 }& {\tiny 72} & \y9 & \y1   &
  &${\scriptstyle e_{127}+e_{136}-e_{245}+e_{379}+e_{469}+e_{478}+e_{568}}$
& ${\scriptstyle \ttt }$  \\
22 & {\tiny 1 4 0 1 1 4 1 1 }& {\tiny 72} & \y9 & \y3   &
  &${\scriptstyle e_{127}+e_{136}-e_{235}+e_{379}+e_{469}+e_{478}-e_{578}}$
& ${\scriptstyle \ttt }$  \\
\gr
23 & {\tiny 0 3 0 3 0 3 3 0 }& {\tiny 72} & \y9 & \y6   & \f{\mu_6}
  &${\scriptstyle e_{127}+e_{136}+e_{145}-e_{235}-e_{468}+e_{479}+e_{568}}$
&  ${\scriptstyle \ttt }$ \\
\gr
  & & &    & &
  &${\scriptstyle -2e_{126}+2e_{137}+e_{145}-2e_{234}-e_{468}-e_{479}+2e_{569}-2e_{578}}$
& ${\scriptstyle \uuu }$  \\
24 & {\tiny 0 0 6 0 0 0 0 6 }& {\tiny 72} & \y9 & \y6   & \f{\mu_6}
  &${\scriptstyle e_{127}+e_{136}+e_{145}-e_{235}+e_{379}+e_{469}-e_{578}}$
  &  ${\scriptstyle \ttt }$ \\
  & & &    & &
  &${\scriptstyle -e_{125}-e_{137}+e_{146}-e_{179}+e_{247}+e_{269}+e_{345}+e_{459}+e_{578}}$
  & ${\scriptstyle \uuu }$  \\
  & & &   & &
  &${\scriptstyle -e_{125}+e_{137}+e_{146}+e_{179}+e_{247}+e_{269}-e_{345}-e_{459}+e_{578}}$
  & ${\scriptstyle \uuu }$  \\
\gr
25 & {\tiny 2 0 4 2 0 6 0 4 }& {\tiny 72} & \y9 & \y1   &
  &${\scriptstyle e_{127}+e_{145}-e_{234}+e_{379}+e_{469}+e_{568}}$
&  ${\scriptstyle \ssl(2,\R) + \ttt }$ \\
26 & {\tiny 3 0 3 0 3 0 3 0 }& {\tiny 71} & \y9 & \y6   & \f{\mu_6}
  &${\scriptstyle e_{127}+e_{136}-e_{245}-e_{378}+e_{479}+e_{568}+e_{569}}$
  & ${\scriptstyle \ttt }$  \\
  & & &    & &
  &${\scriptstyle e_{127}-e_{136}-e_{145}+e_{235}-e_{246}+e_{378}-e_{479}-2e_{568}}$
& ${\scriptstyle \uuu }$  \\
\gr
27 & {\tiny 0 1 5 0 0 1 0 5 }& {\tiny 71} & \y9 & \y2   &
  &${\scriptstyle e_{127}+e_{136}-e_{245}+e_{379}+e_{469}+e_{568}-e_{578}}$
& ${\scriptstyle \ttt }$  \\
\gr
  & & &   & &
  &${\scriptstyle e_{126}+e_{136}+e_{147}+e_{279}+e_{345}-e_{379}+e_{469}-e_{578}}$
& ${\scriptstyle \ttt }$  \\
28 & {\tiny 1 1 2 1 1 1 1 4 }& {\tiny 71} & \y9 & \y1   &
  &${\scriptstyle e_{127}+e_{136}+e_{145}-e_{235}+e_{379}+e_{469}+e_{678}}$
& ${\scriptstyle \ttt }$  \\
\gr
29 & {\tiny 0 4 0 2 0 4 2 0 }& {\tiny 71} & \y9 & \y2   &
  &${\scriptstyle e_{127}+e_{136}-e_{235}-e_{468}+e_{479}+e_{568}}$
& ${\scriptstyle 2\ttt }$  \\
\gr
  & & &   & &
  &${\scriptstyle -2e_{126}+2e_{137}+2e_{235}+2e_{468}-2e_{479}-e_{569}-e_{578}}$
& ${\scriptstyle \ttt +\uuu }$  \\
30 & {\tiny 6 1 0 1 4 1 0 5 }& {\tiny 71} & \y9 & \y1   &
  &${\scriptstyle e_{146}-e_{179}-e_{236}-e_{245}-e_{378}+e_{569}}$
& ${\scriptstyle \ssl(2,\R) + \ttt }$  \\
\gr
31 & {\tiny 0 0 0 0 6 0 0 0 }& {\tiny 70} & \y9 &\y{360} & \y{\mu_3\times S_5}
  &${\scriptstyle e_{137}-e_{246}-e_{247}+e_{348}-e_{358}+e_{368}+e_{458}+e_{569}}$
& ${\scriptstyle 0}$  \\
\gr
  & & &    & &
  &${\scriptstyle -e_{134}-2e_{167}+e_{245}-2e_{358}+2e_{368}-2e_{378}+2e_{379}-e_{458}+e_{469}}$
&  ${\scriptstyle 0}$ \\
\gr
  & & &      & &
  &${\scriptstyle -e_{135}-e_{245}-2e_{267}+2e_{368}+e_{378}-2e_{468}-2e_{479}-e_{569}}$
&  ${\scriptstyle 0}$ \\
32 & {\tiny 2 2 2 2 0 2 0 2 }& {\tiny 70} & \y9 & \y3   &
  &${\scriptstyle e_{127}+e_{146}-e_{236}-e_{245}+e_{379}+e_{478}+e_{568}}$
&   ${\scriptstyle \ttt }$ \\
\gr
33 & {\tiny 1 0 5 0 1 0 1 4 }& {\tiny 70} & \y9 & \y2   &
  &${\scriptstyle e_{127}+e_{136}-e_{245}+e_{379}+e_{479}+e_{568}}$
& ${\scriptstyle 2\ttt }$  \\
\gr
  & & &   & &
  &${\scriptstyle -e_{127}+e_{135}-e_{146}-e_{236}-e_{245}+e_{379}+e_{568}}$
& ${\scriptstyle \ttt +\uuu}$  \\
34 & {\tiny 2 0 2 2 0 2 0 4 }& {\tiny 70} & \y9 & \y1   &
  &${\scriptstyle e_{127}+e_{145}-e_{236}+e_{379}+e_{469}-e_{578}}$
& ${\scriptstyle 2\ttt }$  \\
\gr
35 & {\tiny 2 0 2 0 2 0 2 2 }& {\tiny 69} & \y9 & \y{18}   & \f{\mu_6\times\mu_3}
  &${\scriptstyle e_{127}+e_{136}-e_{245}+e_{379}+e_{479}+e_{569}-e_{578}+e_{678}}$
&  ${\scriptstyle 0}$ \\
\gr
  & & &    & &
  &${\scriptstyle -e_{127}+e_{135}+e_{146}-e_{236}+e_{245}-e_{379}+e_{569}-e_{578}}$
&  ${\scriptstyle 0}$ \\
36 & {\tiny 0 0 1 0 5 0 0 1 }& {\tiny 69} & \y9 &  \y{6}   & \f{S_3}
  &${\scriptstyle e_{136}-e_{245}+e_{379}+e_{479}+e_{568}-e_{578}+e_{678}}$
  & ${\scriptstyle \ttt}$  \\
  & & &    & &
  &${\scriptstyle -e_{135}+e_{147}+e_{236}+e_{379}+e_{459}-e_{578}+e_{678}}$
& ${\scriptstyle \ttt}$  \\
\gr
37  & {\tiny 1 1 4 1 0 1 0 4 }& {\tiny 69} & \y9 & \y1   &
  &${\scriptstyle e_{127}+e_{146}-e_{236}-e_{245}+e_{379}+e_{568}}$
&  ${\scriptstyle 2\ttt }$ \\
38 & {\tiny 2 1 1 1 1 1 1 2 }& {\tiny 68} & \y9 & \y3   &
  &${\scriptstyle e_{127}+e_{146}-e_{236}-e_{245}+e_{379}+e_{569}-e_{578}}$
&  ${\scriptstyle \ttt }$ \\
\gr
39 & {\tiny 1 0 1 0 4 0 1 1 }& {\tiny 68} & \y9 & \y6   & \f{\mu_6}
  &${\scriptstyle e_{136}-e_{245}+e_{379}+e_{479}+e_{569}-e_{578}+e_{678}}$
& ${\scriptstyle \ttt }$  \\
\gr
  & & &    & &
  &${\scriptstyle e_{135}-e_{146}+e_{236}+e_{245}+e_{479}-e_{569}+e_{578}}$
&  ${\scriptstyle \ttt }$ \\
40  & {\tiny 0 1 0 1 4 0 1 0 }& {\tiny 68} & \y9 & \y2   &
  &${\scriptstyle e_{137}-e_{236}-e_{245}-e_{468}+e_{478}+e_{569}}$
  &  ${\scriptstyle 2\ttt}$ \\
  & & &   & &
  &${\scriptstyle -e_{136}+e_{237}-2e_{245}-e_{468}+2e_{479}-e_{569}-2e_{578}}$
&   ${\scriptstyle \ttt + \uuu}$\\
\gr
41 & {\tiny 1 0 1 1 2 1 1 1 }& {\tiny 67} & \y9 & \y3   &
  &${\scriptstyle e_{137}+e_{145}-e_{236}+e_{479}+e_{569}-e_{578}+e_{678}}$
& ${\scriptstyle \ttt }$  \\
42 & {\tiny 0 3 0 0 0 3 0 3 }& {\tiny 67} & \y9 & \y3   &
  &${\scriptstyle e_{127}+e_{136}-e_{245}+e_{379}+e_{479}+e_{569}+e_{678}}$
&  ${\scriptstyle \ssl(2,\R)}$ \\
\gr
43 & {\tiny 3 0 0 0 3 0 3 0 }& {\tiny 67} & \y9 & \y6   & \f{\mu_6}
  &${\scriptstyle e_{127}+e_{136}-e_{245}-e_{378}+e_{478}+e_{579}+e_{679}}$
&  ${\scriptstyle \ssl(2,\R)}$ \\
\gr
  & & &    & &
  &   ${\scriptstyle e_{127}-e_{134}-e_{156}-e_{236}-e_{245}+e_{578}-e_{679}}$
& ${\scriptstyle \ssl(2,\R)}$  \\
44 & {\tiny 1 1 0 1 3 1 0 1 }& {\tiny 67} &  \y9 & \y1   &
  &${\scriptstyle e_{137}+e_{146}-e_{236}-e_{245}+e_{478}+e_{569}}$
& ${\scriptstyle 2\ttt}$  \\
\gr
45 & {\tiny 3 0 3 3 0 0 0 3 }& {\tiny 67} & \y9 & \y3   &
  &${\scriptstyle e_{127}+e_{146}-e_{245}+e_{379}+e_{478}+e_{568}}$
& ${\scriptstyle 2\hs\ssl(2,\R) }$  \\
46 & {\tiny 1 1 1 1 1 1 1 1 }& {\tiny 66} & \y9 & \y1   &
  &${\scriptstyle e_{137}+e_{146}-e_{236}-e_{245}+e_{479}+e_{569}-e_{578}}$
&  ${\scriptstyle \ttt}$ \\
\gr
47 & {\tiny 0 2 0 0 2 2 0 2 }& {\tiny 66} & \y9 & \y2   &
  &${\scriptstyle e_{136}+e_{147}-e_{245}+e_{379}+e_{569}+e_{678}}$
  & ${\scriptstyle 2\ttt }$  \\
\gr
  & & &   & &
  &${\scriptstyle -e_{136}-e_{147}+e_{157}-e_{235}-e_{379}-e_{469}-e_{569}-e_{678}}$
& ${\scriptstyle \ttt +\uuu }$  \\
48 & {\tiny 2 0 0 0 4 0 2 0 }& {\tiny 66} & \y9 & \y2   &
  &${\scriptstyle e_{136}-e_{245}-e_{378}+e_{478}+e_{579}+e_{679}}$
  & ${\scriptstyle \ssl(2,\R) + \ttt }$  \\
  & & &   & &
  &${\scriptstyle e_{134}-e_{156}-e_{236}+e_{245}+e_{578}-e_{679}}$
& ${\scriptstyle \ssl(2,\R) + \ttt }$  \\
\gr
49 & {\tiny 3 0 1 0 2 1 2 0 }& {\tiny 66} & \y9 & \y1   &
  &${\scriptstyle e_{127}+e_{156}-e_{236}-e_{245}-e_{378}+e_{479}}$
&  ${\scriptstyle \ssl(2,\R) + \ttt }$ \\
50 & {\tiny 2 0 4 2 0 0 0 4 }& {\tiny 66} & \y9 & \y1   &
  &${\scriptstyle e_{127}+e_{146}-e_{245}+e_{379}+e_{568}}$
& ${\scriptstyle 2\hs\ssl(2,\R) + \ttt }$  \\
\gr
51 & {\tiny 1 1 1 1 0 1 1 2 }& {\tiny 65} & \y9 & \y3   &
  &${\scriptstyle e_{137}+e_{146}-e_{236}-e_{245}+e_{479}+e_{569}+e_{678}}$
&  ${\scriptstyle \ttt}$ \\
52 & {\tiny 1 1 0 1 1 2 1 1 }& {\tiny 65} & \y9 & \y1   &
  &${\scriptstyle e_{137}+e_{146}-e_{235}+e_{479}+e_{579}+e_{678}}$
& ${\scriptstyle 2\ttt }$  \\
\gr
53 & {\tiny 2 0 1 0 3 1 1 0 }& {\tiny 65} & \y9 & \y1   &
  &${\scriptstyle e_{156}-e_{236}-e_{245}-e_{378}+e_{479}}$
& ${\scriptstyle \ssl(2,\R) + 2\ttt}$  \\
54 & {\tiny 0 3 0 0 3 0 0 0 }& {\tiny 64} & \y9 & \y3 &
  &${\scriptstyle e_{147}+e_{156}-e_{237}-e_{246}-e_{345}+e_{368}+e_{579}}$
&  ${\scriptstyle \ssl(2,\R)}$ \\
\gr
55 & {\tiny 2 0 2 0 0 2 0 2 }& {\tiny 64} & \y9 & \y1 &
  &${\scriptstyle e_{137}+e_{156}-e_{236}-e_{245}+e_{479}-e_{578}}$
&  ${\scriptstyle 2\ttt}$ \\
56 & {\tiny 0 0 0 3 0 3 0 0 }& {\tiny 64} & \y9 & \y3  &
  &${\scriptstyle e_{146}+e_{157}-e_{237}+e_{458}+e_{478}+e_{569}}$
  &  ${\scriptstyle 2\hs\ssl(2,\R)}$ \\
  & & &  & &
  &${\scriptstyle e_{145}+2e_{167}-e_{235}-2e_{469}+2e_{478}+e_{568}+e_{579}}$
  & ${\scriptstyle \ssl(2,\R)+\su(2)}$  \\
  \gr
57 & {\tiny 0 0 0 0 0 0 0 6 }& {\tiny 64} & \y8 &\y6 & \f{\mu_6}
  &${\scriptstyle e_{127}+e_{136}-e_{245}+e_{379}+e_{479}+e_{569}}$
  & ${\scriptstyle \ssl(3,\R)}$  \\
  \gr
  & & &    & &
  &${\scriptstyle e_{123}-e_{146}+e_{179}+2e_{247}+2e_{259}+2e_{357}+2e_{369}}$
  & ${\scriptstyle \su(1,2)}$  \\
  \gr
  & & &    & &
  &${\scriptstyle e_{123}-e_{156}+e_{179}+e_{247}+e_{269}+e_{349}+e_{359}+e_{367}}$
  &  ${\scriptstyle \su(3)}$ \\
58 & {\tiny 1 1 1 1 1 0 1 1 }& {\tiny 63} & \y9 &\y1 &
  &${\scriptstyle e_{137}-e_{246}-e_{345}+e_{479}+e_{569}-e_{578}}$
& ${\scriptstyle 2\ttt}$  \\
\gr
59 & {\tiny 0 4 0 0 2 0 0 0 }& {\tiny 63} & \y9 & \y1 &
  &${\scriptstyle e_{147}+e_{156}-e_{237}-e_{246}+e_{368}+e_{579}}$
&  ${\scriptstyle \ssl(2,\R) + \ttt}$ \\
60 & {\tiny 0 0 3 0 0 0 3 0 }& {\tiny 63} & \y9 & \y{6} &  \f{\mu_6}
  &${\scriptstyle e_{137}+e_{146}-e_{236}-e_{245}+e_{568}+e_{679}}$
  & ${\scriptstyle \ssl(2,\R) + \ttt}$  \\
  & & &     & &
  &${\scriptstyle -e_{136}-e_{145}-e_{147}+e_{235}-e_{237}+e_{246}-e_{568}-e_{579}}$
  &  ${\scriptstyle \ssl(2,\R) + \uuu}$ \\
  & & &    & &
  &${\scriptstyle e_{127}+e_{136}-e_{145}+e_{235}+e_{246}-e_{347}-e_{568}+e_{579}}$
  &  ${\scriptstyle \su(2) + \uuu}$ \\
\gr
61 & {\tiny 0 0 0 0 1 0 0 5 }& {\tiny 63} &\y8 & \y1 &
  &${\scriptstyle e_{137}+e_{146}-e_{236}-e_{245}+e_{479}+e_{569}}$
&  ${\scriptstyle \ssl(2,\R)+\ttt}$ \\
\gr
  & & &    & &
  &${\scriptstyle -2e_{134}-e_{145}+e_{167}-e_{246}+\frac{1}{2}e_{257}-2e_{369}+e_{479}-e_{569}}$
&  ${\scriptstyle \su(2)+\ttt}$ \\
62 & {\tiny 0 1 0 2 0 3 0 1 }& {\tiny 63} & \y9 &  \y2   &
  &${\scriptstyle e_{146}-e_{235}+e_{479}+e_{579}+e_{678}}$
  & ${\scriptstyle \ssl(2,\R) + 2\ttt}$  \\
  & & &   & &
  &${\scriptstyle -e_{146}+2e_{157}+e_{234}-e_{479}-2e_{569}+2e_{678}}$
& ${\scriptstyle \ssl(2,\R) + \ttt + \uuu}$  \\
\gr
63 & {\tiny 1 1 1 0 1 0 2 1 }& {\tiny 62} & \y9 & \y1   &
  &${\scriptstyle e_{127}+e_{146}-e_{236}-e_{345}+e_{579}+e_{678}}$
& ${\scriptstyle \ssl(2,\R) + \ttt}$  \\
64 & {\tiny 0 1 2 0 1 0 2 0 }& {\tiny 62} & \y9 &\y2  &
  &${\scriptstyle e_{137}-e_{246}-e_{345}+e_{568}+e_{579}}$
  &  ${\scriptstyle 3\ttt}$ \\
  & & &    & &
  &${\scriptstyle -e_{136}-e_{147}+e_{237}-e_{246}+e_{345}+e_{568}+e_{579}}$
  & ${\scriptstyle \ttt+2\uuu}$  \\
\gr
65 & {\tiny 0 2 0 0 2 0 0 2 }& {\tiny 61} & \y9 &\y{18} & \f{\mu_3\times S_3}
  &${\scriptstyle e_{137}-e_{246}-e_{247}-e_{345}+e_{569}+e_{678}}$
& ${\scriptstyle 2\ttt}$  \\
\gr
  & & &  & &
  &${\scriptstyle -e_{137}+e_{247}-e_{256}-e_{345}+e_{469}+e_{579}+e_{678}}$
&  ${\scriptstyle \ttt+\uuu}$ \\
66 & {\tiny 1 2 1 1 0 0 1 1 }& {\tiny 61} & \y9 & \y1  &
  &${\scriptstyle e_{137}-e_{246}+e_{479}+e_{569}-e_{578}}$
&  ${\scriptstyle \ssl(2,\R) + 2\ttt}$ \\
\gr
67 & {\tiny 0 0 1 0 0 0 1 4 }& {\tiny 61} &\y8 & \y2 &
  &${\scriptstyle e_{137}+e_{146}-e_{236}-e_{245}+e_{579}}$
  &  ${\scriptstyle \ssl(2,\R) + 2\ttt}$ \\
  \gr
  & & &    & &
  &${\scriptstyle -e_{125}-e_{126}+e_{147}-e_{237}-e_{345}+e_{346}+e_{679}}$
  & ${\scriptstyle \ssl(2,\R) +  \ttt+\uuu}$  \\
  \gr
  & & &    & &
  &${\scriptstyle -e_{127}-e_{136}+e_{145}+e_{235}+e_{246}-e_{347}-e_{679}}$
  &   ${\scriptstyle \su(2) + \ttt+\uuu}$\\
68 & {\tiny 1 0 1 1 0 3 1 0 }& {\tiny 61} & \y9 & \y1   &
  &${\scriptstyle e_{147}+e_{156}-e_{234}-e_{578}+e_{679}}$
& ${\scriptstyle 2\hs\ssl(2,\R) + \ttt}$  \\
\gr
69  & {\tiny 1 1 0 1 1 0 1 1 }& {\tiny 60} & \y9 & \y1 &
  &${\scriptstyle e_{156}-e_{237}-e_{246}-e_{345}+e_{479}+e_{678}}$
& ${\scriptstyle 2\ttt}$  \\
70 & {\tiny 0 1 0 0 1 0 0 4 }& {\tiny 60} &\y8 & \y{6} & \y{S_3}
  &${\scriptstyle e_{137}-e_{246}-e_{247}-e_{345}+e_{569}}$
  &  ${\scriptstyle 3\ttt}$ \\
  & & &    & &
  &${\scriptstyle -e_{136}-e_{147}-e_{257}+e_{345}+e_{379}-e_{469}}$
& ${\scriptstyle 2\ttt +\uuu}$  \\
\gr
71 & {\tiny 0 2 2 0 0 0 2 0 }& {\tiny 59} & \y9 & \y2 &
  &${\scriptstyle e_{137}-e_{246}+e_{568}+e_{579}}$
& ${\scriptstyle 2\hs\ssl(2,\R) + 2 \ttt}$  \\
\gr
  & & &    & &
  &${\scriptstyle -e_{126}-e_{147}-e_{237}-e_{346}+e_{568}+e_{579}}$
& ${\scriptstyle \ssl(2,\C) +  \ttt+\uuu}$  \\
72 & {\tiny 1 0 1 1 0 1 1 0 }& {\tiny 58} & \y9 & \y9 & \f{\mu_3\times \mu_3}
  &${\scriptstyle e_{147}+e_{156}-e_{237}-e_{246}-e_{345}-e_{578}+e_{679}}$
&  ${\scriptstyle \ssl(2,\R)}$ \\
\gr
73 & {\tiny 2 0 1 0 1 1 0 1 }& {\tiny 58} & \y9 & \y1 &
  &${\scriptstyle e_{137}-e_{256}-e_{346}+e_{479}-e_{578}}$
&  ${\scriptstyle \ssl(2,\R) + 2 \ttt}$ \\
74 & {\tiny 1 0 0 1 0 0 1 3 }& {\tiny 58} &\y8 & \y1   &
  &${\scriptstyle e_{156}-e_{237}-e_{246}-e_{345}+e_{479}}$
&  ${\scriptstyle \ssl(2,\R) + 2\ttt}$ \\
\gr
75 & {\tiny 0 3 0 0 0 0 0 3 }& {\tiny 57} & \y9 & \y{18} &\f{\mu_3\times S_3}
  &${\scriptstyle e_{137}-e_{246}-e_{247}+e_{569}+e_{678}}$
  &  ${\scriptstyle 3\hs\ssl(2,\R) }$ \\
\gr
  & & &  & &
  &${\scriptstyle -e_{126}-e_{346}-e_{379}+e_{457}+e_{569}+e_{678}}$
&  ${\scriptstyle \ssl(2,\R)+\ssl(2,\C)}$ \\
76 & {\tiny 0 1 0 0 1 0 1 2 }& {\tiny 56} & \y8 & \y3   &
  &${\scriptstyle e_{147}+e_{156}-e_{237}-e_{246}-e_{345}+e_{679}}$
&  ${\scriptstyle \ssl(2,\R) + \ttt}$ \\
\gr
77 & {\tiny 1 2 0 0 0 1 0 2 }& {\tiny 56} & \y9 & \y1 &
  &${\scriptstyle e_{137}-e_{246}-e_{356}+e_{579}+e_{678}}$
& ${\scriptstyle 2\hs\ssl(2,\R)+ \ttt}$  \\
78 &  {\tiny 0 2 0 0 0 0 0 4 }& {\tiny 56} & \y8 & \y{6}   & \f{S_3}
  &${\scriptstyle e_{137}-e_{246}-e_{247}+e_{569}}$
  &  ${\scriptstyle 3\hs\ssl(2,\R) + \ttt}$ \\
  & & &  & &
  &${\scriptstyle -e_{126}+e_{179}-e_{257}+e_{346}+e_{569}}$
& ${\scriptstyle \ssl(2,\R)+\ssl(2,\C) + \ttt}$  \\
\gr
79  & {\tiny 0 0 2 0 0 4 0 0 }& {\tiny 56} & \y9 & \y1 &
  &${\scriptstyle e_{157}-e_{234}+e_{568}+e_{679}}$
&  ${\scriptstyle 2\hs\ssl(3,\R)}$ \\
80 & {\tiny 1 1 0 1 0 1 0 1 }& {\tiny 55} & \y9 & \y3 &
  &${\scriptstyle e_{147}-e_{237}-e_{256}-e_{346}+e_{579}+e_{678}}$
&  ${\scriptstyle \ssl(2,\R) + \ttt}$ \\
\gr
81 & {\tiny 0 0 2 0 0 2 0 0 }& {\tiny 55} & \y9 & \y3   &
  &${\scriptstyle e_{157}-e_{237}-e_{246}-e_{345}+e_{568}+e_{679}}$
& ${\scriptstyle \ssl(3,\R)}$  \\
82 & {\tiny 1 1 0 0 0 1 0 3 }& {\tiny 55} & \y8 & \y1    &
  &${\scriptstyle e_{137}-e_{246}-e_{356}+e_{579}}$
&  ${\scriptstyle 2\hs\ssl(2,\R) + 2\ttt}$ \\
\gr
83 & {\tiny 0 0 0 3 0 0 0 0 }& {\tiny 54} & \y9 & \y3 &
  &${\scriptstyle e_{157}-e_{247}-e_{256}-e_{346}+e_{458}+e_{679}}$
& ${\scriptstyle \sp(4,\R)}$  \\
84 & {\tiny 1 0 0 1 0 1 0 2 }& {\tiny 54} & \y8 & \y1    &
  &${\scriptstyle e_{147}-e_{237}-e_{256}-e_{346}+e_{579}}$
& ${\scriptstyle \ssl(2,\R) + 2\ttt}$  \\
\gr
85 & {\tiny 0 1 0 2 0 0 0 1 }& {\tiny 53} & \y9 & \y1 &
  &${\scriptstyle e_{147}-e_{256}-e_{346}+e_{579}+e_{678}}$
& ${\scriptstyle 2\hs\ssl(2,\R) + \ttt}$  \\
86 & {\tiny 0 0 0 2 0 0 0 2 }& {\tiny 52} & \y8 & \y2    &
  &${\scriptstyle e_{147}-e_{256}-e_{346}+e_{579}}$
  &  ${\scriptstyle 2\hs\ssl(2,\R)+ 2\ttt}$ \\
  & & &    & &
 &${\scriptstyle -e_{145}-e_{167}+e_{257}-e_{346}-e_{479}-e_{569}}$
& ${\scriptstyle \ssl(2,\C) +\ttt+\uuu}$  \\
\gr
87 & {\tiny 2 1 0 1 0 0 0 2 }& {\tiny 52} & \y9 & \y1 &
  &${\scriptstyle e_{127}+e_{379}-e_{456}+e_{678}}$
&  ${\scriptstyle \ssl(2,\R)+\sp(4,\R) + \ttt}$ \\
88 & {\tiny 0 1 0 1 0 0 1 1 }& {\tiny 51} & \y8 & \y1    &
  &${\scriptstyle e_{157}-e_{247}-e_{256}-e_{346}+e_{679}}$
&  ${\scriptstyle \ssl(2,\R) + 2\ttt}$  \\
\gr
89 & {\tiny 2 0 0 1 0 1 0 0 }& {\tiny 49} & \y9 & \y3 &
  &${\scriptstyle e_{157}-e_{237}-e_{456}+e_{478}+e_{679}}$
&  ${\scriptstyle \ssl(2,\R)+\ssl(3,\R)}$ \\
90 & {\tiny 0 0 0 0 0 0 3 0 }& {\tiny 49} & \y7 & \y3    &
  &${\scriptstyle e_{147}+e_{156}-e_{237}-e_{246}-e_{345}}$
  & ${\scriptstyle \ssl(2,\R)+{\sf G}_2^{\rm spl}}$   \\
  & & &   & &
 &${\scriptstyle -e_{123}-e_{145}-e_{167}+e_{246}-e_{257}-e_{347}-e_{356}}$
& ${\scriptstyle \ssl(2,\R)+{\sf G}_2^{\mathrm{c}}}$  \\
\gr
91 & {\tiny 2 0 0 1 0 0 0 3 }& {\tiny 49} & \y8 &  \y1    &
  &${\scriptstyle e_{127}+e_{379}-e_{456}}$
&  ${\scriptstyle \ssl(3,\R)+\sp(4,\R) + \ttt}$ \\
92 & {\tiny 1 0 1 0 0 1 0 1 }& {\tiny 48} & \y8 &  \y1   &
  &${\scriptstyle e_{157}-e_{247}-e_{356}+e_{679}}$
&  ${\scriptstyle 2\hs\ssl(2,\R)+ 2\ttt}$ \\
\gr
93 & {\tiny 0 0 0 1 0 0 2 0 }& {\tiny 48} & \y7 &  \y1   &
  &${\scriptstyle e_{157}-e_{247}-e_{256}-e_{346}}$
  & ${\scriptstyle 3\hs\ssl(2,\R) + \ttt}$  \\
\gr
  & & &    & &
  &${\scriptstyle -e_{145}+e_{167}-4e_{246}-e_{257}+e_{347}-e_{356}}$
& ${\scriptstyle \ssl(2,\R)+2\hs\su(2) + \ttt}$  \\
94 & {\tiny 0 1 0 0 0 1 1 0 }& {\tiny 45} & \y7 & \y2    &
  &${\scriptstyle e_{167}-e_{247}-e_{356}}$
  &  ${\scriptstyle 3\hs \ssl(2,\R)+ 2\ttt}$ \\
  & & &   & &
  &${\scriptstyle e_{167}+e_{236}+e_{257}-e_{347}-e_{456}}$
&  ${\scriptstyle \ssl(2,\R)+\ssl(2,\C) + \ttt + \uuu }$ \\
\gr
95 & {\tiny 1 0 0 1 0 0 1 0 }& {\tiny 42} & \y7 & \y1    &
  &${\scriptstyle e_{167}-e_{257}-e_{347}-e_{456}}$
  &  ${\scriptstyle \ssl(2,\R)+\ssl(3,\R) + \ttt}$ \\
96 & {\tiny 0 0 0 0 0 2 0 0 }& {\tiny 38} & \y6 & \y2   &
  &${\scriptstyle -e_{247}-e_{356}}$
  & ${\scriptstyle  3\hs\ssl(3,\R)}$  \\
  & & &    & &
  &${\scriptstyle e_{234}+e_{267}+e_{357}+e_{456}}$
& ${\scriptstyle \ssl(3,\R)+\ssl(3,\C)}$  \\
\gr
97 & {\tiny 0 0 1 0 0 1 0 0 }& {\tiny 37} & \y6 & \y1 &
  &${\scriptstyle -e_{267}-e_{357}-e_{456}}$
&  ${\scriptstyle  2\hs\ssl(3,\R)+ \ttt}$ \\
98 & {\tiny 3 0 0 0 0 0 0 0 }& {\tiny 36} & \y9 & \y3     &
  &${\scriptstyle e_{167}-e_{257}-e_{347}-e_{789}}$
&  ${\scriptstyle \sp(8,\R)}$ \\
\gr
99 &  {\tiny 2 0 0 0 0 0 1 0 }& {\tiny 35} & \y7 & \y1    &
  &${\scriptstyle e_{167}-e_{257}-e_{347}}$
&  ${\scriptstyle \ssl(2,\R)+\sp(6,\R) + \ttt}$ \\
100 & {\tiny 1 0 0 0 1 0 0 0 }& {\tiny 30} & \y5 &\y1    &
  &${\scriptstyle -e_{367}-e_{457}}$
& ${\scriptstyle \ssl(4,\R)+\sp(4,\R) + \ttt}$ \\
\gr
101 & {\tiny 0 0 1 0 0 0 0 0 }& {\tiny 19} & \y3 & \y1     &
  &${\scriptstyle -e_{567}}$
  & ${\scriptstyle \ssl(3,\R)+\ssl(6,\R)}$ \\
\hline
\end{longtable}

\subsection{The semisimple orbits}\label{sec:tabsemsim}

A {\em Cartan subspace} of $\g_1^\cC$ is a maximal subspace of $\g_1^\cC$
consisting of commuting semisimple elements.  In \cite{VE1978} it is shown that
\begin{align*}
  p_1 &= e_{123}+e_{456}+e_{789}\\
  p_2 &= e_{147}+e_{258}+e_{369}\\
  p_3 &= e_{159}+e_{267}+e_{348}\\
  p_4 &= e_{168}+e_{249}+e_{357}
\end{align*}
span a Cartan subspace in $\g_1^\cC$. In this paper we denote it by
$\Cg^\cC$. Each complex semisimple orbit has a point in this space.
This can still be refined a bit: in \cite{VE1978} seven {\em canonical sets}
$\Fm_1^\cC,\ldots,\Fm_7^\cC$ in $\Cg^\cC$ are described with the property that
each complex semisimple orbit has a point in precisely one of the
$\Fm_i^\cC$. By $\Fm_i$ we denote the set of real elements of $\Fm_i^\cC$.

In this paper we divide the real semisimple orbits into two
groups, the orbits of canonical and  noncanonical semisimple elements.
We say that a real semisimple element is canonical if it lies in one of the
$\Fm_i$, or more generally, if it is $\SL(9,\R)$-conjugate to an element of
one of the sets $\Fm_i$. The noncanonical real semisimple elements are those that
are not $\SL(9,\R)$-conjugate to elements of some $\Fm_i$.

Similarly to the complex case, we define a {\em Cartan subspace} of
$\g_1$ to be a maximal subspace consisting of commuting semisimple elements \cite[\S 2]{Le2011}.
Clearly, any semisimple element of $\g_1$ is contained in some Cartan subspace of $\g_1$.
By $\Cg$ we denote the space consisting of real linear combinations of
$p_1,\ldots,p_4$, that is, the set of real points of $\Cg^\cC$. It is clear
that this is a Cartan subspace of $\g_1$.
We show that all Cartan subspaces
in $\g_1$ are conjugate under $\SL(9,\R)$ (Theorem \ref{thm:cartanr}).
It follows that every semisimple $\SL(9,\R)$-orbit has a point in $\Cg$.

Below we give the representatives of the orbits of canonical and noncanonical
semisimple elements. The representatives of the noncanonical orbits do
not lie in $\Cg$. However, as just observed, these representatives are
$\SL(9,\R)$-conjugate to elements of $\Cg$. In each case we also give these
elements of $\Cg$. Note that noncanonical semisimple elements are
$\SL(9,\R)$-conjugate to elements of $\Cg$, {\em but not to elements of
some $\Fm_i$.}

Except for $\Fm_7^\cC$, the sets $\Fm_i^\cC$ are parametrized families of
semisimple elements with at least one complex parameter.
The real semisimple elements also occur in parametrized
families.
For each parametrized family we give a finite matrix group that
determines when different elements of the same family are conjugate.
More precisely, consider a family of semisimple elements $p(\lambda_1,\ldots,
\lambda_m)$, where $\lambda_i\in \R$ must satisfy certain polynomial conditions.
Write $\lambda$ for the column vector consisting of the $\lambda_i$. Let $\mu$
be a second column vector with coordinates $\mu_1,\ldots,\mu_m$.
Then we describe a finite matrix group $\mathcal{G}$ with the property that
$p(\lambda_1,\ldots,\lambda_m)$ and $p(\mu_1,\ldots,\mu_m)$ are
$\SL(9,\R)$-conjugate if and only if there is an element $g\in \mathcal{G}$ with
$g\cdot \lambda = \mu$. We express this by saying that the conjugacy of
the elements $p(\lambda_1,\ldots,\lambda_m)$ is determined by $\mathcal{G}$.

{\em Canonical set $\Fm_1^\cC$:} It consists of trivectors
$$p^{1,1}_{\lambda_1,\lambda_2,\lambda_3,\lambda_4} = \lambda_1 p_1+\lambda_2p_2
+\lambda_3p_3+\lambda_4p_4$$
where the $\lambda_i\in \C$ are such that
\begin{align*}
\lambda_1\lambda_2\lambda_3\lambda_4 &\neq 0\\
(\lambda_2^3+\lambda_3^3+\lambda_4^3)^3 -(3\lambda_2\lambda_3\lambda_4)^3 &\neq 0\\
(\lambda_1^3+\lambda_3^3+\lambda_4^3)^3 +(3\lambda_1\lambda_3\lambda_4)^3 &\neq 0\\
(\lambda_1^3+\lambda_2^3+\lambda_4^3)^3 +(3\lambda_1\lambda_2\lambda_4)^3 &\neq 0\\
(\lambda_1^3+\lambda_2^3+\lambda_3^3)^3 +(3\lambda_1\lambda_2\lambda_3)^3 &\neq 0.
\end{align*}
The real elements in $\Fm_1$ consist of $p^{1,1}_{\lambda_1,\lambda_2,\lambda_3,\lambda_4}$
with $\lambda_i\in\R$ satisfying the same polynomial conditions.

There are no noncanonical real semisimple elements  that are
$\SL(9,\C)$-conjugate to elements of $\Fm_1^\cC$.

The conjugacy of the real elements $p^{1,1}_{\lambda_1,\lambda_2,\lambda_3,\lambda_4}$ in $\Fm_1$ is
determined by a group of order 48, isomorphic to $\GL(2,\mathbb{F}_3)$ and
generated by
$$\SmallMatrix{0&0&0&-1\\0&-1&0&0\\0&0&1&0\\-1&0&0&0},\
\SmallMatrix{0&0&-1&0\\1&0&0&0\\0&-1&0&0\\0&0&0&1}.$$

{\em Canonical set $\Fm_2^\cC$:} It consists of trivectors
$$p^{2,1}_{\lambda_1,\lambda_2,\lambda_3} = \lambda_1 p_1+\lambda_2p_2
-\lambda_3p_3$$
where the $\lambda_i\in \C$ are such that
$$\lambda_1\lambda_2\lambda_3(\lambda_1^3-\lambda_2^3)(\lambda_2^3-\lambda_3^3)
(\lambda_3^3-\lambda_1^3)\big[(\lambda_1^3+\lambda_2^3+\lambda_3^3)^3-
  (3\lambda_1\lambda_2\lambda_3)^3\big]\neq 0.$$
The real elements in $\Fm_2$ consist of $p^{2,1}_{\lambda_1,\lambda_2,\lambda_3}$
where  $\lambda_1,\lambda_2,\lambda_3\in \R$ satisfy the same polynomial
conditions.

Here we have noncanonical real semisimple elements
\begin{align*}
p^{2,2}_{\lambda_1,\lambda_2,\lambda_3}= &\lambda_1(-\tfrac{1}{2}e_{126}-\tfrac{1}{2}e_{349}+2e_{358}-e_{457}+e_{789})+\lambda_2(
    -2e_{137}-\tfrac{1}{4}e_{249}-e_{258}-\tfrac{1}{2}e_{456}-\tfrac{1}{2}e_{689})\\
    &-\lambda_3(-e_{159}-e_{238}-\tfrac{1}{2}e_{247}-\tfrac{1}{2}e_{346}-e_{678}),\quad \lambda_1,\lambda_2,\lambda_3\in \R.
\end{align*}
The trivector $p^{2,2}_{\lambda_1,\lambda_2,\lambda_3}$ is $\SL(9,\C)$-conjugate to the
{\em purely imaginary} trivector $p^{2,1}_{i\lambda_1,i\lambda_2,i\lambda_3}\in\Fm_2^\cC$.
So here the $\lambda_i\in \R$ are required to satisfy
$$\lambda_1\lambda_2\lambda_3(\lambda_1^3-\lambda_2^3)(\lambda_2^3-\lambda_3^3)
(\lambda_3^3-\lambda_1^3)[(\lambda_1^3+\lambda_2^3+\lambda_3^3)^3+
  (3\lambda_1\lambda_2\lambda_3)^3]\neq 0.$$

The trivector $p^{2,2}_{\lambda_1,\lambda_2,\lambda_3}$ is $\SL(9,\R)$-conjugate to
$$\frac{1}{3}\sqrt{3}((\lambda_2-\lambda_3)p_1+(\lambda_1-\lambda_3)p_2
+(\lambda_1+\lambda_2+\lambda_3)p_3+(\lambda_1-\lambda_2)p_4)$$
which lies in $\Cg$.

The conjugacy of the elements $p^{2,1}_{\lambda_1,\lambda_2,\lambda_3}\in \Fm_2$ is
determined by a group isomorphic to the dihedral group of order 12 and
generated by
$$\SmallMatrix{0&1&0\\0&0&1\\1&0&0},\
\SmallMatrix{-1&0&0\\0&0&-1\\0&-1&0}.$$
The conjugacy of the elements $p^{2,2}_{\lambda_1,\lambda_2,\lambda_3}$ is
determined by exactly the same group.

{\em Canonical set $\Fm_3^\cC$:}   It consists of trivectors
$$p^{3,1}_{\lambda_1,\lambda_2} = \lambda_1 p_1+\lambda_2p_2$$
where the $\lambda_i\in \C$ are such that
$\lambda_1\lambda_2(\lambda_1^6-\lambda_2^6)\neq 0$.
The real elements in $\Fm_3$ consist of $p^{3,1}_{\lambda_1,\lambda_2}$
where $\lambda_1,\lambda_2\in \R$ satisfy the same polynomial conditions.

Here we have  noncanonical real semisimple elements
\begin{align*}
p^{3,2}_{\lambda_1,\lambda_2}&=\lambda_1(-\tfrac{1}{2}e_{126}-\tfrac{1}{2}e_{349}+
  2e_{358}-e_{457}+e_{789})+\lambda_2(-2e_{137}-\tfrac{1}{4}e_{249}-e_{258}-
  \tfrac{1}{2}e_{456}-\tfrac{1}{2}e_{689})\\
p^{3,3}_{\lambda_1,\lambda_2}&= \lambda_1(e_{129}-e_{138}+2e_{237}-\tfrac{1}{4}e_{456}-
   2e_{789})+\lambda_2(-\tfrac{1}{2}e_{147}-e_{258}+2e_{369})\\
p^{3,4}_{\lambda,\mu}&=\lambda(e_{147}-2e_{169}-e_{245}+e_{289}-e_{356}-
   \tfrac{1}{2}e_{378})+
   \mu(-e_{124}-e_{136}-\tfrac{1}{2}e_{238}+e_{457}-2e_{569}+e_{789}),
\end{align*}
where $\lambda_1,\lambda_2,\lambda,\mu\in \R$.
These are $\SL(9,\C)$-conjugate, respectively, to the nonreal elements of
$\Fm_3^\cC$
$$i\lambda_1p_1+i\lambda_2p_2,\qquad i\lambda_1 p_1+\lambda_2 p_2,\qquad
(\lambda+i\mu) p_1+(\lambda-i\mu) p_2.$$
So for $p^{3,2}_{\lambda_1,\lambda_2}$ the $\lambda_i$ are required to satisfy
$\lambda_1\lambda_2(\lambda_1^6-\lambda_2^6)\neq 0$. For
$p^{3,3}_{\lambda_1,\lambda_2}$ the $\lambda_i$ are required to satisfy
$\lambda_1\lambda_2(\lambda_1^6+\lambda_2^6)\neq 0$. For $p^{3,4}_{\lambda,\mu}$
the $\lambda,\mu$ are required to satisfy
$\lambda\mu(3\lambda^4-10\lambda^2\mu^2+3\mu^4)\neq 0$.

We have that $p^{3,2}_{\lambda_1,\lambda_2}$, $p^{3,3}_{\lambda_1,\lambda_2}$,
$p^{3,4}_{\lambda,\mu}$ are $\SL(9,\R)$-conjugate to respectively
\begin{align*}
& \frac{1}{3}\sqrt{3}(\lambda_2p_1+\lambda_1p_2+(\lambda_1+\lambda_2)p_3
+(\lambda_1-\lambda_2)p_4,\\
& -\lambda_2p_1 +\frac{1}{3}\sqrt{3} \lambda_1 (p_2+p_3+p_4),\\
& (\lambda+\frac{1}{3}\sqrt{3}\mu)p_2+(-\lambda+\frac{1}{3}\sqrt{3}\mu)p_3-
\frac{2}{3}\sqrt{3} \mu p_4,
\end{align*}
which lie in the Cartan subspace $\Cg$.

The conjugacy of the elements $p^{3,1}_{\lambda_1,\lambda_2}\in \Fm_3$ is
determined by a group isomorphic to the dihedral group of order 8 generated
by
$$\SmallMatrix{-1&0\\0&1},\, \SmallMatrix{0&1\\1&0}.$$
The conjugacy of the elements $p^{3,2}_{\lambda_1,\lambda_2}$ is determined by
exactly the same group. The conjugacy of the elements $p^{3,3}_{\lambda_1,\lambda_2}$
is determined by a group of order 4 generated
by
$$\SmallMatrix{-1&0\\0&1},\, \SmallMatrix{1&0\\0&-1}.$$
The conjugacy of the elements $p^{3,4}_{\lambda,\mu}$ is determined by
exactly the same group of order 4.

{\em Canonical set $\Fm_4^\cC$:} It consists of trivectors
$$p^{4,1}_{\lambda,\mu} = \lambda p_1+\mu (p_3-p_4)$$
where the $\lambda,\mu\in \C$ are such that $\lambda\mu(\lambda^3-\mu^3)
(\lambda^3+8\mu^3)\neq 0$.
The real elements in $\Fm_4$ consist of $p^{4,1}_{\lambda,\mu}$ were
$\lambda,\mu\in\R$ satisfy the same polynomial conditions.
Here there are no noncanonical real semisimple elements.

The conjugacy of the elements $p^{4,1}_{\lambda,\mu}\in \Fm_4$ is determined by the
group consisting of $\SmallMatrix{1&0\\0&1}$, $\SmallMatrix{-1&0\\0&-1}$.

{\em Canonical set $\Fm_5^\cC$:} It consists of trivectors
$$p^{5,1}_{\lambda} = \lambda  (p_3-p_4)$$
where the $\lambda\in \C$ is nonzero.
The real elements in $\Fm_5$ consist of $p^{5,1}_{\lambda}$ with
$\lambda\in\R$ nonzero.

Here we have noncanonical real semisimple elements
$$p^{5,2}_\lambda = \lambda (e_{148}-e_{159}-e_{238}+\tfrac{1}{2}e_{239}
     -\tfrac{1}{2}e_{247}+e_{257}-\tfrac{1}{2}e_{346}-e_{356}-e_{678}-\tfrac{1}{2}e_{679})$$
where $\lambda\in \R$ is nonzero. We have that $p^{5,2}_\lambda$ is
$\SL(9,\C)$-conjugate to the {\em purely imaginary} trivector
$p^{5,1}_{i\lambda}\in\Fm_5$.

Furthermore, $p^{5,2}_\lambda$ is $\SL(9,\R)$-conjugate to
$$\frac{1}{3}\sqrt{3}\lambda (2p_2-p_3-p_4)$$
which lies in the Cartan subspace $\Cg$.

The conjugacy of the elements $p^{5,1}_\lambda\in \Fm_5$ is determined by the
group consisting of $1,-1$. The conjugacy of the elements $p^{5,2}_\lambda$ is
determined by the same group.

{\em Canonical set $\Fm_6^\cC$:} It consists of trivectors
$$p^{6,1}_{\lambda} = \lambda  p_1$$
where the $\lambda\in \C$ is nonzero.
The real elements in $\Fm_6$ consist of $p^{6,1}_{\lambda}$ where
$\lambda\in \R$ is nonzero.

Here we have  noncanonical real semisimple elements
$$p^{6,2}_\lambda =\lambda (-\tfrac{1}{2}e_{126}-\tfrac{1}{2}e_{349}+2e_{358}-
e_{457}+e_{789})$$
where $\lambda\in \R$ is nonzero. We have that $p^{6,2}_\lambda$ is
$\SL(9,\C)$-conjugate to the {\em purely imaginary} trivector
$p^{6,1}_{i\lambda}\in\Fm_6$.

We have that $p^{6,2}_\lambda$ is $\SL(9,\R)$-conjugate to
$$\frac{1}{3}\sqrt{3}\lambda (p_2+p_3+p_4)$$
which lies in the Cartan subspace $\Cg$.

The conjugacy of the elements $p^{6,1}_\lambda$ is determined by the group
consisting of $1,-1$. The conjugacy of the elements $p^{6,2}_\lambda$ is
determined by the same group.

{\em Canonical set $\Fm_7^\Cc$:} this set consists just of 0.

\subsection{The mixed orbits}\label{sec:mixedtables}

Here we give tables of representatives of the real orbits of mixed type.
The representatives of those orbits have a semisimple part which is equal to
one of the semisimple elements listed above. So the semisimple parts are
of the form $p^{2,1}_{\lambda_1,\lambda_2,\lambda_3},p^{2,2}_{\lambda_1,\lambda_2,\lambda_3},
\ldots $. For each such semisimple part we have a table listing the possible
nilpotent parts. Note that the centralizer of
$p^{1,1}_{\lambda_1,\lambda_2,\lambda_3,\lambda_4}$ in $\g^\cC$ is equal to the unique
Cartan subalgebra of $\g^\cC$ that contains $\Cg$, so there are
no mixed elements with $p^{1,1}_{\lambda_1,\lambda_2,\lambda_3,\lambda_4}$ as
semisimple part.

In the previous subsection we also computed elements of the Cartan subspace
$\Cg$ that are $\SL(9,\R)$-conjugate to the elements
$p^{2,2}_{\lambda_1,\lambda_2,\lambda_3}, p^{3,2}_{\lambda_1,\lambda_2},\ldots$. However
we do not work with those, because the nilpotent parts tend to become
rather bulky. For example, the mixed element
$p^{6,2}_\lambda-2e_{137}$ is $\SL(9,\R)$-conjugate to
\begin{align*}
\frac{1}{3}\sqrt{3}\lambda (p_2+p_3+p_4)&+\frac{1}{9}\sqrt{3}(
e_{123}+e_{126}+e_{129}-e_{135}-e_{138}+e_{156}+e_{159}-e_{168}+e_{189}+e_{234}\\
& +e_{237}-e_{246}-e_{249}+e_{267}-e_{279}+e_{345}+e_{348}-e_{357}+e_{378}+e_{456}\\
& +e_{459}-e_{468}+
e_{489}+e_{567}-e_{579}+e_{678}+e_{789}).
\end{align*}

The caption of each table has the semisimple part $p$ of the mixed elements
whose nilpotent parts are listed in the table.
All tables have three columns. The first column has the number of the complex
nilpotent orbit, which corresponds to the numbering in the tables in
\cite{VE1978}. The
second column has the representatives of the real nilpotent orbits. In the
third column we display the isomorphism type of the centralizer in
$\z_{\g_0}(p)$ of a homogeneous $\ssl_2$-triple in $\z_\g(p)$ containing the
nilpotent element on the same line. Among the nilpotent elements we also
include 0; in that case the centralizer is equal to $\z_{\g_0}(p)$.

\bigskip

\begin{longtable}{|r|l|l|}
  \caption{Nilpotent parts of mixed elements with semisimple part
  $p^{2,1}_{\lambda_1,\lambda_2,\lambda_3}$}\label{tab:fam2_1}
\endfirsthead
\hline
\endhead
\hline
\endfoot
\endlastfoot

\hline
No. &\qquad\quad Reps. of nilpotent parts & $\qquad\mathfrak{z}_0(p,h,e,f)$\\
\hline
1 & ${\scriptstyle e_{168}+e_{249}}$ & ${\scriptstyle 0}$ \\
\gr
2 & ${\scriptstyle e_{168}}$ & ${\scriptstyle \ttt}$ \\
3 & ${\scriptstyle  0}$ & ${\scriptstyle 2\ttt}$ \\
\hline
\end{longtable}

\begin{longtable}{|r|l|l|}
    \caption{Nilpotent parts of mixed elements with semisimple part
  $p^{2,2}_{\lambda_1,\lambda_2,\lambda_3}$}\label{tab:fam2_2}
\endfirsthead
\hline
\endhead
\hline
\endfoot
\endlastfoot

\hline
No. &\qquad\quad Reps. of nilpotent parts & $\qquad\mathfrak{z}_0(p,h,e,f)$ \\
\hline
1 & ${\scriptstyle e_{235}+\tfrac{1}{2}e_{279}-\tfrac{1}{2}e_{369}-e_{567}}$ &
${\scriptstyle 0}$\\
\gr
2 & ${\scriptstyle -e_{148}}$ & ${\scriptstyle \uuu}$\\
\gr
  & ${\scriptstyle e_{148}}$ & ${\scriptstyle \uuu}$\\
3 & ${\scriptstyle 0}$ & ${\scriptstyle \ttt+\uuu}$\\
\hline
\end{longtable}

\bigskip\medskip

\begin{longtable}{|r|l|l|}
      \caption{Nilpotent parts of mixed elements with semisimple part
  $p^{3,1}_{\lambda_1,\lambda_2}$}\label{tab:fam3_1}
\endfirsthead
\hline
\endhead
\hline
\endfoot
\endlastfoot

\hline
No. &\qquad\quad Reps. of nilpotent parts & $\qquad\mathfrak{z}_0(p,h,e,f)$ \\
\hline
1 & ${\scriptstyle e_{159}+e_{168}+e_{249}+e_{267}}$ & ${\scriptstyle 0}$ \\
\gr
2 & ${\scriptstyle e_{159}+e_{168}+e_{249}}$ & ${\scriptstyle \ttt}$\\
3 & ${\scriptstyle e_{159}+e_{168}+e_{267}}$ & ${\scriptstyle \ttt}$\\
\gr
4 & ${\scriptstyle e_{159}+e_{168}}$ & ${\scriptstyle 2\ttt}$\\\
5 & ${\scriptstyle e_{159}+e_{267}}$ & ${\scriptstyle 2\ttt}$\\
\gr
6 & ${\scriptstyle e_{168}+e_{249}}$ & ${\scriptstyle 2\ttt}$\\
7 & ${\scriptstyle e_{159}}$ & ${\scriptstyle 3\ttt}$\\
\gr
8 & ${\scriptstyle e_{168}}$ & ${\scriptstyle 3\ttt}$\\
9 & ${\scriptstyle 0}$ & ${\scriptstyle 4\ttt}$\\
\hline
\end{longtable}
\bigskip

\medskip

\begin{longtable}{|r|l|l|}
        \caption{Nilpotent parts of mixed elements with semisimple part
  $p^{3,2}_{\lambda_1,\lambda_2}$}\label{tab:fam3_2}
\endfirsthead
\hline
\endhead
\hline
\endfoot
\endlastfoot

\hline
No. &\qquad\quad Reps. of nilpotent parts & $\qquad\mathfrak{z}_0(p,h,e,f)$ \\
\hline
1 & ${\scriptstyle  e_{235}+e_{238}+\tfrac{1}{2}e_{247}+\tfrac{1}{2}e_{279}+
  \tfrac{1}{2}e_{346}-\tfrac{1}{2}e_{369}-e_{567}+e_{678}}$ & ${\scriptstyle 0}$\\
\gr
2 & ${\scriptstyle e_{159}+e_{235}+\tfrac{1}{2}e_{279}-\tfrac{1}{2}e_{369}-e_{567}}$
& ${\scriptstyle \uuu}$\\
\gr
& ${\scriptstyle -e_{159}-e_{235}-\tfrac{1}{2}e_{279}+\tfrac{1}{2}e_{369}+e_{567}}$
& ${\scriptstyle \uuu}$ \\
3 &  ${\scriptstyle -e_{148}+\tfrac{1}{2}e_{234}+e_{278}+e_{368}+
  \tfrac{1}{2}e_{467}}$ &${\scriptstyle \uuu}$\\
& ${\scriptstyle e_{148}-\tfrac{1}{2}e_{234}-e_{278}-e_{368}-\tfrac{1}{2}e_{467}}$ &
${\scriptstyle \uuu}$\\
\gr
4 & ${\scriptstyle -e_{148}+e_{159}}$ & ${\scriptstyle 2\uuu}$\\
\gr
& ${\scriptstyle -e_{148}-e_{159}}$ & ${\scriptstyle 2\uuu}$\\
\gr
& ${\scriptstyle e_{148}+e_{159}}$ & ${\scriptstyle 2\uuu}$\\
\gr
& ${\scriptstyle e_{148}-e_{159}}$ & ${\scriptstyle 2\uuu}$\\
5 & ${\scriptstyle \tfrac{1}{2}e_{234}+e_{278}+e_{368}+\tfrac{1}{2}e_{467}}$ &
${\scriptstyle \ttt+\uuu}$\\
\gr
6 & ${\scriptstyle e_{235}+\tfrac{1}{2}e_{279}-\tfrac{1}{2}e_{369}-e_{567}}$ &
${\scriptstyle \ttt+\uuu}$\\
7 & ${\scriptstyle -e_{159}}$ &${\scriptstyle \ttt+2\uuu}$\\
& ${\scriptstyle e_{159}}$ & ${\scriptstyle \ttt+2\uuu}$\\
\gr
8 & ${\scriptstyle -e_{148}}$ & ${\scriptstyle \ttt+2\uuu}$\\
\gr
& ${\scriptstyle e_{148}}$ & ${\scriptstyle \ttt+2\uuu}$ \\
9 & ${\scriptstyle 0}$ & ${\scriptstyle 2\ttt+2\uuu}$\\
\hline
\end{longtable}

\bigskip

\begin{longtable}{|r|l|l|}
        \caption{Nilpotent parts of mixed elements with semisimple part
  $p^{3,3}_{\lambda_1,\lambda_2}$}\label{tab:fam3_3}
\endfirsthead
\hline
\endhead
\hline
\endfoot
\endlastfoot

\hline
No. &\qquad\quad Reps. of nilpotent parts & $\qquad\mathfrak{z}_0(p,h,e,f)$ \\
\hline
1 & ${\scriptstyle e_{159}-2e_{249}-2e_{348}+2e_{357}}$ & ${\scriptstyle 0}$\\
\gr
4 & ${\scriptstyle e_{159}+2e_{357}}$ & ${\scriptstyle \ttt+\uuu}$\\
9 & ${\scriptstyle 0}$ & ${\scriptstyle 2\ttt+2\uuu}$\\
\hline
\end{longtable}
\bigskip

\begin{longtable}{|r|l|l|}
  \caption{Nilpotent parts of mixed elements with semisimple part
  $p^{3,4}_{\lambda,\mu}$}\label{tab:fam3_4}
\endfirsthead
\hline
\endhead
\hline
\endfoot
\endlastfoot

\hline
No. &\qquad\quad Reps. of nilpotent parts & $\qquad\mathfrak{z}_0(p,h,e,f)$ \\
\hline
1 & ${\scriptstyle e_{123}-2e_{179}-2e_{259}-2e_{267}-\tfrac{1}{2}e_{349}+e_{357}}$
& ${\scriptstyle 0}$ \\
\gr
2 & ${\scriptstyle e_{123}-2e_{179}-2e_{259}-\tfrac{1}{2}e_{349}+e_{357}}$ &
${\scriptstyle \ttt}$\\
3 & ${\scriptstyle -2e_{267}-\tfrac{1}{2}e_{349}+\tfrac{1}{4}e_{468}}$ &
${\scriptstyle \uuu}$ \\
& ${\scriptstyle 2e_{267}-\tfrac{1}{2}e_{349}-\tfrac{1}{4}e_{468}}$ &
${\scriptstyle \uuu}$\\
\gr
4 & ${\scriptstyle -2e_{349}+e_{468}}$ & ${\scriptstyle \ttt+\uuu}$\\
\gr
& ${\scriptstyle -2e_{349}-e_{468}}$ & ${\scriptstyle \ttt+\uuu}$\\
5 & ${\scriptstyle -2e_{267}-\tfrac{1}{2}e_{349}}$ & ${\scriptstyle \ttt+\uuu}$\\
\gr
6 & ${\scriptstyle e_{123}-2e_{179}-2e_{259}+e_{357}}$ & ${\scriptstyle 2\ttt}$\\
7 & ${\scriptstyle e_{349}}$ & ${\scriptstyle 2\ttt+\uuu}$\\
\gr
8 & ${\scriptstyle e_{468}}$ &  ${\scriptstyle 2\ttt+\uuu}$\\
\gr
& ${\scriptstyle -e_{468}}$ &  ${\scriptstyle 2\ttt+\uuu}$\\
9 & ${\scriptstyle 0}$ & ${\scriptstyle 3\ttt+\uuu}$\\
\hline
\end{longtable}
\bigskip

\begin{longtable}{|r|l|l|}
          \caption{Nilpotent parts of mixed elements with semisimple part
  $p^{4,1}_{\lambda,\mu}$}\label{tab:fam4_1}
\endfirsthead
\hline
\endhead
\hline
\endfoot
\endlastfoot

\hline
No. &\qquad\quad Reps. of nilpotent parts & $\qquad\mathfrak{z}_0(p,h,e,f)$ \\
\hline
1 & ${\scriptstyle e_{149}+e_{167}+e_{258}+e_{347} }$ & ${\scriptstyle 0}$\\
\gr
2 & ${\scriptstyle e_{149}+e_{158}+e_{167}+e_{248}+e_{257}+e_{347} }$ &
${\scriptstyle 0}$\\
3 & ${\scriptstyle e_{147}+e_{258} }$ & ${\scriptstyle 0}$\\
& ${\scriptstyle e_{147}-e_{158}-e_{248}-e_{257} }$ & ${\scriptstyle 0}$\\
\gr
4 & ${\scriptstyle e_{148}+e_{157}+e_{247} }$ & ${\scriptstyle \ttt}$  \\
5 & ${\scriptstyle e_{147} }$ & ${\scriptstyle \ssl(2,\R)}$ \\
\gr
6 & ${\scriptstyle 0}$ & ${\scriptstyle \ssl(3,\R)}$\\
\hline
\end{longtable}

\begin{longtable}{|r|l|l|}
            \caption{Nilpotent parts of mixed elements with semisimple part
  $p^{5,1}_{\lambda}$}\label{tab:fam5_1}
\endfirsthead
\hline
\endhead
\hline
\endfoot
\endlastfoot

\hline
No. &\qquad\quad Reps. of nilpotent parts & $\qquad\mathfrak{z}_0(p,h,e,f)$ \\
\hline
1 & ${\scriptstyle e_{123}+e_{149}+e_{167}+e_{258}+e_{347}+e_{456}}$ &
${\scriptstyle 0}$ \\
\gr
2 & ${\scriptstyle e_{123}+e_{149}+e_{158}+e_{167}+e_{248}+e_{257}+e_{347}+e_{456}}$
& ${\scriptstyle 0}$\\
3 & ${\scriptstyle e_{123}+e_{149}+e_{167}+e_{258}+e_{347}}$ &
${\scriptstyle \ttt}$ \\
\gr
4 & ${\scriptstyle e_{123}+e_{147}+e_{258}+e_{456}}$ & ${\scriptstyle 0}$ \\
\gr
& ${\scriptstyle e_{123}-2e_{148}-2e_{157}-2e_{247}+2e_{258}+e_{456}}$ &
${\scriptstyle 0}$ \\
5 & ${\scriptstyle e_{123}+e_{149}+e_{158}+e_{167}+e_{248}+e_{257}+e_{347}}$ &
${\scriptstyle \ttt}$\\
\gr
6 & ${\scriptstyle e_{149}+e_{167}+e_{258}+e_{347}}$ & ${\scriptstyle 2\ttt}$ \\
7 & ${\scriptstyle e_{123}+e_{148}+e_{157}+e_{247}+e_{456}}$ &
${\scriptstyle \ttt}$\\
\gr
8 & ${\scriptstyle e_{123}+e_{147}+e_{258}}$ & ${\scriptstyle \ttt}$\\
\gr
& ${\scriptstyle e_{123}-2e_{148}-2e_{157}-2e_{247}+2e_{258}}$ &
${\scriptstyle \ttt}$\\
9 & ${\scriptstyle e_{149}+e_{158}+e_{167}+e_{248}+e_{257}+e_{347}}$ &
${\scriptstyle 2\ttt}$\\
\gr
10 & ${\scriptstyle e_{123}+e_{148}+e_{157}+e_{247}}$ & ${\scriptstyle 2\ttt}$\\
11 & ${\scriptstyle e_{147}+e_{258}}$ & ${\scriptstyle \ttt}$ \\
& ${\scriptstyle e_{147}-e_{158}-e_{248}-e_{257}}$ & ${\scriptstyle \ttt}$ \\
\gr
12 & ${\scriptstyle e_{123}+e_{147}+e_{456}}$ & ${\scriptstyle \ssl(2,\R)}$\\
13 & ${\scriptstyle e_{148}+e_{157}+e_{247}}$ & ${\scriptstyle 3\ttt}$\\
\gr
14 & ${\scriptstyle e_{123}+e_{147}}$ & ${\scriptstyle \ssl(2,\R)+\ttt}$\\
15 & ${\scriptstyle e_{147}}$ & ${\scriptstyle \ssl(2,\R)+2\ttt}$\\
\gr
16 & ${\scriptstyle e_{123}+e_{456}}$ & ${\scriptstyle \ssl(3,\R)}$\\
17 & ${\scriptstyle e_{123}}$ & ${\scriptstyle \ssl(3,\R)+\ttt}$\\
\gr
18 & ${\scriptstyle 0}$ & ${\scriptstyle \ssl(3,\R)+2\ttt}$\\
\hline
\end{longtable}

\begin{longtable}{|r|l|l|}
              \caption{Nilpotent parts of mixed elements with semisimple part
  $p^{5,2}_{\lambda}$}\label{tab:fam5_2}
\endfirsthead
\hline
\endhead
\hline
\endfoot
\endlastfoot

\hline
No. &\qquad\quad Reps. of nilpotent parts & $\qquad\mathfrak{z}_0(p,h,e,f)$ \\
\hline
1 & ${\scriptstyle 2e_{138}+e_{139}+e_{147}+2e_{157}+2e_{237}+e_{345}-e_{389}-e_{468}+
  \tfrac{1}{2}e_{469}+\tfrac{1}{2}e_{479}+2e_{568}-e_{569}-2e_{578}}$ &
${\scriptstyle 0}$\\
\gr
2 & ${\scriptstyle
  -\tfrac{1}{2}e_{134}+e_{135}-e_{178}+\tfrac{1}{2}e_{179}-\tfrac{1}{4}e_{248}
+\tfrac{1}{8}e_{249}+\tfrac{1}{2}e_{258}-\tfrac{1}{4}e_{259}+e_{345}+e_{367}-e_{389}
+\tfrac{1}{2}e_{456}+\tfrac{1}{2}e_{479}-2e_{578}+\tfrac{1}{2}e_{689}}$ &
${\scriptstyle 0}$\\
3 & ${\scriptstyle -\tfrac{1}{2}e_{126}+e_{134}-2e_{135}+2e_{178}-e_{179}-e_{248}-
  \tfrac{1}{2}e_{249}-2e_{258}-e_{259}-2e_{367}}$ & ${\scriptstyle \uuu}$ \\
& ${\scriptstyle \tfrac{1}{2}e_{126}+e_{134}-2e_{135}+2e_{178}-e_{179}-e_{248}-
  \tfrac{1}{2}e_{249}-2e_{258}-e_{259}-2e_{367}}$ & ${\scriptstyle \uuu}$ \\
\gr
4 & ${\scriptstyle \tfrac{1}{2}e_{245}+\tfrac{1}{2}e_{289}+e_{345}-e_{389}+
  \tfrac{1}{4}e_{469}+\tfrac{1}{2}e_{479}+e_{568}-2e_{578}}$ &
${\scriptstyle 0}$\\
\gr
& ${\scriptstyle  \tfrac{1}{8}e_{248}+\tfrac{1}{16}e_{249}+\tfrac{1}{4}e_{258}+
  \tfrac{1}{8}e_{259}+e_{345}-e_{389}-e_{468}+\tfrac{1}{2}e_{469}+
  \tfrac{1}{2}e_{479}+2e_{568}-e_{569}-2e_{578}}$ & ${\scriptstyle 0}$\\
5 & ${\scriptstyle -\tfrac{1}{2}e_{126}+2e_{138}+e_{139}+e_{147}+\tfrac{1}{4}e_{148}-
  \tfrac{1}{8}e_{149}+2e_{157}-\tfrac{1}{2}e_{158}+\tfrac{1}{4}e_{159}+2e_{237}+
  \tfrac{1}{4}e_{346}-\tfrac{1}{2}e_{356}+\tfrac{1}{2}e_{678}-\tfrac{1}{4}e_{679}}$
& ${\scriptstyle \uuu}$ \\
  & ${\scriptstyle \tfrac{1}{2}e_{126}+2e_{138}+e_{139}+e_{147}+\tfrac{1}{4}e_{148}-
  \tfrac{1}{8}e_{149}+2e_{157}-\tfrac{1}{2}e_{158}+\tfrac{1}{4}e_{159}+2e_{237}+
  \tfrac{1}{4}e_{346}-\tfrac{1}{2}e_{356}+\tfrac{1}{2}e_{678}-\tfrac{1}{4}e_{679}}$
& ${\scriptstyle \uuu}$\\
\gr
6 & ${\scriptstyle 2e_{138}+e_{139}+e_{147}+2e_{157}+2e_{237}-e_{468}+
  \tfrac{1}{2}e_{469}+2e_{568}-e_{569}}$ & ${\scriptstyle \ttt+\uuu}$\\
7 & ${\scriptstyle e_{134}-2e_{135}+2e_{178}-e_{179}+e_{345}-2e_{367}-e_{389}+
  \tfrac{1}{2}e_{479}-2e_{578}}$ & ${\scriptstyle \ttt}$\\
\gr
8 & ${\scriptstyle  -\tfrac{1}{2}e_{126}-\tfrac{1}{4}e_{249}-e_{258}-
  \tfrac{1}{2}e_{456}-\tfrac{1}{2}e_{689}}$ & ${\scriptstyle \uuu}$ \\
\gr
& ${\scriptstyle -\tfrac{1}{2}e_{126}+\tfrac{1}{4}e_{137}+e_{468}-
  \tfrac{1}{2}e_{469}-2e_{568}+e_{569}}$ & ${\scriptstyle \uuu}$ \\
\gr
& ${\scriptstyle \tfrac{1}{2}e_{126}-\tfrac{1}{4}e_{249}-e_{258}-
  \tfrac{1}{2}e_{456}-\tfrac{1}{2}e_{689}}$ & ${\scriptstyle \uuu}$\\
\gr
& ${\scriptstyle \tfrac{1}{2}e_{126}+\tfrac{1}{4}e_{137}+e_{468}-
  \tfrac{1}{2}e_{469}-2e_{568}+e_{569}}$ & ${\scriptstyle \uuu}$\\
9 & ${\scriptstyle  2e_{138}+e_{139}+e_{147}+\tfrac{1}{4}e_{148}-\tfrac{1}{8}e_{149}
  +2e_{157}-\tfrac{1}{2}e_{158}+\tfrac{1}{4}e_{159}+2e_{237}+\tfrac{1}{4}e_{346}-
  \tfrac{1}{2}e_{356}+\tfrac{1}{2}e_{678}-\tfrac{1}{4}e_{679}}$ &
${\scriptstyle \ttt+\uuu}$\\
\gr
10 & ${\scriptstyle -\tfrac{1}{2}e_{126}+2e_{138}+e_{139}+e_{147}+2e_{157}+2e_{237}}$
& ${\scriptstyle \ttt+\uuu}$\\
11 & ${\scriptstyle \tfrac{1}{2}e_{245}+\tfrac{1}{2}e_{289}+\tfrac{1}{4}e_{469}+
  e_{568}}$ & ${\scriptstyle \ttt+\uuu}$\\
 & ${\scriptstyle e_{248}+\tfrac{1}{2}e_{249}+2e_{258}+e_{259}+\tfrac{1}{8}e_{468}-
  \tfrac{1}{16}e_{469}-\tfrac{1}{4}e_{568}+\tfrac{1}{8}e_{569}}$ &
${\scriptstyle \ttt+\uuu}$\\
\gr
12 & ${\scriptstyle e_{345}-e_{389}-e_{468}+\tfrac{1}{2}e_{469}+\tfrac{1}{2}e_{479}+
  2e_{568}-e_{569}-2e_{578}}$ & ${\scriptstyle \ssl(2,\R)}$ \\
13 & ${\scriptstyle e_{134}-2e_{135}+2e_{178}-e_{179}-2e_{367}}$ &
${\scriptstyle 2\ttt+\uuu}$ \\
\gr
14 & ${\scriptstyle -\tfrac{1}{2}e_{126}-2e_{137}}$ &
${\scriptstyle \ssl(2,\R)+\uuu}$ \\
\gr
& ${\scriptstyle \tfrac{1}{2}e_{126}-2e_{137}}$ &
${\scriptstyle \ssl(2,\R)+\uuu}$\\
15 & ${\scriptstyle -e_{468}+\tfrac{1}{2}e_{469}+2e_{568}-e_{569}}$ &
${\scriptstyle \ssl(2,\R)+\ttt+\uuu}$\\
\gr
16 & ${\scriptstyle e_{345}-e_{389}+\tfrac{1}{2}e_{479}-2e_{578}}$ &
${\scriptstyle \ssl(3,\R)}$\\
17 & ${\scriptstyle -e_{126}}$ & ${\scriptstyle \ssl(3,\R)+\uuu}$\\
& ${\scriptstyle e_{126}}$ & ${\scriptstyle \ssl(3,\R)+\uuu}$\\
\gr
18 & ${\scriptstyle 0}$ & ${\scriptstyle \ssl(3,\R)+\ttt+\uuu}$\\
\hline
\end{longtable}

\begin{longtable}{|r|l|l|}
                \caption{Nilpotent parts of mixed elements with semisimple part
  $p^{6,1}_{\lambda}$}\label{tab:fam6_1}
\endfirsthead
\hline
\endhead
\hline
\endfoot
\endlastfoot

\hline
No. &\qquad\quad Reps. of nilpotent parts & $\qquad\mathfrak{z}_0(p,h,e,f)$ \\
\hline
1 & ${\scriptstyle e_{159}+e_{168}+e_{249}+e_{258}+e_{267}+e_{347}}$ & 0\\
\gr
2 & ${\scriptstyle e_{159}+e_{168}+e_{249}+e_{257}+e_{258}+e_{347}}$ & 0 \\
3 & ${\scriptstyle e_{149}+e_{158}+e_{167}+e_{248}+e_{259}+e_{347}}$ & 0\\
& ${\scriptstyle -e_{148}-e_{159}-e_{249}+e_{258}-e_{267}-e_{357}}$ & 0\\
\gr
4 & ${\scriptstyle e_{149}+e_{158}+e_{248}+e_{257}+e_{367}}$ &
${\scriptstyle \ttt}$ \\
5 & ${\scriptstyle e_{149}+e_{167}+e_{168}+e_{257}+e_{348}}$ &
${\scriptstyle \ttt}$ \\
\gr
6 & ${\scriptstyle e_{149}+e_{158}+e_{248}+e_{267}+e_{357}}$ &
${\scriptstyle \ttt}$\\
7 & ${\scriptstyle e_{149}+e_{158}+e_{167}+e_{248}+e_{357}}$ &
${\scriptstyle \ttt}$\\
\gr
8 & ${\scriptstyle e_{149}+e_{167}+e_{258}+e_{347}}$ &
${\scriptstyle 2\ttt}$\\
9 & ${\scriptstyle e_{147}+e_{158}+e_{258}+e_{269}}$ &
${\scriptstyle 2\ttt}$ \\
& ${\scriptstyle 2e_{147}-e_{159}+e_{168}-e_{249}-2e_{257}}$ &
${\scriptstyle \ttt + \uuu}$ \\
\gr
10 & ${\scriptstyle e_{149}+e_{158}+e_{167}+e_{248}+e_{257}+e_{347}}$ & 0 \\
11 & ${\scriptstyle e_{149}+e_{167}+e_{248}+e_{357}}$ &
${\scriptstyle 2\ttt}$\\
\gr
12 & ${\scriptstyle e_{149}+e_{167}+e_{247}+e_{258}}$ & ${\scriptstyle 2\ttt}$\\
13 & ${\scriptstyle e_{149}+e_{158}+e_{167}+e_{248}+e_{257}}$ &
${\scriptstyle \ttt}$\\
\gr
14 & ${\scriptstyle e_{149}+e_{157}+e_{168}+e_{247}+e_{348}}$ &
${\scriptstyle \ssl(2,\R)}$ \\
15 & ${\scriptstyle e_{158}+e_{169}+e_{247}}$ &
${\scriptstyle \ssl(2,\R)+2\ttt}$\\
\gr
16 & ${\scriptstyle e_{149}+e_{158}+e_{167}+e_{247}}$ &
${\scriptstyle 2\ttt}$\\
17 & ${\scriptstyle e_{148}+e_{157}+e_{249}+e_{267}}$ &
${\scriptstyle \ssl(2,\R)+\ttt}$\\
\gr
18 & ${\scriptstyle e_{147}+e_{158}+e_{248}+e_{259}}$ &
${\scriptstyle \ssl(2,\R)+\ttt}$\\
19 & ${\scriptstyle e_{149}+e_{157}+e_{248}}$ & ${\scriptstyle 3\ttt}$\\
\gr
20 & ${\scriptstyle e_{147}+e_{258}}$ & ${\scriptstyle 4\ttt}$\\
\gr
& ${\scriptstyle 2e_{147}-2e_{158}-2e_{248}-2e_{257}}$ &
${\scriptstyle 2\ttt + 2\uuu}$\\
21 & ${\scriptstyle e_{148}+e_{157}+e_{247}}$ & ${\scriptstyle 3\ttt}$\\

\gr
22 & ${\scriptstyle  e_{147}+e_{158}+e_{169}}$ &
${\scriptstyle \ssl(2,\R)+\ssl(3,\R)}$\\
23 & ${\scriptstyle e_{147}+e_{158}}$ &
${\scriptstyle 2\ssl(2,\R)+2\ttt}$\\
\gr
24 & ${\scriptstyle e_{147}}$ & ${\scriptstyle 3\ssl(2,\R)+2\ttt}$\\
25 & ${\scriptstyle 0}$ & ${\scriptstyle 3\ssl(3,\R)}$\\
\hline
\end{longtable}

\bigskip

\begin{longtable}{|r|l|l|}
\caption{Nilpotent parts of mixed elements with semisimple part
  $p^{6,2}_{\lambda}$}\label{tab:fam6_2}
\endfirsthead
\hline
\endhead
\hline
\endfoot
\endlastfoot

\hline
No. &\qquad\quad Reps. of nilpotent parts & $\qquad\mathfrak{z}_0(p,h,e,f)$ \\
\hline
1 & ${\scriptstyle e_{139}-e_{148}+2e_{157}+e_{245}+e_{289}+2e_{367}}$ & 0\\
\gr
2 & ${\scriptstyle -2e_{135}-e_{179}+\tfrac{1}{2}e_{234}+e_{278}+e_{368}+
  \tfrac{1}{2}e_{467}-e_{569}}$ & 0\\
3 & ${\scriptstyle -2e_{135}-e_{145}-e_{179}-e_{189}+\tfrac{1}{2}e_{234}+e_{278}+
  e_{368}+\tfrac{1}{2}e_{467}}$ & 0\\
& ${\scriptstyle e_{145}+e_{189}-2e_{237}-e_{248}+e_{346}+2e_{678}}$ & 0 \\
\gr
4 & ${\scriptstyle e_{159}+2e_{237}-e_{248}-\tfrac{1}{4}e_{249}-e_{258}-
  \tfrac{1}{2}e_{456}-\tfrac{1}{2}e_{689}}$ &  ${\scriptstyle \uuu}$ \\
\gr
& ${\scriptstyle -4e_{159}+\tfrac{1}{32}e_{237}-\tfrac{1}{2}e_{238}-
  \tfrac{1}{8}e_{239}-\tfrac{1}{4}e_{247}+20e_{248}+5e_{249}-\tfrac{1}{4}e_{257}+
  20e_{258}+4e_{259}-\tfrac{1}{4}e_{356}+2e_{456}-\tfrac{1}{8}e_{679}+2e_{689}}$ &
${\scriptstyle \uuu}$\\
5 &  ${\scriptstyle e_{159}-\tfrac{1}{2}e_{239}-e_{248}-e_{257}+e_{356}+
  \tfrac{1}{2}e_{679}}$ & ${\scriptstyle \uuu}$\\
& ${\scriptstyle -e_{159}-\tfrac{1}{2}e_{239}-e_{248}-e_{257}+e_{356}+
  \tfrac{1}{2}e_{679}}$ & ${\scriptstyle \uuu}$ \\
\gr
6 & ${\scriptstyle -e_{134}-2e_{178}+2e_{235}+e_{279}-e_{569}}$ &
${\scriptstyle \ttt}$\\
7 & ${\scriptstyle e_{139}-e_{148}+2e_{157}+e_{238}+\tfrac{1}{2}e_{247}+
  \tfrac{1}{2}e_{346}+e_{678}}$ & ${\scriptstyle \uuu}$ \\
& ${\scriptstyle -4e_{135}-\tfrac{1}{4}e_{148}-2e_{179}-\tfrac{1}{2}e_{238}-
  \tfrac{1}{4}e_{247}+\tfrac{1}{4}e_{346}+\tfrac{1}{2}e_{678}}$ &
${\scriptstyle \uuu}$ \\
\gr
8 & ${\scriptstyle  -e_{145}-e_{189}+2e_{237}+e_{468}}$ &
${\scriptstyle \ttt+\uuu}$\\
9 & ${\scriptstyle e_{235}+\tfrac{1}{2}e_{279}-\tfrac{1}{2}e_{369}+e_{468}-e_{567}}$
& ${\scriptstyle \ttt+\uuu}$ \\
& ${\scriptstyle \tfrac{1}{2}e_{237}+e_{248}+e_{259}+\tfrac{1}{2}e_{469}+2e_{568}}$
& ${\scriptstyle 2\uuu}$ \\
& ${\scriptstyle \tfrac{1}{2}e_{237}-e_{248}-e_{259}+\tfrac{1}{2}e_{469}+2e_{568}}$
& ${\scriptstyle 2\uuu}$  \\
\gr
10 & ${\scriptstyle  e_{159}+2e_{237}-\tfrac{1}{2}e_{239}-\tfrac{1}{2}e_{249}-
  e_{257}-2e_{258}-e_{356}-\tfrac{1}{2}e_{679}}$ & 0 \\
11 & ${\scriptstyle -2e_{135}-e_{179}+\tfrac{1}{2}e_{245}+\tfrac{1}{2}e_{289}-
  \tfrac{1}{4}e_{469}-e_{568}}$ & ${\scriptstyle \ttt+\uuu}$\\
\gr
12 & ${\scriptstyle  -2e_{135}-e_{179}+e_{468}-e_{569}}$ &
${\scriptstyle \ttt+\uuu}$\\
\gr
& ${\scriptstyle -2e_{139}-4e_{157}+\tfrac{1}{4}e_{468}+4e_{569}}$ &
${\scriptstyle \ttt+\uuu}$\\
13 & ${\scriptstyle -2e_{138}-e_{147}+e_{159}-e_{239}-2e_{257}}$ &
${\scriptstyle \ttt}$\\
\gr
14 & ${\scriptstyle  e_{159}-\tfrac{1}{2}e_{239}-\tfrac{1}{4}e_{249}-e_{257}-e_{258}
  +e_{356}-\tfrac{1}{2}e_{456}+\tfrac{1}{2}e_{679}-\tfrac{1}{2}e_{689}}$
& ${\scriptstyle \ssl(2,\R)}$ \\
15 & ${\scriptstyle  -2e_{135}-e_{179}+e_{468}}$ &
${\scriptstyle \ssl(2,\R)+\ttt+\uuu}$\\
& ${\scriptstyle \tfrac{1}{4}e_{137}+8e_{159}-e_{468}}$ &
${\scriptstyle \su(2)+\ttt+\uuu}$\\
\gr
16 & ${\scriptstyle  -2e_{137}+\tfrac{1}{2}e_{149}+2e_{158}-e_{259}}$ &
${\scriptstyle \ttt+\uuu}$\\
17 & ${\scriptstyle e_{235}+\tfrac{1}{2}e_{245}+\tfrac{1}{2}e_{279}+
  \tfrac{1}{2}e_{289}-\tfrac{1}{2}e_{369}+\tfrac{1}{4}e_{469}-e_{567}+e_{568}}$
& ${\scriptstyle \ssl(2,\R)+\ttt}$\\
\gr
18 & ${\scriptstyle  e_{159}+2e_{237}-\tfrac{1}{2}e_{239}-e_{257}+e_{356}+
  \tfrac{1}{2}e_{679}}$ & ${\scriptstyle \ssl(2,\R)+\uuu}$\\
\gr
& ${\scriptstyle \tfrac{44}{5}e_{137}-3e_{139}-6e_{157}+2e_{159}-
  \tfrac{28}{5}e_{237}+e_{239}+2e_{257}+10e_{356}+5e_{679}}$ &
${\scriptstyle \su(2)+\uuu}$\\
\gr
& ${\scriptstyle \tfrac{4}{5}e_{137}+e_{139}+2e_{157}+2e_{159}+\tfrac{12}{5}e_{237}
  +e_{239}+2e_{257}-10e_{356}-5e_{679}}$ & ${\scriptstyle \su(2)+\uuu}$\\
19 & ${\scriptstyle e_{159}-\tfrac{1}{2}e_{239}-e_{257}-e_{259}+e_{356}+
  \tfrac{1}{2}e_{679}}$ & ${\scriptstyle \ttt+2\uuu}$\\
& ${\scriptstyle -4e_{159}+e_{239}+2e_{257}+4e_{259}-2e_{356}-e_{679}}$ &
${\scriptstyle \ttt+2\uuu}$\\
\gr
20 & ${\scriptstyle e_{235}+\tfrac{1}{2}e_{279}-\tfrac{1}{2}e_{369}-e_{567}}$ &
${\scriptstyle 2\ttt+2\uuu}$\\
\gr
& ${\scriptstyle \tfrac{1}{4}e_{148}+4e_{159}-\tfrac{1}{2}e_{469}-2e_{568}}$ &
${\scriptstyle 2\ttt+2\uuu}$\\
21 & ${\scriptstyle -2e_{135}-e_{179}-e_{569}}$ & ${\scriptstyle 2\ttt+\uuu}$\\
\gr
22 & ${\scriptstyle e_{139}-e_{148}+2e_{157}}$ &
${\scriptstyle \ssl(2,\R)+\su(1,2)}$\\
\gr
& ${\scriptstyle 18e_{137}+22e_{138}+4e_{139}+11e_{147}+14e_{148}+\tfrac{5}{2}
  e_{149}+8e_{157}+10e_{158}+2e_{159}}$ & ${\scriptstyle \ssl(2,\R)+\su(3)}$\\
23 & ${\scriptstyle -2e_{135}-e_{179}}$ & ${\scriptstyle 2\ssl(2,\R)+\ttt+\uuu}$\\
& ${\scriptstyle e_{137}+2e_{159}}$ & ${\scriptstyle \ssl(2,\R)+\su(2)+\ttt+\uuu}$\\
\gr
24 & ${\scriptstyle -2e_{137}}$ & ${\scriptstyle \ssl(2,\R)+\ssl(2,\C)+\ttt+\uuu}$\\
25 & ${\scriptstyle 0}$ & ${\scriptstyle \ssl(3,\R)+\ssl(3,\C)}$\\
\hline
\end{longtable}

%%%%%%%%%%%%%%%%%%%%%%%%%%%%%%%%%%%%%%galcohom.tex%%%%%%%%%%%%%%%%%%%%%%%%%%

\section{Real Galois cohomology}\label{sec:galcohom}

\subsection{Group cohomology for a group of order 2}\label{s:group-coh}

In this section we give an alternative  definition of
$\Ho^1(\Gamma,A)$ and $\Ho^2(\Gamma,A)$ and establish   their  properties.

\begin{definition}\label{d:H1-nonab}
Let $A$ be a $\Gamma$-group, {\em not necessarily abelian},
where $\Gamma=\{1,\gamma\}$ is a group of order 2.
We define
\begin{equation}\label{d:Z1}
\Zl^1\hm A=\{c\in A\mid c\cdot{}^\gamma\kern-0.8pt c=1\}.
\end{equation}
We say that an element $c$ as in \eqref{d:Z1} is a {\em $1$-cocycle of\/ $\Gamma$ in $A$}.
The group $A$ acts on $\Zl^1\hm A$ on the right by
\[ c*a=a^{-1}\cdot c\cdot\upgam a\quad \text{for } c\in \Zl^1\hm A,\ a\in A.\]
We say that the cocycles $c$ and $c*a$ are {\em cohomologous} or {\em equivalent}.
We denote by  $\Ho^1(\Gamma,A)$, or for brevity $\Ho^1\hm A$,
the set of equivalence classes, that is,
the set of orbits of $A$ in $\Zl^1\hm A$.
If $c\in \Zl^1\hm A$, we denote by $[c]\in \Ho^1\hm A$ its cohomology class.
\end{definition}

In general $\Ho^1\hm A$ has no natural group structure,
but it has a {\em neutral element} denoted by  $[1]$,
 the class of the unit element $1\in \Zl^1\hm A\subseteq A$.

\begin{definition}\label{d:H1}
Let $A$ be an {\em abelian} $\Gamma$-group, where $\Gamma$ is a group of order 2.
We define abelian groups
\[\Zl^1\hm A=\{c\in A\mid c\cdot\hm\upgam c=1\},\quad
   \Bd^1\hm A=\{a^{-1}\cdot\hm\upgam a\mid a\in A\}, \quad
   \Ho^1\hm A=\Zl^1\hm A\hs/\hs\Bd^1\hm A.\]
\end{definition}

\begin{definition}\label{d:H2}
Let $A$ be an {\em abelian} $\Gamma$-group, where $\Gamma$ is a group of order 2.
We define abelian groups
\[\Zl^2\hm A=A^\Gamma\ce \{c\in A\mid \upgam c=c\},\quad
    \Bd^{2}\hm A=\{a\cdot\hm\upgam a\mid a\in A\},\quad
    \Ho^2\hm A=\Zl^2\hm A\hs/\hs\Bd^2\hm A.\]
\end{definition}

\begin{remark}\label{r:def}
Definitions \ref{d:H1} and \ref{d:H2}
(given for  a group $\Gamma$ of order 2 only!)
are equivalent to the standard ones.
Namely, for $c\in \Zl^1\hm A$ we construct a function of one variable
\[f_c\colon\Gamma\to A,\quad f_c(1)=1,\ f_c(\gamma)=c,\]
which is a 1-cocycle in the sense of  \cite[Section I.5.1]{Serre1997}.
Similarly, for $c\in \Zl^2\hm A$ we construct a function of two variables
\[ \phi_c\colon \Gamma\times\Gamma\to A,\quad
     \phi_c(1,1)=\phi_c(\gamma,1)=\phi_c(1,\gamma)=0,\ \,\phi_c(\gamma,\gamma)=c,\]
which is a 2-cocycle in the sense of Serre \cite[Section I.2.2]{Serre1997}.
In this way we obtain canonical isomorphisms of pointed sets
(for Definition \ref{d:H1-nonab}), and of abelian groups
(for Definitions \ref{d:H1} and \ref{d:H2})
between the cohomology sets and groups defined above
and the corresponding cohomology sets and groups defined in \cite{Serre1997}.
\end{remark}

We shall need the following result of Borel and Serre:

\begin{proposition}[{\cite[Section I.5.4, Corollary 1 of Proposition 36]{Serre1997}}]
\label{p:serre}
Let $B$ be a $\Gamma$-group, $A\subseteq B$ be a $\Gamma$-subgroup, and $Y=B/A$,
which has a natural structure of a $\Gamma$-set.
Then $\Bd^\Gamma$ naturally acts on $Y^\Gamma$,
and the connecting map $\delta$ in the cohomology exact sequence
\[1\to A^\Gamma\to \Bd^\Gamma\to Y^\Gamma\labelto{\delta} \Ho^1\hm A\to \Ho^1 B\]
induces a canonical bijection
between the set of orbits $Y^\Gamma/\Bd^\Gamma$ and the kernel
\[\ker\big[\hs\Ho^1\hm A\to \Ho^1 B\hs\big].\]
\end{proposition}

\subsection{Real structures on complex algebraic groups and algebraic varieties}

\begin{subsec}\label{sss:real-group}
Let $\GG$ be a real linear algebraic group.
In the coordinate language, the reader may regard $\GG$ as
a subgroup in the general linear group $\GL_n(\C)$ (for some integer~$n$)
defined by polynomial equations with {\em real} coefficients in the matrix entries;
see Borel \cite[Section 1.1]{Borel1966}.
More conceptually, the reader may assume that $\GG$ is an affine group scheme
of finite type over $\R$; see  Milne \cite[Definition 1.1]{Milne2017}.
With any of these two equivalent  definitions, $\GG$ defines a covariant functor
\begin{equation*}
A\mapsto \GG(A)
\end{equation*}
from the category of commutative unital $\R$-algebras to the category of groups.
Applying this functor to the $\R$-algebra $\R$, we obtain a real Lie group $\GG(\R)$.
Applying this functor  to the $\R$-algebra $\C$ and to the morphism of $\R$-algebras
\[\gamma\colon \C\to\C, \quad z\mapsto \bar z\quad\text{for }z\in \C,\]
we obtain a complex Lie group $\GG(\C)$ together with an anti-holomorphic involution
$\GG(\C)\to \GG(\C),$
which will be denoted by $\sigma_\GG$.
The Galois group $\Gamma$ naturally acts on $\GG(\C)$;
namely, the complex conjugation $\gamma$ acts by $\sigma_\GG$.
We have $\GG(\R)=\GG(\C)^\Gamma$ (the subgroup of fixed points).

We shall consider the complex linear algebraic group $\GG_\C\ce \GG\times_\R\C$
obtained from $\GG$ by extension of scalars from $\R$ to $\C$.
In this article we shall denote $\GG_\C$ by $G$,
the same Latin  letter, but non-boldface.
By abuse of notation we shall identify $G$ with $\GG(\C)$;
in particular, we shall write $g\in G$ meaning that $g\in \GG(\C)$.

Since $G$ is an affine group scheme over $\C$,
we have the ring of regular function
$\C[G]=\R[\GG]\otimes_\R\C.$
Our anti-holomorphic involution $\sigma_\GG$ of $\GG(\C)$
is {\em anti-regular} in the following sense:
when acting on the ring of holomorphic functions on $G$
by sending a holomorphic function $f$ to the holomorphic function
\[(\upgam f)(g)=\upgam(f(\upgam g))\quad\text{for}\ g\in G,\]
it preserves the subring $\C[G]$ of regular functions.
An anti-regular involution of $G$ is called
also  a {\em real structure on} $G$.
\end{subsec}

\begin{remark}
If $G$ is a {\em reductive} algebraic group
over $\C$ (not necessarily connected),
then any anti-holomorphic involution of $G$ is anti-regular.
See  Adams and Ta\"\i bi \cite[Lemma 3.1]{AT2018}.
\end{remark}

\begin{subsec}
A morphism of real linear algebraic groups
$\GG\to\GG'$ induces a morphism of pairs
$(G,\sigma_\GG)\to (G',\sigma_{\GG'})$.
In this way we obtain a functor
$$\GG\rightsquigarrow(G,\sigma_\GG)$$
from the category of {\em real} linear algebraic groups
to the category of pairs $(G,\sigma)$,
where $G$ is a {\em complex} linear algebraic group
and $\sigma$ is a real structure on $G$.
By Galois descent, see
\cite[Section V.4.20,  Corollary 2 of Proposition 12]{Serre1988}
or  \cite[Theorem 2.2]{Jahnel}, or \cite[Section 6.2, Example B]{BLR},
this functor is an equivalence of categories.
In particular, any pair $(G,\sigma)$, where $G$ is
a complex linear algebraic group and $\sigma$ is a real structure on $G$,
is isomorphic to a pair coming from a real linear algebraic group $\GG$,
and any morphism of pairs $(G,\sigma)\to (G',\sigma')$ comes
from  a unique  morphism of the corresponding real algebraic groups.
\end{subsec}

\begin{subsec}
Let $\YY$ be a real {\em quasi-projective} variety.
The reader may regard $\YY$ as a locally closed subvariety
in the complex projective space $\P^n_\C$ (for some integer $n$)
defined using homogeneous polynomials with {\em real} coefficients.
As above, by functoriality we obtain a complex analytic space $\YY(\C)$
together with a {\em real structure} (anti-regular involution)
\[\mu_\YY\colon \YY(\C)\to\YY(\C).\]
We consider the base change $Y\ce \YY\times_\R\C$,
which by abuse of notation we shall identify with $\YY(\C)$.
We obtain a functor
$$\YY\rightsquigarrow(Y,\mu_\YY)$$
from the category of {\em real} quasi-projective varieties
to the category of pairs $(Y,\mu)$,
where $Y$ is a {\em complex} quasi-projective variety and $\mu$ is a real structure on $Y$.
By Galois descent
this functor is an equivalence of categories.
Here it is important that we consider {\em quasi-projective} varieties.
Note that any homogeneous space of a linear algebraic group is quasi-projective;
see, for instance, \cite[Theorem 6.8]{Borel1991}.
\end{subsec}

\begin{subsec}
From now on, when mentioning a {\em real} algebraic group $\GG$,
we shall actually work with a pair $(G,\sigma)$,
where $G$ is a {\em complex} algebraic group and $\sigma$ is a real structure on $G$.
We shall write $\GG=(G,\sigma)$.
We shall shorten ``real linear algebraic group'' to ``$\R$-group''.
Similarly,  when mentioning a {\em real} algebraic variety $\YY$,
we shall actually work with a pair $(Y,\mu)$,
where $Y$ is a {\em complex} algebraic variety and $\mu$ is a real structure on $Y$.
We shall write $\YY=(Y,\mu)$.
\end{subsec}

\subsection{Using $\Ho^1$ for finding real orbits in homogeneous spaces}

\begin{subsec}\label{sss:G-Y}
Let $\GG$ be a real algebraic group acting (over $\R$)
on a real algebraic variety $\YY$.
We write $\GG=(G,\sigma)$, $\YY=(Y,\mu)$.
Here $\mu$ is an anti-regular involution of $Y$ given by $\mu(y)=\upgam y$.
The assertion that $\GG$ acts on $\YY$ over $\R$ means
that $G$ acts on $Y$ by $(g,y)\mapsto g\cdot y$,
and this action is $\Gamma$-equivariant:
for $g\in G$, $y\in Y$ we have
\[ \upgam(g\cdot y)=\upgam g\cdot\hm\upgam y,\quad\text{that is,}
     \quad \mu(g\cdot y)=\sigma(g)\cdot \mu(y).\]
We assume that $\GG$ acts on $\YY$ {\em transitively},
that is, $G$ acts on $Y$ transitively.
\end{subsec}

\begin{subsec}
We have $\GG(\R)=G^\sigma$, $\YY(\R)=Y^\mu$.
The group $\GG(\R)$ naturally acts on $\YY(\R)$, and this action might not be transitive.
Assuming that $\YY(\R)$ is {\em non-empty}, we describe the set of orbits
$\YY(\R)/\GG(\R)$ in terms of Galois cohomology.

Let $e\in \YY(\R)=Y^\mu$ be an $\R$-point.
Let $C=\Stab_G(e)$.
Since $\sigma(e)=e$, we have $\sigma(C)=C$.
We consider the real algebraic subgroup
\[ \CC=\Stab_\GG(e)=(C,\sigma_C),\quad\text{where }\sigma_C=\sigma|_C\hs,\]
and the homogeneous space $\GG/\CC$, on which $\Gamma$ acts by $\upgam(gC)=\upgam g\cdot C$.
We have a canonical bijection
\[G/C\isoto Y, \quad gC\mapsto g\cdot e,\]
and an easy calculation shows that this bijection is $\Gamma$-equivariant.
Taking into account Proposition \ref{p:serre}, we obtain bijections
\begin{equation}\label{e:ker-G/C-Y/C}
\ker\big[\hs\Ho^1\hs\CC\to\Ho^1\hs\GG\hs\big]\isoto (G/C)^\Gamma/\GG(\R)\isoto\YY(\R)/\GG(\R).
\end{equation}
\end{subsec}

\begin{construction}\label{con:e-c}
We describe explicitly the composite bijection \eqref{e:ker-G/C-Y/C}.
Let $c\in\Zl^1\CC$ be such that $[c]\in \ker\big[\Ho^1\hs\CC\to\Ho^1\hs\GG\big]$.
Then there exists $g\in G$ such that $c=g^{-1}\cdot\upgam g$, that is, $\sigma(g)=gc$.
We set $e_c=g\cdot e\in Y$.
We compute
\[\mu(e_c)=\mu(g\cdot e)=\sigma(g)\cdot\mu(e)=gc\cdot e=g\cdot(c\cdot e)=g\cdot e=e_c\hs.\]
Thus $e_c\in\YY(\R)$.
To $[c]$ we associate the $\GG(\R)$-orbit $\GG(\R)\cdot e_c\subset\YY(\R)$.
\end{construction}

\begin{proposition}\label{p:coh-orbits}
The correspondence $c\mapsto e_c$ of Construction \ref{con:e-c} induces a bijection
 \[\ker\big[\hs\Ho^1\hs\CC\to\Ho^1\hs\GG\hs\big]\isoto\YY(\R)/\GG(\R).\]
\end{proposition}

\begin{proof}
The proposition follows from Proposition \ref{p:serre}.
\end{proof}

\begin{subsec}
Let $\GG_0$ be a real algebraic group
embedded into a larger real algebraic group $\GG$,
and let  $h\in\g\ce \Lie\GG$.
Then $\GG_0$ acts on $\g$ via the adjoint representation of $\GG$.
Let $\YY_h$ be the $\GG_0$-orbit of $h$, that is, the real algebraic variety defined by
\[Y_h=\{h'\in \g^\cC \mid h'=\Ad(g)\cdot h \ \text{ for some } g\in G_0\},\]
where $\g^\cC=\g\otimes_\R\C$.
Let $\ZZ_h$ denote the centralizer of $h$ in $\GG_0$.
The group $\GG_0$ acts transitively (over $\C$) on the left on $\YY_h$ by
\[(g,h')\mapsto \Ad(g)\cdot h',\quad g\in G,\ h'\in Y_h\]
with stabilizer $\ZZ_h$ of $h$.
\end{subsec}

\begin{corollary}[from Proposition \ref{p:coh-orbits}]
The set of the $\GG_0(\R)$-conjugacy classes in $\YY(\R)$
is in a canonical bijection with
\[\ker\big[\Ho^1\hs \ZZ_h\to\Ho^1\hs\GG_0\big].\]
\end{corollary}

\begin{corollary}\label{c:Rconj}
If $\Ho^1\hs \ZZ_h=1$, then any element $h'\in\g$ that is $\GG_0$-conjugate to $h$
over $\C$, is $\GG_0$-conjugate to $h$ over $\R$.
\end{corollary}

\subsection{Using $\Ho^2$ for finding a real point in a complex orbit}

\begin{subsec}
\label{sec:findreal}
In order to use Construction \ref{con:e-c}, we need a real point $e\in \YY(\R)$.
We describe a general method of finding a real point using second Galois cohomology,
following an idea of Springer \cite[Section 1.20]{Springer1966};
see also \cite[Section 7.7]{Borovoi1993} and \cite[Section 2]{DLA2019}.

Let $\GG=(G,\sigma)$ and $\YY=(Y,\mu)$ be as in Subsection \ref{sss:G-Y}
We choose a {\em  complex} point $e\in Y$ and set $C=\Stab_G(e)\subset G$.
In this article we consider only the special case when
the complex algebraic group $C$ is {\em abelian};
this suffices for our applications.

Consider  $\mu(e)\in Y$. We have
\[e=g\cdot\mu(e)\quad\text{for some } g\in G,\]
because $Y$ is homogeneous.
Since $\mu^2=\id$, we have
\begin{align*}
e=\mu(\mu(e))=\mu(g^{-1}\cdot e)=&\sigma(g^{-1})\cdot\mu(e)
=\sigma(g^{-1})\cdot g^{-1}\cdot e=(g\cdot\sigma(g))^{-1}\cdot e.
\end{align*}
Thus $g\cdot\sigma(g)\in C$.
We set
\[d=g\cdot\sigma(g)\in C.\]

We define an anti-regular involution $\nu$ of $C$.
Let $c\in C$.  We calculate:
\begin{align*}
c\cdot e=e\quad &\Rightarrow \quad \sigma(c)\cdot\mu(e)=\mu(e)\quad \Rightarrow \quad
   g\cdot\sigma(c)\cdot\mu(e)=g\cdot\mu(e) \quad \Rightarrow \\
&\Rightarrow\quad  g\hs\sigma(c)g^{-1}\cdot g\cdot\mu(e)=g\cdot\mu(e) \quad \Rightarrow \quad
   g\hs\sigma(c)g^{-1}\cdot e=e.
\end{align*}
We see that $ g\hs\sigma(c)g^{-1}\in C$. We set
\begin{equation}\label{e:nu}
\nu(c)=g\hs\sigma(c)g^{-1}\in C.
\end{equation}
We compute
\begin{align*}
\nu^2(c)=\nu(\nu(c))=g\cdot \sigma(\nu(c))\cdot g^{-1}
   &=g\cdot\sigma(g\hs\hs\sigma(c)\hs g^{-1})\cdot g^{-1}\\
&=(g\hs\hs\sigma(g))\cdot c\cdot (g\hs\hs\sigma(g))^{-1}=c,
\end{align*}
because $g\hs\hs \sigma(g)\in C$ and $C$ is abelian.
Thus $\nu^2=1$, that is, $\nu$ is an involution,
and it is clear from \eqref{e:nu} that it is anti-regular.
We have
\[\nu(d)=g\hs\sigma(d) g^{-1}=g\hs\sigma(g)\hs g\hs g^{-1}=g\hs\sigma(g)=d.\]
This $d\in C^\nu$.
We have constructed an anti-regular involution
\[\nu\colon C\to C,\quad c\mapsto g\hs\hs\sigma(c)g^{-1}\]
and an element
\[d=g\hs\hs\sigma(g)\in C^\nu.\]
\end{subsec}

\begin{subsec}
We consider  the real algebraic group $\CC=(C,\nu)$.
Recall that
\[\Ho^2\hs\CC=\Zl^2\hs\CC/\Bd^2\hs\CC,\]
where
\begin{align*}
\Zl^2\hs\CC=C^\nu\ce \{c\in C\mid\nu(c)=c\},\qquad
\Bd^2\hs\CC=\{c'\cdot\nu(c')\mid c'\in C\}.
\end{align*}

We define the {\em class of $Y$}
\[\Cl(Y)=[d]\in \Ho^2\hs\CC.\]
\end{subsec}

\begin{proposition}\label{p:no-mu-fixed}
If $Y$ has a $\mu$-fixed point, then $\Cl(Y)=1$.
\end{proposition}

\begin{proof}
Assume that $Y$ has a $\mu$-fixed point $y$, that is, $y=\mu(y)$.
Write $y=u^{-1}\cdot e$ with $u\in G$. Then $\mu(u^{-1}\cdot e)=u^{-1}\cdot e$,
whence
\[u\cdot \sigma(u)^{-1}\cdot g^{-1}\cdot e=e.\]
Set $c=u\hs\hs \sigma(u)^{-1} g^{-1}$; then $c\cdot e=e$, whence $c\in C$.
We calculate:
\[ c\cdot\nu(c)\cdot d=
u\hs\hs \sigma(u)^{-1} g^{-1}\cdot g\cdot\sigma(u)\hs\hs u^{-1} \sigma(g)^{-1}\cdot g^{-1}\cdot g\hs\hs\sigma(g)=1.\]
Thus $\Cl(Y)=[d]=1\in \Ho^2\hs\CC$, as required.
\end{proof}

\begin{proposition}\label{p:mu-fixed}
If  $\Cl(Y)=1$ and $\Ho^1\hs\GG=1$, then $Y$ has a $\mu$-fixed point.
\end{proposition}

\begin{proof}
Choose $e\in Y$ and choose $g\in G$ such that $e=g\cdot\mu(e)$.
Write
\[d=g\cdot\sigma(g)\in \Zl^2\CC=C^\nu.\]
By assumption $[d]=1$, that is,
\[c\cdot\nu(c)\cdot d =1\quad\text{for some } c\in C.\]
Set $g'=cg$.
We have
\[ 1=c\cdot\nu(c)\cdot d=c\cdot g\hs\sigma(c)\hs g^{-1}\cdot g\hs\sigma(g)
=cg\cdot\sigma(cg)=g'\cdot\sigma(g').\]
Thus $g'$ is a 1-cocycle, $g'\in \Zl^1\GG$.

By assumption $\Ho^1\GG=1$.
It follows that there exists
$u\in G$ such that
\[u^{-1} g'\hs\sigma(u)=1.\]
Set $y=u^{-1}\cdot e\in Y$.
Then
\begin{align*}
\mu(y)&=\sigma(u^{-1})\cdot\mu(e)=\sigma(u)^{-1}g^{-1}\cdot e=\sigma(u)^{-1}g^{-1}\cdot c^{-1}\cdot e\\
          &=\sigma(u)^{-1}(g')^{-1}\cdot e=\sigma(u)^{-1 }(g')^{-1}u\cdot y=\big(u^{-1}g'\hs\sigma(u)\big)^{-1}\cdot y=y.
\end{align*}
Thus $y$ is a $\mu$-fixed point in $Y$, as required.
\end{proof}

\begin{remark}
Propositions \ref{p:mu-fixed} and \ref{p:no-mu-fixed} and their proofs
give a method of finding a real (that is, $\mu$-fixed) point in $Y$
or proving that $Y$ has no real points, assuming that $\Ho^1\GG=1$ and that $C$ is abelian.
The general case (without these two assumptions) is treated in \cite{Borovoi2021}.
\end{remark}

%%%%%%%%%%%%%%%%%%%%%%%%%%%%%%%gradedla.tex%%%%%%%%%%%%%%%%%%%%%%%%%%%%%%%%

\section{A $\Z_3$-grading of the simple complex Lie algebra of type $\EEE_8$}\label{sec:gradedLie}

\subsection{Constructing $\EEE_8$ with a $\Z_3$-grading}
We write $\Z_3 = \Z/3\Z$.  We have
\[\Z_3=\{\, \bar 0 =0+3\Z,\ \bar 1=1+3\Z,\ \ov{-1}=-1+3\Z\,\}.\]
Generalizing and simplifying a construction of Vinberg and Elashvili \cite{VE1978},
here we construct a $\Z_3$-grading of a split simple Lie algebra $\g$
of type $\EEE_8$ over a field $\K$ of characteristic 0.
This grading plays a pivotal role in our
classification of trivectors of a real 9-dimensional space.
We follow \cite[Example 3.3.v]{Le2011}.

Consider the free abelian group $\Z^9$ with the standard basis $\vet_1,\dots,\vet_9$.
Set
\[Q=\Z^9/\langle\vet_1+\dots+\vet_9\rangle.\]
For $i=1,\dots,9$, let $\ve_i\in Q$ denote the image of $\vet_i$ in $Q$. Then
\[ \ve_1+\dots+\ve_9=0.\]

We construct a subset $\Phi_\vk\subset Q$ for each $\vk\in\Z_3$.
We write $\Phi_0$ for $\Phi_{\bar0}$ etc.
Set
\begin{align*}
&\Phi_0=\{\ve_i-\ve_j\mid i\neq j\},\\
&\Phi_1=\{\ve_i+\ve_j+\ve_k\mid i,j,k\  \text{pairwise different}\},\\
&\Phi_{-1}=-\Phi_1=\{-(\ve_i+\ve_j+\ve_k)\mid i,j,k\  \text{pairwise different}\}\}.
\end{align*}
We set
\begin{gather*}
\Phi=\Phi_0\cup\Phi_1\cup\Phi_{-1}\hs,\\
\alpha_1=\ve_1-\ve_2,\ \alpha_2=\ve_2-\ve_3,\ \dots,\
      \alpha_7=\ve_7-\ve_8,\ \alpha_8=\ve_6+\ve_7+\ve_8,\\
\Pi=\{\alpha_1,\dots,\alpha_7,\alpha_8\}\subset \Phi,\\
 V=Q\otimes_\Z\R.
\end{gather*}
Then $\Phi$ is a root system of type $\EEE_8$ in $V$, and $\Pi$ is a basis of $\Phi$;
see  Onishchik and Vinberg \cite[Table 1]{OV1990}.
Note that $\Phi=\Phi_0\cup\Phi_1\cup\Phi_{-1}$ is a
{\em $\Z_3$-grading of $\Phi$} in the following sense:
\begin{equation}\label{e:Phi-grading}
(\Phi_\iota+\Phi_\vk)\cap\Phi\,\subset\, \Phi_{\iota+\vk}\quad
     \text{for}\ \iota,\vk\in\Z_3,\ \iota\neq \vk.
\end{equation}

We denote by $V^*$ the dual vector space to $V$.
Consider the dual root system
\[\Phi^\vee=\{\eta^\vee\in V^*\mid\eta\in \Phi\},\]
where $\eta^\vee$ is the coroot corresponding to the root $\eta$.
Set $Q^\vee=\langle\Phi^\vee\rangle\subset V^*$, the lattice generated by $\Phi^\vee$.
Then $\{\alpha^\vee\mid\alpha\in \Pi\}$ is a basis of $Q^\vee$.

Let $\K$ be a field of characteristic 0. Consider the pair $(\g,\h)$
defined by $(\Phi,\Pi)$; see Bourbaki  \cite[Section VIII.4.3, Theorem 1]{Bourbaki79}.
Here $(\g,\h)$ is a split simple Lie algebra of type $\EEE_8$ over $\K$
(namely, $\g$ is a simple Lie algebra and  $\h\subset\g$ is a split Cartan subalgebra).
They come with a generating family
\[(x_{-\alpha}, h_\alpha, x_\alpha)_{\alpha\in\Pi}, \ \ \text{where}\ \
             h_\alpha\in \h\subset \g, \ \  x_\alpha\hs,\hs x_{-\alpha}\in \g.\]
These elements $x_{-\alpha}, h_\alpha, x_\alpha$ satisfy
conditions (16 -- 22) of  Bourbaki \cite[Section VIII.4.3]{Bourbaki79},
in particular,
\[  [h_\alpha,h_\beta]=0,\quad [h_\alpha, x_\beta]=n(\beta,\alpha)\hs x_\beta, \quad
    [h_\alpha, x_{-\beta}]=-n(\beta,\alpha)\hs x_{-\beta},\quad
    [x_\alpha,x_{-\alpha}]=-h_\alpha  \]
for $\alpha,\beta\in\Pi$, where $n(\beta,\alpha)=\langle\beta,\alpha^\vee\rangle$.
The family $(h_\alpha)_{\alpha\in\Pi}$ is a basis of $\h$.
Let $\h^*$ denote the dual vector $\K$-space to $\h$.
The homomorphism $\eta\mapsto\eta_\PPi\colon Q\to\h^*$
such that $\langle\eta_\PPi,h_\alpha\rangle=\langle\eta,\alpha^\vee\rangle$,
takes $\Phi$ to the root system of $(\g,\h)$.
We have a root decomposition
\[\g=\h\oplus\bigoplus_{\eta\in\Phi}\g_\eta\hs,\]
where in $\g_\eta$ we write $\eta$ instead of $\eta_\PPi$,
and where $g_\alpha=\K\hs x_\alpha$ for $\alpha\in\Pi\subset \Phi$.

Set
\[\g_0=\h\oplus\bigoplus_{\eta\in\Phi_0}\g_\eta,\quad
   \g_1=\bigoplus_{\eta\in\Phi_1}\g_\eta,\quad
   \g_{-1}=\!\bigoplus_{\eta\in\Phi_{-1}}\g_\eta.\]
It follows from \eqref{e:Phi-grading} that
\[ [g_\iota\hs,\hs\g_\vk]\subset\g_{\iota+\vk}\quad\text{for}\ \iota,\vk\in\Z_3.\]
In other words,
\[\g=\g_0\oplus\g_1\oplus\g_{-1}\]
is a $\Z_3$-grading of $\g$.
In particular, $\g_0$ is a Lie subalgebra of $\g$,
and the subspaces $\g_1$ and $\g_{-1}$ are naturally $\g_0$-modules.

Set
\[    \alpha_8^\AAA=\ve_8-\ve_9,\quad
      \alpha_i^\AAA=\alpha_i\ \text{for}\ i=1,\dots,7,\quad
      \Pi_0=\{\alpha_1^\AAA,\dots,\alpha_7^\AAA,\alpha_8^\AAA\}.\]
Clearly, $\Phi_0$ is a root system of type $\AAA_8$, and $\Pi_0$ is a basis of $\Phi_0$.
Therefore, $\g_0\simeq\ssl(9,\K)$.
For the simple root $\alpha^\AAA_8$ of $(\g_0,\h)$,
we choose a nonzero element $x_{\alpha_8^\AAA}\in \g_{\alpha_8^\AAA}$.
We define $h_8^\AAA\in\h$ by the formula
\[ [h_8^\AAA, x_\beta]=n(\beta,\alpha_8^\AAA)\hs x_\beta\ \ \text{for}\ \beta\in\Pi_0,\]
and we define $x_{-\alpha_8^\AAA}\in \g_{-\alpha_8^\AAA}$ by the formula
\[ [x_{\alpha_8^\AAA},x_{-\alpha_8^\AAA}] = -h_8^\AAA.\]
We set
\[h_i^\AAA=h_{\alpha_i}\ \ \text{for}\ i=1,\dots,7.\]
We specify an isomorphism $\g_0\isoto \ssl(9,\K)$ by sending
\[h_i^\AAA\mapsto E_{i,i}-E_{i+1,i+1},\quad x_{\alpha_i^\AAA}
      \mapsto E_{i,i+1},\quad x_{-\alpha_i^\AAA}\mapsto -E_{i+1,i}
      \ \ \text{for}\ i=1, \dots, 8,\]
where $E_{i,j}\in\gl(9,\K)$ is the matrix with the $(i,j)$-entry $1$ and all other entries $0$.
This isomorphism defines a structure of $\ssl(9,\K)$-module on $\g_1$.
From the formula for $\Phi_1$ we see that $\g_1$ is an irreducible $\ssl(9,\K)$-module
with highest weight $\ve_1+\ve_2+\ve_3$.
Therefore, it is isomorphic to $\bigwedge^3 \K^9$.
We specify an isomorphism
\begin{equation}\label{e:bigwedge2}
\g_1\isoto\bigwedge^3 \K^9
\end{equation}
by sending $x_{\alpha_8}$ to $e_{678}\ce e_6\wedge e_7\wedge e_8$
(recall that $\alpha_8=\ve_6+\ve_7+\ve_8$).

Set $G=\Aut(\g)$. It is a split simple algebraic $\K$-group of type $\EEE_8$,
simply connected and of adjoint type, with Lie algebra $\g$.
The Lie algebra $\ssl(9,\K)$ is the Lie algebra
of the simply connected simple algebraic $\K$-group $\SL_{9,\K}$.
Since $\SL_{9,\K}$ is connected and  simply connected,
the homomorphism of Lie algebras
\[\psi\colon \ssl(9,\K)\isoto \g_0\into \g \]
is the differential of a uniquely defined  homomorphism of algebraic $\K$-groups
\begin{equation}\label{e:SL(9)-G}
\Psi\colon\SL_{9,\K}\to G.
\end{equation}
Let $G_0\subset G$ denote the image of this homomorphism,
which is an algebraic $\K$-subgroup of $G$ with Lie algebra $\g_0$.
From the formulas for $\Phi_0$, $\Phi_1$, $\Phi_{-1}$, and $\Phi$ we see
that the kernel of the homomorphism \eqref{e:SL(9)-G} is
\[\mu_3=\{\diag(z,\dots,z)\mid z\in \ov \K,\ z^3=1\},\]
where $\ov \K$ denotes an algebraic closure of $\K$.
which is a central algebraic $\K$-subgroup of $\SL_{9,\K}$.
Thus we may and shall identify the algebraic $\K$-groups  $G_0$ and $\SL_{9,\K}/\mu_3$.

Since the algebraic $\K$-subgroup $G_0\subset G=\Aut(\g)$ is connected,
and its Lie algebra $\g_0$ preserves $\g_1$,
we see that $G_0$ itself preserves $\g_1$.
Thus we can regard $\g_1$ as a $G_0$-module and as an $\SL_{9,\K}$-module.
Our isomorphism \eqref{e:bigwedge2} $\g_1\isoto\bigwedge^3 \K^9$  is $\g_0$-equivariant,
and hence $G_0$-equivariant, because $G_0$ is connected.
Thus the isomorphism \eqref{e:bigwedge2}
is $\SL_{9,\K}$-equivariant.

The homomorphism  \eqref{e:SL(9)-G} induces homomorphism on $\K$-points
\begin{equation*}
\Psi\colon \SL(9,\K)\to G_0(\K).
\end{equation*}
If $\K=\C,\R$, then this homomorphism is bijective (for $\K=\R$  by \cite[Corollary 3.3.14]{BGL2021}).

An element $x\in \g_1$ is  called {\it semisimple}
(respectively {\it nilpotent}) if the linear operator $\mathrm{ad\,} x \colon \g\to\g$
is  semisimple  (respectively nilpotent)  in $\g$.
By  \cite[\S 2]{Le2011} the \emph{homogeneous} Jordan decomposition holds:
any $x\in \g_1$ has a unique decomposition $x= x_s + x_n$  with  $x_s, x_n \in \g _1$,
where  $x_s$ is semisimple, $x_n$ is nilpotent, and $[x_s, x_n] = 0$.

It follows that the elements of $\g_1$ 
are naturally divided into three classes consisting,
respectively, of the nilpotent elements, the
semisimple elements, and the elements that are neither semisimple nor
nilpotent. We say that the elements of the third class are of {\em mixed}
type. Thus each of the problems of classification of the $\SL(9,\C)$-orbits in $\g_1^\cC$ and
the $\SL(9,\R)$-orbits in $\g_1=\g_1^\rR$  naturally splits into three subproblems.

\subsection{Nilpotent elements and homogeneous $\ssl_2$-triples}\label{sec:sl2t}

A triple $(h,e, f)$ of elements  in $\g$ is called an
{\em $\ssl_2$-triple} if
$$ [e,f]=h,\:  [h,e]=2e,\: [h,f]=-2f.$$
We say that an $\ssl_2$-triple $(h,e,f)$ is {\em homogeneous} if
$h\in \g_0$, $e\in \g_1$, $f\in \g_{-1}$.

The element $h$ is called the {\em characteristic} of the triple $(h,e,f)$.

\begin{proposition}[Jacobson-Morozov-Vinberg theorem  for a $\Z_m$-graded semisimple Lie algebra]
\label{prop:JMVL}
We write $\g^\rR$ for $\g$.
For $\K=\C,\R$, let $e \in \g _1^\kK$ be  a  nonzero  nilpotent element.
\begin{enumerate}[{\normalfont (i)}]	
\item There is a semisimple element $h \in \g _0^\kK$ and a nilpotent element  $f \in \g_{-1}^\kK$   such that
$$ [h, e] = 2e, \, [h, f] = -2f, \, [e,f] = h.$$
\item The element $h$ is defined  uniquely up to  conjugacy via an element in the centralizer $\Zz _{ G _0(\K)} (e)$   of $e$ in $G_0 (\K)$.\\
\item  Given $e$ and $h$, the element $f$ is defined uniquely.
\end{enumerate}
\end{proposition}

\begin{proof}
For $\K=\C$ see Vinberg \cite[Theorem 1]{Vinberg1979}. For $\K=\R$ see \cite[Theorem 2.1]{Le2011}.
See also \cite[Lemma 8.3.5]{Graaf2017}.
\end{proof}

For $\K=\C,\R$, let $\sT^\kK$ be the set of homogeneous $\ssl_2$-triples in $\g^\kK$.
We also write $\sT$ for $\sT^\rR$.
The group $\SL(9,\K)$ acts on $\sT^\kK$.

\begin{corollary}\label{cor:sl2conj}
For $\K=\C,\R$, let $(h,e,f)$ and  $(h',e',f')$ be
two homogeneous $\ssl_2$-triples in $\g^\kK$.
Then $e,e'$ are $\SL(9,\K)$-conjugate if and only if there exists $g\in
\SL(9,\K)$ with $(g\cdot h,\hs g\cdot e,\hs g\cdot f) = (h',e',f')$.
\end{corollary}

\begin{proof}
Suppose that there is $g_1\in \SL(9,\K)$ such that $g_1\cdot e = e'$.
Then the image of $g_1$ in $G_0(\K)$ also maps $e$ to $e'$.
From Proposition \ref{prop:JMVL}(ii) and (iii), it now follows that there exists
$g_0\in G_0(\K)$ with $g_0 \cdot h=h'$, $g_0 \cdot e=e'$,
$g_0 \cdot f = f'$. 
The homomorphism
$\Psi \colon \SL(9,\K)\to G_0(\K)$ is bijective (see above).
So the preimage $g$ of $g_0$ in $\SL(9,\K)$ does the job.
\end{proof}

We note the following important fact: each $\SL(9,\C)$-orbit in $\sT^\cC$ has
a real representative. This follows immediately from the classification of
these orbits in \cite{VE1978}, where a real representative is given for each
orbit. It can also be proved more conceptually and more generally, see
\cite[Proposition 4.3.22]{BGL2021}.

\begin{theorem}\label{thm:galois1}
Let $e \in \g_1$ be a nilpotent element and $t=(h, e, f)$ be a homogeneous
$\ssl_2$-triple containing $e$.
Write
$$\Zm_{\SL(9,\C)}(t) = \{ g\in \SL(9,\C) \,\mid\, g\cdot f=f,\,  g\cdot h=h,\,
g\cdot e=e\}.$$
The $\SL(9,\R)$-orbits in $(\SL(9,\C)\cdot e) \,\cap\, \g_1$ correspond
bijectively to the elements of $\Ho^1 \Zm_{\SL(9,\C)}(t)$.
\end{theorem}

\begin{proof}
By Corollary \ref{cor:sl2conj} the $\SL(9,\R)$-orbits in $(\SL(9,\C)\cdot e)\, \cap\, \g_1$
correspond bijectively to the  $\SL(9,\R)$-orbits in  $(\SL(9,\C)\cdot t)\, \cap\,\sT$.
By Proposition \ref{p:coh-orbits} the latter correspond bijectively
to $\ker \big[\hs\Ho^1\hs\Zm_{\SL(9,\C)}(t)\to\Ho^1\hs\hs\SL(9,\C)\hs\big]$.
Since $\Ho^1\hs\hs \SL(9,\C)=1$, the theorem follows.
\end{proof}

\begin{remark}\label{rem:char}
Let $(h,e,f)$ be a homogeneous $\ssl_2$-triple in $\g^\cC$ and
consider the element
$h'=(\psi^\cC)^{-1}(h) \in \ssl(9,\C)$. It is not difficult to see that
$h'$ has rational eigenvalues. So $h'$ is $\SL(9,\C)$-conjugate to a
unique real diagonal matrix $h''$ with weakly decreasing diagonal entries.
Furthermore, the differences of two eigenvalues of $h'$ are integral.
In the second column of
Table \ref{tab:orbitreps} we write the \emph{indices}  of  $h'$, i.e.
the value $(\eps_i-\eps_{i+1})(h')$, where $\eps_i -\eps_{i+1}$ are the simple
roots of $\ssl (9, \C)$ with respect to the Cartan subalgebra consisting of
the diagonal matrices.
\end{remark}

\subsection{Semisimple elements in $\g_1^\cC$}\label{subs:ss}

For $\K=\C,\R$, a {\em Cartan subspace} in $\g_1^\kK$  is, by definition,  a maximal subspace in
$\g_1^\kK$  consisting of commuting semisimple elements
(cf. \cite{Vinberg1976} and \cite[\S 2]{Le2011}).
In \cite{VE1978} a specific Cartan subspace $\Cg^\cC\subset\g_1^\cC$ is given
(it is described in Section \ref{sec:tabsemsim} above).

Define
\begin{align*}
\Nm_{\SL(9,\C)}(\Cg^\cC) &=\{g\in\SL(9,\C)\mid g\cdot \Cg^\cC=\Cg^\cC\},\\
\Zm_{\SL(9,\C)}(\Cg^\cC) &=\{g\in \SL(9,\C)\mid g\cdot p=p\text{ for all }
p\in\Cg^\cC\}.
\end{align*}
Then $W = \Nm_{\SL(9,\C)}(\Cg^\cC)/\Zm_{\SL(9,\C)}(\Cg^\cC)$ is called the Weyl
group (or also little Weyl group) of the $\Z_3$-graded Lie algebra $\g^\cC$.

In \cite{Vinberg1976} the following is shown:
\begin{itemize}
	\item Every semisimple $\SL(9,\C)$-orbit in $\g_1^\cC$ has a point in $\Cg^\cC$.
	\item Two elements of $\Cg^\cC$ are $\SL(9,\C)$-conjugate if and only if
	they are $W$-conjugate.
\end{itemize}

In \cite{VE1978} it is shown how to refine these statements. Seven
{\em canonical subsets} $\Fm_1^\cC,\ldots, \Fm_7^\cC$ of $\Cg^\cC$ along with
finite groups $\Gamma_{1},\ldots,\Gamma_{7}$ are determined such that
\begin{itemize}
	\item Every semisimple $\SL(9,\C)$-orbit has a point in precisely one of
	the $\Fm_k^\cC$.
	\item Each group $\Gamma_k$ acts on $\Fm_k^\cC$. Furthermore,
	  two elements of $\Fm_k^\cC$ are $\SL(9,\C)$-conjugate if and only if
          they are $\Gamma_{k}$-conjugate.
\end{itemize}

Here we summarize some of the constructions of \cite{VE1978} that are
used to establish this.

Let $\c^\cC$ be a Cartan subalgebra of $\g^\cC$ containing $\Cg^\cC$. It turns out
that such a Cartan subalgebra  $\c^\cC$ is unique and that
$\c^\cC = (\c^\cC\cap\g_{-1}^\cC) \oplus \Cg^\cC$.
Let $\Pi$ be the root system of $\g^\cC$ with respect to $\c^\cC$.

The $\Z_3$-grading $\g^\cC = \g^\cC_{-1} \oplus \g^\cC_0 \oplus \g^\cC_1$
yields an automorphism $\theta \colon \g^\cC \to \g^\cC$ defined by
$\theta(x) = \zeta^i x$ for $x\in \g_i^\cC$, where $\zeta\in\C$ is a fixed primitive
third root of unity.
We see that $\c^\cC$ is $\theta$-stable and we consider
the dual map on the dual space   $\c^{\cC,*}$ to $\c^\cC$:
\[ \theta^{*}\colon \c^{\cC,*}\to \c^{\cC,*},\quad
    (\theta^*\gamma) (c) = \gamma(\theta^{-1}c)\ \ \text{for}\ \gamma\in\c^{\cC,*},\, c\in \c^\cC.\]
We have $\theta^{*\hs2}+\theta^*+1=0$,
so that for $\alpha\in \Pi$ the set
$\Pi(\alpha)\ce \{ \pm\alpha,\,\pm \theta^*(\alpha),\, \pm \theta^{*\hs2}(\alpha)\}$
lies in a 2-dimensional space.
Furthermore, it is a root subsystem of type $\AAA_2$.
The sets $\Pi(\alpha)$ form a partition of $\Pi$, and hence there are 40 such sets.
Note that for  $\alpha\in \Pi\subset\c^{\cC,*}$ and $c\in \Cg^\cC\subset\c^\cC$ we have
$\theta^*(\alpha)(c) = \zeta^2 \alpha(c)$.
Hence the restrictions of the elements of $\Pi(\alpha)\subset\c^{\cC,*}$
to $\Cg^\cC\subset\c^\cC$ are scalar multiples of the restriction of $\alpha$.

By definition a {\em complex reflection} is a linear transformation $r$
of a complex vector space $V$ such that there is a basis of $V$ with respect to
which $r$ has the matrix $\diag(\omega,1,\ldots,1)$, where $\omega$ is a
primitive $m$-th root of unity for some $m\geq 2$ (see
\cite[Definition 1.7]{LT2009}, \cite[Definition 3.6]{Wallach2017}).
A complex reflection group is a subgroup of $\GL(V)$ generated by complex
reflections. Vinberg has shown that the Weyl
group $W$ is a complex reflection group in $\GL(\Cg^\cC)$
(\cite[Theorem 8]{Vinberg1976}, \cite[Theorem 3.69]{Wallach2017}).
In \cite{VE1978} a complex reflection $w_\alpha$ of order 3 is constructed
corresponding
to each $\Pi(\alpha)$. So we have 40 of those complex reflections, and together they
generate
the Weyl group $W$. (In fact, $W$ is already generated by four of those
complex reflections.) In \cite{VE1978}
shown that there exist $p_\alpha\in \Cg^\cC$ such that
\begin{equation}\label{eq:refl}
w_\alpha(h) = h -\alpha(h) p_\alpha \text{ for } h\in \Cg^\cC.
\end{equation}

Now let $p\in \Cg^\cC$. The centralizer $\z(p)$ of $p$ in $\g^\cC$ is given by
\begin{equation}\label{eq:pcen}
\z(p) = \c^\cC \oplus\! \bigoplus_{\substack{\alpha\in \Pi\\ \alpha(p)=0}} \g_\alpha^\cC\hs,
\end{equation}
where $\g_\alpha^\cC$ denotes the root subspace in $\g^\cC$
corresponding to a root $\alpha$.

Let $W_p$ be the subgroup of $W$ generated by the complex reflections $w_\alpha$ where $\alpha$
is such that $\alpha(p)=0$, or equivalently $w_\alpha(p)=p$. From
\eqref{eq:refl} and \eqref{eq:pcen} it follows that $\z(p)$ and $W_p$
determine each other. Indeed, if we know the group $W_p$\hs, then we know the
set of $w_\alpha$ contained in it, and hence we know all $\alpha\in \Pi$ such that
$\alpha(p)=0$; this in turn determines $\z(p)$. (Also note that if
$\alpha(p)=0$, then $\beta(p)=0$ for all $\beta\in \Pi(\alpha)$.) The
argument for the converse is similar. By \cite[Proposition 14]{Vinberg1976}
the stabilizer in $W$ of an element $p\in \Cg^\cC$ is generated by
complex reflections.
Hence $W_p$ coincides with the stabilizer of $p$ in $W$, that is
$$W_p = \{ w\in W \mid w\cdot p = p\}.$$

Define
\begin{align*}
\Cg_p^\cC &= \{ h\in \Cg^\cC \mid w\cdot h=h \text{ for all } w\in W_p\}\\
\Cg_p^{\cC,\circ} & = \{ q\in \Cg_p^\cC \mid W_q = W_p \}.
\end{align*}
Then $\Cg_p^{\cC,\circ}$ is a Zariski-open subset of $\Cg_p^\cC$ because it is defined
by the inequalities $w\cdot q \neq q$ for all $w\in W\setminus W_p$.
It is straightforward
to check that for $p,q\in \Cg^\cC$ and $v\in W$ we have
\begin{equation}\label{eq:CpCq}
\Cg_q^\cC = v\cdot \Cg_p^\cC\  \text{ if and only if }\  W_q = v\hs W_p\hs v^{-1}.
\end{equation}
This yields the following. Let $R=\{ w_\alpha\in W \mid \alpha\in \Pi\}$.
Then $R$ is a single conjugacy class in $W$; this is stated in \cite{VE1978},
Section 3.3, page 84 of the English version; we have also checked it by computer.
Note that for $\beta\in \Pi(\alpha)$ we have $w_\beta = w_\alpha$, so $R$ consists
of 40 elements. Let $w$ be a complex reflection in $W$, then it is known
that either $w\in R$ or $w^{-1}\in R$; see \cite[Table D.2]{LT2009}.
A subgroup of $W$ is said to be a reflection
subgroup if it is generated by complex reflections. Since the set $R$ along with
the inverses of the elements of $R$ exhaust all complex reflections in $W$, it follows
that the reflection subgroups of $W$ are the subgroups generated by elements
of $R$. Because a conjugate of a complex reflection
is a complex reflection we have that a conjugate of a reflection subgroup
is a reflection subgroup as well. It can be shown that each
reflection subgroup arises as  $W_p$ for an element $p\in \Cg^\cC$; see
\cite[Section 3.4]{VE1978}. So there are
$q_1,\ldots,q_m\in \Cg^\cC$ be such that the $W_{q_i}$ are representatives of the
conjugacy classes of reflection subgroups.
Then each $q\in \Cg^\cC$ is $W$-conjugate to an element of precisely one of the
$\Cg_{q_i}^\circ$. It turns out that $m=7$, so there are seven ``canonical
sets'' of semisimple elements $\Fm_k^\cC=\Cg_{q_k}^{\cC,\circ}$.
Each semisimple element of
$\g_1^\cC$ is $\SL(9,\C)$-conjugate to an element of precisely one canonical
set. In \cite{VE1978} explicit descriptions of the sets $\Fm_k^\cC$ are given;
see also Section \ref{sec:tabsemsim} above.

Now define
$$N(W_p) = \{ v\in W \mid v\hs W_p\hs v^{-1} = W_p\}.$$
By \eqref{eq:CpCq} it is clear that $N(W_p) = \{ v\in W \mid v\cdot\Cg_p^\cC =
\Cg_p^\cC\}$.

\begin{lemma}\label{lem:WC}
Let $p_1,p_2\in \Cg_p^{\cC,\circ}$ and let $w\in W$ be such that
$w\cdot p_1 = p_2$. Then $w\in N(W_p)$.
\end{lemma}

\begin{proof}
From $wp_1=p_2$ it follows that $W_{p_2} = wW_{p_1}w^{-1}$. Since
$p_i\in \Cg_p^{\cC,\circ}$, we have $W_{p_1}=W_{p_2}=W_p$. Hence $w\in N(W_p)$.
\end{proof}

Define $\Gamma_p = N(W_p)/W_p$. Then $\Gamma_p$ acts naturally
on $\Cg_p^\cC$. Let $p_1,p_2\in \Cg_p^{\cC,\circ}$.
Then $p_1,p_2$ are $\SL(9,\C)$-conjugate
if and only if they are $W$-conjugate,
if and only if they are $N(W_p)$-conjugate (by Lemma \ref{lem:WC}),
if and only if they are $\Gamma_p$-conjugate.
For $1\leq k\leq 7$ set $\Gamma_k = \Gamma_{q_k}$.
In \cite{VE1978} in each of the seven cases, the group $\Gamma_k$ is determined.

%%%%%%%%%%%%%%%%%%%%%%%%%%%%%%%%%%%%classification.tex%%%%%%%%%%%%%%%%%%%%%%%

\section{Classification of the orbits of $\SL(9,\R)$ on $\bigwedge^3 \R^9$}
\label{sec:methods}

In this section we describe the methods that we used to classify the
orbits of $\SL(9,\R)$ on $\bigwedge^3 \R^9$. We use the setup of
Section \ref{sec:gradedLie}. Throughout we write
$\Gtil_0=\Gtil_0(\C)=\SL(9,\C)$ and $\Gtil_0(\R) = \SL(9,\R)$.
For $\K=\C,\R$, we identify  the vector spaces $\g_1^\kK$ and $\bigwedge^3 \K^9$
on which the group $\Gtil_0(K)$ acts.
 By {\em complex orbits} we mean the $\Gtil_0(\C)$-orbits in $\g_1^\cC=\bigwedge^3 \C^9$,
and by {\em real orbits} we mean the $\Gtil_0(\R)$-orbits in $\g_1^\rR=\bigwedge^3 \R^9$.
Then any real orbit is contained in a complex orbit,
and any complex orbit contains finitely many  real orbits.
As seen in Section \ref{sec:gradedLie}, the orbits
are divided into three groups: nilpotent, semisimple and mixed. For each
we have a subsection.

\subsection{The nilpotent orbits}
As seen in Section \ref{sec:sl2t},
the nilpotent orbits over $\K$ are in bijection with
the orbits of homogeneous $\ssl_2$-triples over $\K$, for $\K=\C,\R$.
The complex nilpotent orbits are listed in \cite{VE1978}.
As noted in Section \ref{sec:sl2t}, any complex nilpotent orbit has a real representative $e$.
By Proposition \ref{prop:JMVL}(i), there exists
a real homogeneous $\ssl_2$-triple $t=(h,e,f)\in\sT^\rR$ containing $e$.
By Theorem \ref{thm:galois1}, the
$\Gtil_0(\R)$-orbits contained in $\Gtil_0(\C)\cdot e$ are in a canonical bijection with
the elements of the Galois cohomology set $\Ho^1 \Zm_{\Gtil_0}(t)$. So the classification of
these orbits involves the following steps:
\begin{enumerate}
\item Find $t=(h,e,f)\in\sT^\rR$ containing $e$.
\item Determine $\Zm_{\Gtil_0}(t)$.
\item Determine $\Ho^1 \Zm_{\Gtil_0}(t)$.
\item For each $[a]\in \Ho^1 \Zm_{\Gtil_0}(t)$, where $a\in\Zl^1 \Zm_{\Gtil_0}(t)$,
  find an element $g\in \Gtil_0$ with $g^{-1}\ov g = a$.
  Then $g\cdot e$ is a representative of the
  $\Gtil_0(\R)$-orbit corresponding to $[a]$.
\end{enumerate}

In Step 1 we first find an $\ssl_2$-triple in $\g$ that contains $e$.
This is done by solving linear equations, according to an algorithm that
closely follows the proof of the existence of such an $\ssl_2$-triple; we refer
to \cite[Section 2.13]{Graaf2017} for the details. If the constructed
$\ssl_2$-triple is not homogeneous, then we find a homogeneous $\ssl_2$-triple
containing $e$ following the steps outlined in the proof of
\cite[Lemma 8.3.5]{Graaf2017}.

In Step 2 we need to find a description of $\Zm_{\Gtil_0}(t)$ that is as detailed
as possible in order to be able to execute Step 3. It is straightforward to
obtain polynomials in 81 indeterminates whose zero locus is $\Zm_{\Gtil_0}(t)$.
By themselves they do not give a useful description of the group. However,
by computing a Gr\"obner basis (see \cite{CLO15})
of the ideal that these polynomials generate,
we are often able to find a set of polynomials defining the same group and
from which it is straightforward to read off the group structure. There are
also quite a few cases for which it is computationally too hard to compute
a Gr\"obner basis, or for which the Gr\"obner basis does not yield the desired
information. For those cases we have developed ad hoc computational methods.
We refer to \cite{BGL2021} for details.

Step 3 is essentially carried out by hand, considering each case individually.
The paper \cite{BGL2021} contains complete descriptions of the centralizers
$\Zm_{\Gtil_0}(t)$ as well as detailed computations of the sets
$\Ho^1 \Zm_{\Gtil_0}(t)$.

For Step 4 we use a computational method. The equation $g^{-1}\ov g=a$ is
the same as $\ov g = ag$. The latter is equivalent to a set of linear
equations over $\R$ for the coefficients of $g$. We solve these equations and
in the solution space we look for an element lying in $\Gtil_0$.

\begin{example}
Consider the nilpotent orbit with representative
$e=e_{136}+e_{147}-e_{245}+e_{379}+e_{569}+e_{678}$ (this is the orbit with number 47
in Table \ref{tab:orbitreps}).
We compute a homogeneous $\ssl_2$-triple $t=(h,e,f)$ containing $e$;
let $\Zm_0=\Zm_{\Gtil_0}(t)$ denote its stabilizer.
Computer calculations show that the identity component $\Zm_0^\circ$ consists of
$$X(a,b) = \diag(a^{-1}b^{-1},a^{-2},a,a^2b^2,b^{-2},b,a^{-1}b^{-1},a,b),
\text{ for } a,b\in \C^\times.$$
The component group $C$ is of order 2 and generated by the image $c_1$ of
$$g_0=\SmallMatrix{ 0 & 0 & 0 & 0 & 0 & 0 & 0 & 0 & -1\\
                0 & -1 & 0 & 0 & 0 & 0 & 0 & 0 & 0\\
                0 & 0 & 1 & 0 & 0 & 0 & 0 & 0 & 0\\
                0 & 0 & 0 & 0 & -1 & 0 & 0 & 0 & 0\\
                0 & 0 & 0 & -1 & 0 & 0 & 0 & 0 & 0\\
                0 & 0 & 0 & 0 & 0 & 0 & -1 & 0 & 0\\
                0 & 0 & 0 & 0 & 0 & 1 & 0 & 0 & 0\\
                0 & 0 & 0 & 0 & 0 & 0 & 0 & 1 & 0\\
                1 & 0 & 0 & 0 & 0 & 0 & 0 & 0 & 0\\}. $$
We have
\[g_0^2=X(1,-1),\quad g_0 X(a,b)\hs g_0^{-1}=X(a,a^{-1}b^{-1}).\]
Set $g_1=g_0\cdot X(-1,1)$; then $g_1^2=1$.
In particular, $g_1$ is a cocycle with image $c_1$ in $C$.
We have a short exact sequence
$$1\to \Zm_0^\circ\to \Zm_0 \labelto{j} C\to 1.$$
By \cite[Section I.5.5, Proposition 38]{Serre1997} this
yields an exact sequence
$$\Ho^1 \hm \Zm_0^\circ \to\Ho^1 \hm \Zm_0\labelto{j_*} \Ho^1\hs C.$$
Since $C$ is a group of order 2, we have
$\Ho^1 \hs C = \{ [1], [c_1]\}$. So
$$\Ho^1 \hm \Zm_0 = j_*^{-1}([1]) \cup j_*^{-1} ([c_1]).$$
We have $\Ho^1\hm \Zm_0^\circ=1$, hence $j_*^{-1}([1])=\{[1]\}$.
By twisting the above exact sequence by the cocycle $g_1$ it can be shown
that $j_*^{-1}([c_1])=\{[g_1]\}$. We refer to \cite[Section 3.1]{BGL2021} for
a description of this technique. Furthermore, \cite[Section 6]{BGL2021}
has the details of the proof in this case.
It follows that $\Ho^1\hm \Zm_0=\{[1], [g_1]\}$.

By computer we compute a basis of the real vector space consisting of all
$9\times 9$ complex matrices $u$ with $\ov u = g_1u$. In this space we select
9 elements that have a matrix of maximal rank in their span, and that commute
with the image of $h$ in $\ssl(9,\C)$. By some random tries we find an element
of determinant 1 in the space spanned by these 9 elements. It is
$$u_0 = \SmallMatrix{ i & 0 & 0 & 0 & 0 & 0 & 0 & 0 & i\\
                0 & -2i & 0 & 0 & 0 & 0 & 0 & 0 & 0\\
                0 & 0 & 0 & -\tfrac{1}{2} & \tfrac{1}{2} & 0 & 0 & 0 & 0\\
                0 & 0 & \tfrac{1}{2}i & 0 & 0 & 0 & 0 & 0 & 0\\
                0 & 0 & 0 & -\tfrac{1}{2}i & -\tfrac{1}{2}i & 0 & 0 & 0 & 0\\
                0 & 0 & 0 & 0 & 0 & i & -i & 0 & 0\\
                0 & 0 & 0 & 0 & 0 & 1 & 1 & 0 & 0\\
                0 & 0 & 0 & 0 & 0 & 0 & 0 & \tfrac{1}{2}i & 0\\
                1 & 0 & 0 & 0 & 0 & 0 & 0 & 0 & -1\\}. $$
We have
$$u_0\cdot e =-e_{136}-e_{147}+e_{157}-e_{235}-e_{379}-e_{469}-e_{569}-e_{678}.$$
We conclude that the $\SL(9,\C)$-orbit of $e$ contains two $\SL(9,\R)$-orbits
with representatives $e$ and $u_0\cdot e$.
\end{example}

\subsection{The semisimple orbits}\label{sec:semsim}

In this section we consider the problem to classify the semisimple
$\Gtil_0(\R)$-orbits in $\g_1$. We use the notation and results described
in Section \ref{subs:ss}.

Let $\Fm^\cC=\Fm_i^\cC$ be one of the canonical subsets of $\Cg^\cC$.
Then there is a trivector $r\in \Cg^\cC$ with $\Fm = \Cg_{r}^{\cC,\circ}$.
Let $p\in \Fm^\cC$ and let $\OOm=\Gtil_0(\C)\cdot p$ denote the $\Gtil_0(\C)$-orbit of
$p$. We wish to know whether $\OOm$ contains a real point
and whether $\OOm\cap\Fm^\cC$ contains a real point.
If $\OOm$ has a real point, we wish to classify the real orbits in $\OOm$.

For  $q\in \Cg$ define
$$\Zm_{\Gtil_0}(q) =\{g\in \Gtil_0(\C)\mid g\cdot q=q\}.$$

\begin{lemma}\label{lem:Zp}
Let $p,q\in \Fm$ then $\Zm_{\Gtil_0}(q)=\Zm_{\Gtil_0}(p)$.
\end{lemma}

\begin{proof}
  Since $p,q\in \Fm$ we have $W_p=W_q$ (notation as in Section
    \ref{subs:ss}). As seen in Section \ref{subs:ss}, the centralizer
    $\z_{\g^\cC}(p)$ can be determined from $W_p$. Hence $\z_{\g^\cC}(p)=
    \z_{\g^\cC}(q)$.
  We recall the following general fact:
  let $H^\cC$ be a reductive algebraic group over $\C$
  with Lie algebra $\h^\cC$, and  let $s\in \h^\cC$ be a semisimple element;
  then the algebraic subgroup $\{h\in H^\cC \mid \Ad(h)(s) = s \}\hs\subset H$
  is  connected; see Steinberg \cite[Theorem 3.14]{Steinberg1975}.
  It follows that $p$ and $q$ have the same stabilizer in $G$,
  and hence they have the same stabilizer in $G_0$. The inverse images of
  these stabilizers in
  $\Gtil_0$ are $\Zm_{\Gtil_0}(q)$, $\Zm_{\Gtil_0}(p)$, which are therefore
  equal as well.
\end{proof}

Define
\begin{align*}
  \Zm_{\Gtil_0}(\Fm^\cC) & = \{ g\in \Gtil_0 \mid g\cdot q=q \text{ for all }
  q\in \Fm^\cC\},\\
  \Nm_{\Gtil_0}(\Fm^\cC) & = \{ g\in \Gtil_0 \mid g\cdot q\in \Fm^\cC \text{ for all }
  q\in \Fm^\cC\}.
\end{align*}

\begin{lemma}\label{lem:gp1p2}
  Let $p_1,p_2\in \Fm^\cC$ and let $g\in \Gtil_0$ be such that
  $g\cdot p_1=p_2$. Then $g\in \Nm_{\Gtil_0}(\Fm^\cC)$.
\end{lemma}

\begin{proof}
  Let $q\in \Fm$. Since $p_1,p_2\in \Cg^\cC$ are $\Gtil_0$-conjugate they are
  $W$-conjugate (\cite[Theorem 2]{Vinberg1976}). So by Lemma \ref{lem:WC}
  there is $w\in N(W_r)$ such that $w\cdot p_1=p_2$,
  where $r\in \Cg^\cC$ was introduced in the beginning of this subsection.
   Let $\hat w\in \Nm_{\Gtil_0}(\Cg^\cC)$ be a preimage of $w$. Then $g^{-1}\hat w
  \in \Zm_{\Gtil_0}(p_1)$. Hence by Lemma \ref{lem:Zp} we see that
  $g^{-1}\hat w \in \Zm_{\Gtil_0}(q)$,
  from which it follows that $g\cdot q=\hat w\cdot q\in \Fm^\cC$.
\end{proof}

From Section \ref{subs:ss} we recall that $\Gamma_r = N(W_r)/W_r$.
We define a map $\varphi \colon \Nm_{\Gtil_0}(\Fm^\cC) \to \Gamma_r$. Let $g\in
\Nm_{\Gtil_0}(\Fm^\cC)$. Then $g\cdot r\in \Fm^\cC$, hence there is
$w\in N(W_r)$ such that $w\cdot r=g\cdot r$.
 We set $\varphi(g) = wW_r$. Note that $\varphi$
is well defined: if $w'\in N(W_r)$ also satisfies $w'\cdot r=g\cdot r$, then
$w^{-1}w'\in W_r$, so that $w' W_r = wW_r$.

\begin{lemma}\label{lem:Np}
The map $\varphi\colon \Nm_{\Gtil_0}(\Fm^\cC) \to \Gamma_r$
is a surjective group homomorphism with kernel $\Zm_{\Gtil_0}(\Fm^\cC)$.
Furthermore, for $g\in \Nm_{\Gtil_0}(\Fm^\cC)$ and $q\in \Fm^\cC$ we have
$g\cdot q = \varphi(g)\cdot q$.
\end{lemma}

\begin{proof}
  First we claim the following: let $g\in \Nm_{\Gtil_0}(\Fm^\cC)$ and let
  $w\in N(W_r)$ be such that $w\cdot r=g\cdot r$;
  then  $g\cdot q=w\cdot q$ for all $q\in \Fm^\cC$.
  Indeed, let $\hat w\in \Nm_{\Gtil_0}(\Cg^\cC)$ be a preimage
  of $w$. Then $\hat w^{-1} g\in \Zm_{\Gtil_0}(r)$. So by Lemma
  \ref{lem:Zp} it follows that $\hat w^{-1} g\in \Zm_{\Gtil_0}(q)$, or
  $g\cdot q=\hat w\cdot q=w\cdot q$.

  Our claim immediately implies the last statement of the lemma.

  Now let $g_1,g_2\in \Nm_{\Gtil_0}(\Fm^\cC)$ and let $w_1,w_2\in N(W_r)$ be
  such that $w_i\cdot r = g_i\cdot r$, $i=1,2$. Then by our claim
  we see that $g_1g_2\cdot r = w_1w_2\cdot r$ implying that $\varphi$ is a group
  homomorphism.

  Let $\hat w\in \Nm_{\Gtil_0}(\Cg^\cC)$ be a preimage of $w\in N(W_r)$. Then
  $\hat w \in \Nm_{\Gtil_0}(\Fm^\cC)$ and $\varphi(\hat w) = wW_r$, so that
  $\varphi$ is surjective. If $g\in \ker \varphi$, then $g\cdot r=r$, and by our
  claim we see that $g\in \Zm_{\Gtil_0}(\Fm^\cC)$.
\end{proof}

It follows that $\varphi$ induces an isomorphism, which we also denote by
$\varphi$, between $\Am =
\Nm_{\Gtil_0}(\Fm^\cC)/\Zm_{\Gtil_0}(\Fm^\cC)$ and $\Gamma_r$.

\begin{proposition}\label{prop:ssorb}
  As before, let $\OOm = \Gtil_0\cdot p$. Write $\Nm =  \Nm_{\Gtil_0}(\Fm^\cC)$,
  $\Zm=\Zm_{\Gtil_0}(\Fm^\cC)$.
\begin{enumerate}
  \item[\rm (i)]  $\OOm$ has an $\R$-point if and only if
    $\pbar=n^{-1}\cdot p$ for some $n\in \Zl^1\Nm$\hs.
  \item[\rm (ii)] Assume that $\OOm$ has an $\R$-point and let $n$ be as in (i).
    Write $a=n\Zm\in\Zl^1\Am$ and $\xi=[a]\in\Ho^1\Am$\hs. Then $\xi$ only
    depends on $\OOm$ and not on the choices of $p\in \OOm\cap \Fm^\cC$ and
    $n\in \Zl^1 \Nm$.
 \item[\rm (iii)] With the hypothesis and notation of (ii), the orbit $\OOm$
   has a real point in $\Fm^\cC$ if and only if $\xi=1$.
\end{enumerate}
\end{proposition}

\begin{proof}
(i) Assume that $\OOm$ has an $\R$-point $p_\R=g\cdot p$.
Then $\ov{g\cdot p}=g\cdot p$, whence
\[ \pbar=\gbar^{-1}\cdot g\cdot p=(g^{-1}\gbar)^{-1}\cdot p.\]
Write
\begin{equation}\label{e:n-g-gbar}
n= g^{-1}\gbar.
\end{equation}
Since $p,\pbar\in\Fm^\cC$\hs, we see that $n\in\Nm$ by Lemma \ref{lem:gp1p2}.
It follows from \eqref{e:n-g-gbar} that $n\in \Zl^1\Nm$\hs.

Conversely, assume that $\pbar=n^{-1}\cdot p$ where $n\in\Zl^1\Nm$.
Since $\Ho^1\Gtil_0=\{1\}$, there exists $g\in \Gtil_0$ such that
$n=g^{-1}\gbar$. Set
\[ p_\R=g\cdot p\in\OOm.\]
Then
\[\ov{p_\R}=\gbar\cdot\pbar=gn\cdot n^{-1}\cdot p = g\cdot p=p_\R\hs.\]
Thus $p_\R$ is an $\R$-point of $\OOm$, which proves (i).

(ii) Assume that $\OOm$ has an $\R$-point and let $n$ be as in (i).
We show that $\xi$ depends only on $\OOm$.
First suppose that $\pbar=\hat n^{-1}\cdot p$ for some $\hat n\in \Nm$.
Then $\hat n n^{-1} \in \Zm_{\Gtil_0}(p)$, which is equal to $\Zm$ by Lemma
\ref{lem:Zp}. Hence $\xi$ does not depend on the choice of $n$.

We show that $\xi$ does not depend on the choice of $p$.
Indeed, if $p'\in\OOm\cap\Fm^\cC$\hs,
then $p'=a'\cdot p$ for some $a'\in\Am$\hs,
and we have
\[\ov{p'}=\ov{a'}\cdot\pbar=\ov{a'}\cdot a^{-1}\cdot p
   =\ov{a'}\cdot a^{-1}(a')^{-1}\cdot a'\cdot p
   =(a'a\hs \ov{a'}^{\hs -1})^{-1}\cdot p'.\]
We obtain the 1-cocycle
\[a'a\hs\ov{a'}^{\hs -1}\sim a.\]
Thus $\xi=[a]$ does not depend on the choice of $p$.

(iii) Now assume that $\xi=1$. We have
\[\pbar=a^{-1}\cdot p\quad\text{and} \quad a=(a')^{-1}\hs \ov{a'}
   \text{ for some }a'\in\Am\hs.\]
Set $p_\R=a'\cdot p\in \OOm\cap\Fm^\cC$\hs.
Then
\[\ov{p_\R}=\ov{a'}\cdot\pbar=\ov{a'}\cdot a^{-1}\cdot p=\ov{a'}
     \cdot\ov{a'}^{\hs -1} \cdot a'\cdot p=a'\cdot p=p_\R\hs.\]
Thus $p_\R$ is real.

Conversely, if $\OOm\cap \Fm^\cC$ contains a real point $p_\R$\hs,
then clearly $\xi=1$, which proves (iii).
\end{proof}

Assume that $\OOm$ contains a real point $p_\R=g\cdot p$.
We wish to classify real orbits in $\OOm$. We write $q$ for $p_\R$.
Write $C_q=\Zm_{\Gtil_0}(q)$ and set
\[\CC_q=(C_q,\sigma_q),\quad\text{where }\sigma_q(c)=\ov c.\]

Similarly we write $C_p=\Zm_{\Gtil_0}(p)$.
Since $q=g\cdot p$, we have an isomorphism
\[ \iota_g\colon C_p\labelto{\sim} C_q\hs, \quad c\mapsto gcg^{-1}.\]
We {\em transfer} the real structure $\sigma_q$ on $C_q$ to $C_p$ using $\iota_g$.
We obtain a real structure $\sigma_p$ on $C_p$:
\begin{align*}
\sigma_p\colon\, C_p\labelto{\iota_g} C_q\labelto{\sigma _q} C_q\labelto{\iota_g^{-1}} C_p\hs,\quad
c\mapsto gcg^{-1}\mapsto\ov{gcg^{-1}}\mapsto g^{-1}\hs\ov{gcg^{-1}} g.
\end{align*}
Let $n=g^{-1}\ov g$ (see also the proof of Proposition \ref{prop:ssorb}), then
$\sigma_p(c)=n\hs\ov c\hs n^{-1}$ for $c\in C_p$.
We obtain a real algebraic group
\[ \CC_p=(C_p,\sigma_p),\quad \text{where } \sigma_p(c)=n\hs\ov c\hs n^{-1}\ \text{ for }c\in C_p\hs,\]
and an isomorphism
\[\iota_g\colon \CC_p\labelto{\sim}\CC_q,\quad c\mapsto g\hs c\hs g^{-1}\]
inducing a bijection on cohomology
\[\Ho^1\hs\CC_p\labelto{\sim}\Ho^1\hs\CC_q\hs.\]

By Proposition \ref{p:coh-orbits}, the real orbits in $\OOm$
are classified by $\Ho^1\hs\CC_q$\hs, and hence by $\Ho^1\hs\CC_p$\hs.
The map is as follows.
To $c\in\Zl^1\hs\CC_p$ we associate $gcg^{-1}\in \Zl^1\hs\CC_q$.
We find  $g_1\in \Gtil_0$ such that $g_1^{-1}\hs\ov g_1=g\hs c\hs g^{-1}$ and
set $r=g_1\cdot q =g_1\hs g\cdot p$.
To $[c]$ we associate the real orbit $\Gtil_0(\R)\cdot r\subseteq \OOm$.

In order to check our calculations, we show that $r=g_1\hs g\cdot q$ is real.
We calculate:
\begin{align*}
\ov r =\ov{g_1\hs g\cdot p}=\ov g_1\cdot\gbar\cdot \pbar= g_1\hs gcg^{-1}\cdot gn\cdot n^{-1}\cdot p
=g_1\hs gc\cdot p=g_1\hs g\cdot p=r,
\end{align*}
because $c\in C_p=\Zm_{\Gtil_0}(p)$. Thus $r$ is real.

This leads to the following procedure to list the real semisimple orbits
having a representative that is $\Gtil_0$-conjugate to an element of $\Fm^\cC$.
First we compute $\Ho^1\hm\Am$, and then for every $[a]\in \Ho^1\hm\Am$ we do
the following:

\begin{enumerate}
\item Find all $p\in \Fm^\cC$ such that $\pbar=a^{-1}\cdot p$.
\item Lift $a$  to a {\em cocycle} $n\in\Zl^1\Nm_{\Gtil_0}(\Fm^\cC)$ and
compute an element $g\in \Gtil_0$ with $g^{-1}\gbar =n$.
\item Set $C_p = \Zm_{\Gtil_0}(p)$ and define $\sigma_p\colon C_p\to C_p$
  by $\sigma_p(c)= n\hs\ov c\hs n^{-1}$ and set $\CC_p = (C_p,\sigma_p)$.
\item Compute $\Ho^1\hs \CC_p$ and for each $[c]\in \Ho^1\hs \CC_p$ find an element
  $g_1\in \Gtil_0$ with $g_1^{-1}\ov g_1 = gcg^{-1}$. Then $g_1g\cdot p$ is
  a representative of the semisimple $\Gtil_0(\R)$-orbit corresponding
  to $[c]$.
\end{enumerate}

On this procedure we remark the following. We have that $\varphi \colon \Am \to
\Gamma_r$ is a $\Gamma$-equivariant isomorphism. So it induces a bijection
between $\Ho^1\Am$ and $\Ho^1\hs \Gamma_r$.
The latter can be computed by brute force because $\Gamma_r$ is a known
finite group. In the cases that are relevant to our classification, the
lifting in Step 2 always turned out to be possible, but we cannot prove
this a priori.

The resulting classification of the semisimple orbits is described in
Section \ref{sec:tabsemsim}. For the details of the computations we
refer to \cite{BGL2021}. Using some of the results of our computations we
can also prove the following result.

\begin{theorem}\label{thm:cartanr}
All Cartan subspaces in $\g_1$ are conjugate under $\SL(9,\R)$.
\end{theorem}

\begin{proof}
The short exact sequence
\[1\to\Zm_{\Gtil_0}(\Cg^\cC)\to \Nm_{\Gtil_0}(\Cg^\cC)\to W\to 1,\]
gives rise to a cohomology exact sequence
\[ \Ho^1\Zm_{\Gtil_0}(\Cg^\cC) \to \Ho^1\Nm_{\Gtil_0}(\Cg^\cC)\to \Ho^1\hs W.\]
By brute force (computer) computations it is easily established that
$\Ho^1\hs W = 1$. We have explicitly determined $\Zm_{\Gtil_0}(\Cg^\cC)$, which is a
group of order $3^5$. By \cite[Corollary 3.2.5]{BGL2021}
we have $\Ho^1 \hm A=1$ for any finite group $A$ of order $p^m$, where $p$ is an odd prime.
It follows that $\Ho^1\Zm_{\Gtil_0}(\Cg^\cC)=1$ and hence $\Ho^1\Nm_{\Gtil_0}(\Cg^\cC)=1$. By
\cite[Theorem 4.4.9]{BGL2021}, the conjugacy classes of Cartan subspaces in
$\g_1$ are
in bijection with $\ker\big[ \Ho^1 \Nm_{\Gtil_0}(\Cg^\cC) \to \Ho^1 \Gtil_0\big]$, and the
theorem follows.
\end{proof}

\begin{example}\label{exa:Fm3}
Let $\Fm^\cC=\Fm_3^\cC$ be the third canonical set consisting of
$$p^{3,1}_{\lambda_1,\lambda_2} = \lambda_1 p_1+\lambda_2p_2$$
where $p_1,p_2$ are given in Section \ref{sec:tabsemsim} and the $\lambda_i\in
\C$ are such that $\lambda_1\lambda_2(\lambda_1^6-\lambda_2^6)\neq 0$.
The group $\Gamma_3$ is of order 72 and generated by
$$\SmallMatrix{0&-1\\-1&0},\  \SmallMatrix{-1&0\\0&\zeta},$$
where $\zeta$ is a primitive third root of unity. Here the elements of
$\Gamma_3$ are given as linear transformations of the space spanned by
$p_1,p_2$.

A small brute force computation shows that $\Ho^1\hs \Gamma_3 = \{[1],[-1],
[u_1],[u_2]\}$, where
$$u_1 = \SmallMatrix{-1&0\\0&1},\ \, u_2=\SmallMatrix{0&1\\1&0}.$$
So the elements of $\Fm^\cC$ whose orbits have real points are divided into
four classes consisting of $q\in \Fm^\cC$ with $\ov q=q$, $\ov q=-q$,
$\ov q = u_1^{-1}\cdot q$, and $\ov q = u_2^{-1}\cdot q$, respectively.

We consider the fourth class consisting of $q=\lambda_1 p_1 +\lambda_2p_2$ with
$\lambda_1=x+iy$, $\lambda_2=x-iy$, $x,y\in \R$. The polynomial conditions
on $\lambda_1,\lambda_2$ translate to $xy(x^2-3y^2)(x^2-\tfrac{1}{3}y^2)\neq 0$.
Let $\OOm_{x,y}$ denote the orbit of this element.
Set
$$n_3=\SmallMatrix{ -1 & 0 & 0 & 0 & 0 & 0 & 0 & 0 & 0 \\
                0 & 0 & 0 & 0 & 0 & 0 & -1 & 0 & 0 \\
                0 & 0 & 0 & -1 & 0 & 0 & 0 & 0 & 0 \\
                0 & 0 & -1 & 0 & 0 & 0 & 0 & 0 & 0 \\
                0 & 0 & 0 & 0 & 0 & 0 & 0 & 0 & -1 \\
                0 & 0 & 0 & 0 & 0 & -1 & 0 & 0 & 0 \\
                0 & -1 & 0 & 0 & 0 & 0 & 0 & 0 & 0 \\
                0 & 0 & 0 & 0 & 0 & 0 & 0 & -1 & 0 \\
                0 & 0 & 0 & 0 & -1 & 0 & 0 & 0 & 0},
\quad
g_3 = \SmallMatrix{ 0 & 0 & -1 & 1 & 0 & 0 & 0 & 0 & 0 \\
                0 & -1 & 0 & 0 & 0 & 0 & 1 & 0 & 0 \\
                0 & 0 & 0 & 0 & -\tfrac{1}{2} & 0 & 0 & 0 & \tfrac{1}{2} \\
                \tfrac{1}{2}i & 0 & 0 & 0 & 0 & 0 & 0 & 0 & 0 \\
                0 & 0 & i & i & 0 & 0 & 0 & 0 & 0 \\
                0 & 0 & 0 & 0 & 0 & i & 0 & 0 & 0 \\
                0 & i & 0 & 0 & 0 & 0 & i & 0 & 0 \\
                0 & 0 & 0 & 0 & 0 & 0 & 0 & \tfrac{1}{2}i & 0 \\
                0 & 0 & 0 & 0 & -i & 0 & 0 & 0 & -i}. $$
Then $n_3$ is a cocycle (in fact, $n_3^2=1$) in $\Nm_{\Gtil_0}(\Fm_3^\cC)$
projecting to $u_2$. Furthermore, $g_3^{-1}\ov{g}_3 = n_3$. So a real
representative of $\OOm_{x,y}$ is
$$p_{x,y}=
g_3\cdot q = x(e_{147}-2e_{169}-e_{245}+e_{289}-e_{356}-\tfrac{1}{2}e_{378}) + y(
e_{124}+e_{136}+\tfrac{1}{2}e_{238}-e_{457}+2e_{569}-e_{789}),$$
where $x,y\in \R$ satisfy the polynomial condition written above.

In \cite{BGL2021} the stabilizer $C_q$ of $q$ in $\SL(9,\C)$ is explicitly
determined. We have $C_q = T_4\rtimes H$, where $H$ is a
group of order 9 and $T_4$ a 4-dimensional torus consisting of
$$T_4(t_1,t_2,t_3,t_4)=
\diag\big(\hs t_1,t_2,(t_1t_2)^{-1},t_3,t_4,(t_3t_4)^{-1},(t_1t_3)^{-1},(t_2t_4)^{-1},
t_1t_2t_3t_4\hs\big),$$
for $t_1,t_2,t_3,t_4\in \C^*$. For $g\in C_q$ we set $\sigma_q(g) =
n_3\hs\ov g\hs n_3^{-1}$. Then
$$\sigma_q(T_4(t_1,t_2,t_3,t_4)) =
T_2(\bar t_1,\bar t_1^{-1}\bar t_3^{-1}, \bar t_1^{-1}\bar t_2^{-1},\bar t_1
\bar t_2 \bar t_3 \bar t_4).$$
Using this formula it can be shown (see \cite{BGL2021} for the details)
that $\Ho^1 (T_4,\sigma_q)=1$. Since the component group $H$ is of order 9, this
implies that $\Ho^1 (C_q,\sigma_q) = 1$ as well
(\cite[Proposition 3.3.16]{BGL2021}). It follows that $\OOm_{x,y}$ contains
one $\SL(9,\R)$-orbit with representative $p_{x,y}$.
\end{example}

\subsection{The mixed orbits}\label{subs:mixed}

For $\K=\C,\R$, let $p+e$ in $\g_1^\kK$ be a mixed element, that is, $p$ is
semisimple, $e$ is nilpotent, both are nonzero and $[p,e]=0$.
After acting with $\Gtil_0(\K)$,
we may assume that $p$ is a representative that we obtained when classifying
the semisimple orbits. If we fix such a semisimple element $p$,
then the classification of the
mixed elements with semisimple part equal to $p$ amounts to the
classification of the nilpotent elements in the intersection of the
centralizer of $p$ and $\g_1^\kK$,
up to the action of the stabilizer of $p$ in $\Gtil_0(\K)$.

First we briefly comment on the classification of the elements of mixed
type in $\g_1^\cC$. This classification was carried out in \cite{VE1978}.
From Section \ref{subs:ss} we recall that the
Cartan subspace contains seven canonical sets $\Fm_k^\cC$, $1\leq k\leq 7$ such
that every semisimple element of $\g_1^\cC$ is $\Gtil_0$-conjugate to an
element in precisely one of the $\Fm_k^\cC$. For the classification
of the elements of mixed type, the canonical sets $\Fm_1^\cC$ and $\Fm_7^\cC$
are not relevant, because the elements of $\Fm_1^\cC$ do not centralize non-unit nilpotent
elements and because $\Fm_7^\cC=\{0\}$.

Now  fix a semisimple element
$p\in \Fm_k^\cC$ and write $\a = \z_{\g^\cC}(p)$. Then $\a$ inherits the grading
from $\g^\cC$. The nilpotent elements $e$ such that $p+e$ is of mixed type lie in $\a_1$.
Furthermore, $p+e_1$, $p+e_2$ are $\Gtil_0$-conjugate if and only if
$e_1,e_2$ are $\Zm_{\Gtil_0}(p)$-conjugate. The Lie algebra of
$\Zm_{\Gtil_0}(p)$ is $\a_0$. Hence by using the method of \cite{Vinberg1979}
(see also \cite{Graaf2017}, Chapter 8) we can determine the
$\Zm_{\Gtil_0}(p)^\circ$-orbits of nilpotent elements in $\a_1$. Finally we
reduce the list by considering conjugacy under the elements of the component
group of $\Zm_{\Gtil_0}(p)$. The paper \cite{VE1978} contains lists of
representatives
of nilpotent parts of elements of mixed type, for $p \in \Fm_k^\cC$,
$2\leq k\leq 6$. Let $\Tt^\cC_p$ be the set of homogeneous $\ssl_2$-triples in
$\a$. Note that by Proposition \ref{prop:JMVL}, the nilpotent
$\Zm_{\Gtil_0}(p)$-orbits in $\a_1$ correspond bijectively to the
$\Zm_{\Gtil_0}(p)$-orbits in $\Tt^\cC_p$.
Note also that $\Zm_{\Gtil_0}(p)$ and $\z_{\g^\cC}(p)$ only depend
on $\Fm_k^\cC$, not on the particular element $p$
(Section \ref{subs:ss} and Lemma \ref{lem:Zp}).
Hence the orbits of the nilpotent parts also only depend only on $\Fm_k^\cC$.

Now we turn to the problem of classifying $\Gtil_0(\R)$-orbits of mixed type.
As in Section \ref{sec:tabsemsim}, we denote the set of real points of
$\Fm_k^\cC$ by $\Fm_k$.
We divide the real semisimple elements into two groups: those that are
$\Gtil_0(\R)$-conjugate to an element of some $\Fm_k$ (these are called
canonical semisimple elements) and those that are
not (which are called noncanonical semisimple elements).

Let $p$ be a representative of a semisimple $\Gtil_0(\R)$-orbit and
again let $\a=\z_{\g^\cC}(p)$. Similarly to the complex case, we need to
determine the $\Zm_{\Gtil_0(\R)}(p)$-orbits of real nilpotent elements in
$\a_1$. By Proposition \ref{prop:JMVL} these correspond bijectively to the
$\Zm_{\Gtil_0(\R)}(p)$-orbits in the set $\Tt_p$ of real homogeneous
$\ssl_2$-triples in $\a$. Let $t=(h,e,f)$ be a real homogeneous
$\ssl_2$-triple in $\a$ containing $e$. Let
\[\Zm_0(p,t) = \{ g\in \Zm_{\Gtil_0}(p)\, \mid\, g\cdot h =h,\, g\cdot e=e,\,
        g\cdot f=f\}.\]
Then by Proposition \ref{p:coh-orbits} there is a bijection between
the real $\Zm_{\Gtil_0(\R)}(p)$-orbits contained in the complex $\Zm_{\Gtil_0}(p)$-orbit
of $t$ and $\ker \big[ \Ho^1 \Zm_0(p,t) \to \Ho^1 \Zm_{\Gtil_0}(p)\big]$. In
\cite{BGL2021} it has been established that
$\Ho^1 \Zm_{\Gtil_0}(p)=1$ in all cases. Hence the orbits we are interested
in here correspond bijectively to $\Ho^1 \Zm_0(p,t)$.

The procedure that we use to classify the mixed elements whose semisimple
part is noncanonical is markedly more complex than the procedure for
classifying those with canonical semisimple part.
If the semisimple element $p$ is canonical, then we consider a real nilpotent
$e\in \a=\z_{\g^\cC}(p)$ lying in the homogeneous real $\ssl_2$-triple $t=
(h,e,f)$. In order to compute representatives of the real $\Gtil_0(\R)$-orbits
contained in the $\Gtil_0$-orbit of $p+e$, we compute
the centralizer $\Zm_0(p,t)$ and its Galois cohomology $\Ho^1 \Zm_0(p,t)$.
The elements of the latter set correspond to the real orbits that we are looking
for.

It is possible to take the same approach when $p$ is not canonical.
However, in that case the groups $\Zm_0(p,t)$ tend to be difficult to describe
(they can be nonsplit, for example) and therefore difficult to work with.
Moreover, in these cases it can happen that the complex $\Zm_{\Gtil_0}(p)$-orbit
of a nilpotent $e\in \a_1$ has no real points.

In the next paragraphs we describe a different method for this case.
The main idea is the following. We have that $p$ is conjugate over $\C$
to a canonical
semisimple element $q$. So if $t_1$ is a homogeneous $\ssl_2$-triple
in $\z_{\g^\cC}(p)$ then $t_1$ is $\Gtil_0$-conjugate to a homogeneous
$\ssl_2$-triple
$t$ in $\z_{\g^\cC}(q)$. Also the stabilizers $\Zm_0(p,t_1)$ and $\Zm_0(q,t)$ are
conjugate. We define a conjugation on the latter so that these two groups
are $\Gamma$-equivariantly isomorphic. By doing computations in $\z_{\g^\cC}(q)$,
we decide whether the complex orbit $\Zm_{\Gtil_0}(p)\cdot t_1$ has a real point.
If it does, we work with the preimage $t'$ in $\z_{\g^\cC}(q)$ of such a real point
and compute $\Ho^1 \Zm_0(q,t')$ where we use the modified conjugation.

Let $p\in \g_1$ be a real noncanonical semisimple element and of the form
$p = g\cdot q$, where $g\in \Gtil_0$ and $q\in \Fm_k^\cC$ where $2\leq k\leq 6$.
From our construction
(Proposition \ref{prop:ssorb}) it follows that setting $n=g^{-1}\ov g$
entails $n\in \Zl^1\Nm_{\Gtil_0}(\Fm_k^\cC)$ and $n^{-1}\cdot q = \ov q$.
The cocycles $n$ are explicitly given in \cite{BGL2021}.
In all cases it turns out that $n$ is real, so that
$n^2=1$. In the sequel these $g$ and $n$ are fixed.

Also define
\[\varphi\colon\, \z_{\g^\cC}(q)\cap \g_1^\cC\, \to\,
         \z_{\g^\cC}(p)\cap \g_1^\cC,\quad x\mapsto g\cdot x.\]
Because $\Zm_{\Gtil_0}(p) = g\Zm_{\Gtil_0}(q)g^{-1}$,
this is a bijection between the sets of nilpotent
elements in the respective spaces mapping $\Zm_{\Gtil_0}(q)$-orbits to
$\Zm_{\Gtil_0}(p)$-orbits. For $x\in \z_{\g^\cC}(q)\cap \g_1^\cC$,
the point $g\cdot x$ is real (that is, $\ov{g\cdot x} = g\cdot x$) if and
only if $n\hs\ov x = x$. Since $n\in \Nm_{\Gtil_0}(\Fm_k^\cC)$, we have that
$q$ and $n\cdot q$ both lie in $\Fm_k^\cC$ and therefore
have the same centralizer in $\g^\C$ (see Section \ref{subs:ss});
hence $n\cdot \z_{\g^\cC}(q) = \z_{\g^\cC}(q)$.
From $n^{-1}\cdot q=\ov q$ it follows  that $\z_{\g^\cC}(q)$ is stable under
complex conjugation $x\mapsto \ov x$. We set $\u = \z_{\g^\cC}(q)\cap \g_1^\cC$
and define $\mu \colon \u \to \u$ by $\mu(x) = n\ov x$. So for $x\in \u$ we have
that $\varphi(x)$ is real if and only if $\mu(x)=x$.
Because $n$ is a cocycle, we have $\mu^2(x) = x$ for all $x\in \u$.

Now fix a nilpotent $e\in \z_{\g^\cC}(q)\cap \g_1^\cC=\u$ lying in a homogenous
$\ssl_2$-triple $t=(h,e,f)$. We let $Y = \Zm_{\Gtil_0}(q)\cdot e\subset \u$
be its orbit.
Then $\varphi(Y)$ is a $\Zm_{\Gtil_0}(p)$-orbit in $\z_{\g^\cC}(p)\cap \g_1^\cC$
(and all nilpotent $\Zm_{\Gtil_0}(p)$-orbits in $\z_{\g^\cC}(p)\cap \g_1^\cC$ are
obtained in this way).
We want to determine the real $\Zm_{\Gtil_0(\R)}(p)$-orbits contained in
$\varphi(Y)$.

\begin{lemma}\label{lem:Ymu}
Let $y_0$ be any element of $Y$.
We have $\mu(Y)=Y$ if and only if $\mu(y_0)\in Y$.
\end{lemma}

\begin{proof}
Only one direction needs a proof, so suppose that $n\cdot\ov y_0 = g_1\cdot y_0$ for some
$g_1\in  \Zm_{\Gtil_0}(q)$.

Note that $\psi \colon \Zm_{\Gtil_0}(q) \to \Zm_{\Gtil_0}(p)$,
$\psi(h) = ghg^{-1}$, is an isomorphism. As $p$ is real, i.e., $\ov p = p$,
we have that $\Zm_{\Gtil_0}(p)$ is closed under conjugation. So for $h\in
\Zm_{\Gtil_0}(q)$ we see that $\psi^{-1} (\ov{\psi(h)})$ lies in $\Zm_{\Gtil_0}(q)$.
But the latter element is equal to $n \ov hn^{-1}$. We conclude that
$\Zm_{\Gtil_0}(q)$ is closed under $h\mapsto n \ov hn^{-1}$.

Let $g_2\in \Zm_{\Gtil_0}(q)$, then
\[\mu(g_2\cdot y_0) = n\cdot\ov{g_2\cdot y_0} = n \ov g_2 n^{-1}\cdot n\cdot\ov y_0
      = n \ov g_2 n^{-1}\cdot g_1\cdot y_0 =g_3\cdot y_0\]
with $g_3\in \Zm_{\Gtil_0}(q)$.
We conclude that $\mu(Y)=Y$.
\end{proof}

Now there are two possibilities. If $\mu(Y)\neq Y$, then $\varphi(Y)$ has no
real points by the previous lemma. In this case there are no real mixed elements
of the form $p+e_1$ conjugate to $q+e$. So we do not consider this $e$.
Note that we can check whether $\varphi(Y) = Y$ using Lemma \ref{lem:Ymu}
along with computational methods for determining conjugacy of nilpotent
elements that are outlined in \cite{BGL2021}.

On the other hand, if $\mu(Y)=Y$ then we consider the restriction of $\mu$ to
$Y$. We set $\YY = (Y,\mu)$. With the methods of Section \ref{sec:findreal}
we establish whether $\YY$ has a real point (that is, a point $y\in Y$ with $\mu(y)=y$)
and, if so, we find one.
If, on the other hand, $\YY$ does not have a real point, then $\varphi(Y)$
has no real points either and also in this case we do not consider this $e$.

We briefly summarize the main steps of the method of Section \ref{sec:findreal}
to find a real point in $\YY$. We set $H= \Zm_{\Gtil_0}(q)$ and define the
conjugation $\tau \colon H\to H$, $\tau(h) =n\ov hn^{-1}$. Set $\HH=(H,\tau)$ and we
assume that $\Ho^1\hs\HH=1$. Then we do the following
\begin{enumerate}
\item Compute $h_0\in H$ such that $\mu(e) = h_0^{-1}e$.
\item Set $C=\Zm_H (e)$ and let $\nu \colon C\to C$ be the conjugation defined by
  $\nu(c) = h_0\tau(c)h_0^{-1}$. Set $\CC = (C,\nu)$.
\item Set $d=h_0\tau(h_0)$ and consider the class $[d]\in \Ho^2 \CC$.
  If $[d]\neq 1$, then $\YY$ has no real point and we stop.
\item Otherwise we can find an element $c\in C$ with $c\hs\nu(c)\hs d=1$. Set $h_1 = c\hs h_0$.
  (Then $\mu(e) = h_1^{-1}\cdot e$ and $h_1\tau(h_1) = 1$.)
\item Because $\Ho^1\hs\HH=1$, we can find $u\in H$ such that $u\hs h_1\tau(u)^{-1} =1$.
  Then $y=u\cdot e$ is a real point of $\YY$.
\end{enumerate}

Now assume that $\YY$ has a real point $e'$.
Then we set $e_1 = g\cdot e'$ and find a real homogeneous $\ssl_2$-triple
$t_1=(h_1,e_1,f_1)$ in $\z_{\g^\cC}(p)$. Set $t'=(h',e',f')$ where
$h'=g^{-1}\cdot h_1$, $f'=g^{-1}\cdot f_1$. Furthermore, we determine an element $g' \in
\Zm_{\Gtil_0}(q)$ such that $g' \cdot t = t'$, then $gg'\cdot t = t_1$.
We set $g_0 = gg'$ and $n_0 = g_0^{-1}\ov g_0$. Then $n_0 \in \Zl^1
\Nm_{\Gtil_0}(\Fm_k^\cC)$.
We compute $Z_q = \Zm_{\Gtil_0}(q,t)$ (the stabilizer of $t$ in $\Zm_{\Gtil_0}(q)$).
We define $\sigma \colon Z_q \to Z_q$ by $\sigma(z) = n_0 \ov z n_0^{-1}$ (in the same
way as in the proof of Lemma \ref{lem:Ymu} it is seen that $Z_q$ is closed under
$\sigma$). We set $\mathbf{Z}_q =(Z_q,\sigma)$. Also, we set
$Z_p = \Zm_{\Gtil_0}(p,t_1)$
and $\mathbf{Z}_p =(Z_p,\ov{\phantom u})$. Then $\psi \colon \mathbf{Z}_q\to
\mathbf{Z}_p$, $\psi(z) = g_0zg_0^{-1}$ is a $\Gamma$-equivariant isomorphism.
Therefore the real orbits in $\varphi(Y)$ correspond bijectively to
$\Ho^1\hs\mathbf{Z}_q$.

The most straightforward way to find representatives of the orbits is to use
the group $\Zm_{\Gtil_0}(q)$, the conjugation $\sigma$ and the elements of
$\Ho^1\hs\mathbf{Z}_q$. For a class $[c]\in \Ho^1\hs\mathbf{Z}_q$ we find an
element $a\in \Zm_{\Gtil_0}(q)$ with $ac = \sigma(a) = n_0 \ov a n_0^{-1}$.
Then $g_0 a\cdot e$ is a nilpotent element in $\z_{\g^\cC}(p)\cap \g_1^\cC$
and $p+g_0 a\cdot e$ is a representative of the orbit of mixed elements
corresponding to the class $[c]$.

\begin{remark}\label{rem:adhc}
We have used an ad hoc method by which it is possible to show in many
cases that $Y$ has a real point. The map $\mu \colon \u\to \u$ is an $\R$-linear
involution. Set $\u_\R = \{ x \in \u \mid \ov x =x\}$.
In all cases that we have considered we have $\ov n=n$. Hence
we have $\mu(\u_\R)=\u_\R$.
We have $\u_\R = \u_+\oplus \u_-$, where the
latter are the eigenspaces of $\mu$ corresponding to the eigenvalues $1$ and
$-1$ respectively. We compute bases of these spaces and consider elements of
three forms: $u\in \u_+$, $iu$ for $u\in \u_-$ and $u_1+iu_2$ for $u_1\in \u_+$,
$u_2\in \u_-$. All these elements are fixed by $\mu$. We list many of these
elements that are nilpotent and check to which $\Zm_{\Gtil_0}(q)$-orbit they
belong. This way we have found real points in most cases.
\end{remark}

\begin{example}
Here we consider the same situation as in Example \ref{exa:Fm3}. We wish
to list the mixed orbits whose semisimple part is $p_{x,y}$. We have
$p_{x,y} = g_3\cdot q$ where $q\in \Fm_3^\cC$. Furthermore $g_3^{-1}\ov g_3 = n_3$.
As above we set $\u = \z_{\g^\cC}(q) \cap \g_1^\cC$ and we define $\mu \colon \u\to\u$
by $\mu(x) = n_3\ov x$.

We know the classification of the nilpotent $\Zm_{\Gtil_0}(q)$-orbits in
$\u$. This classification is given in \cite[Table 2]{VE1978} and also in
Table \ref{tab:fam3_1} (this table
contains the real classification, but here this happens to coincide with
the complex classification). We consider the third nilpotent element
$e=e_{159}+e_{168}+e_{267}$, which lies in a homogeneous $\ssl_2$-triple
$t=(h,e,f)$ (which we do not explicitly describe).
By the ad hoc method of the previous remark we
found an element $e'$ in the $\Zm_{\Gtil_0}(q)$-orbit of $e$ with $\mu(e')=e'$.
It is $e'=-e_{159}+ie_{168}-e_{267}$.

Now we set $e_1=g_3\cdot e'=-2e_{267}-\tfrac{1}{2}e_{349}+\tfrac{1}{4}e_{468}$.
We compute a homogeneous $\ssl_2$-triple $t_1=(h_1,e_1,f_1)$ and set
$h'= g_3^{-1} \cdot h_1$, $f'= g_3^{-1}\cdot f_1$ (here we do not give
explicit expressions for these elements). We determine an element $g'\in
\Zm_{\Gtil_0}(q)$ such that $g'\cdot t=t'$ and set $g_0=gg'$. Then
$g_0\cdot t=t_1$. Furthermore, we set $n_0 = g_0^{-1}\ov g_0$, so that
$n_0 \in \Zl^1 \Nm_{\Gtil_0}(\Fm_3^\cC)$. We have
  $$g_0=\SmallMatrix{ 0 & 0 & i & -1 & 0 & 0 & 0 & 0 & 0 \\
                0 & -1 & 0 & 0 & 0 & 0 & i & 0 & 0 \\
                0 & 0 & 0 & 0 & -\tfrac{1}{2}i & 0 & 0 & 0 & \tfrac{1}{2} \\
                -\tfrac{1}{2} & 0 & 0 & 0 & 0 & 0 & 0 & 0 & 0 \\
                0 & 0 & 1 & -i & 0 & 0 & 0 & 0 & 0 \\
                0 & 0 & 0 & 0 & 0 & -1 & 0 & 0 & 0 \\
                0 & i & 0 & 0 & 0 & 0 & -1 & 0 & 0 \\
                0 & 0 & 0 & 0 & 0 & 0 & 0 & \tfrac{1}{2} & 0 \\
                0 & 0 & 0 & 0 & 1 & 0 & 0 & 0 & -i},
  \quad
  n_0=\SmallMatrix{ 1 & 0 & 0 & 0 & 0 & 0 & 0 & 0 & 0 \\
                0 & 0 & 0 & 0 & 0 & 0 & i & 0 & 0 \\
                0 & 0 & 0 & i & 0 & 0 & 0 & 0 & 0 \\
                0 & 0 & i & 0 & 0 & 0 & 0 & 0 & 0 \\
                0 & 0 & 0 & 0 & 0 & 0 & 0 & 0 & i \\
                0 & 0 & 0 & 0 & 0 & 1 & 0 & 0 & 0 \\
                0 & i & 0 & 0 & 0 & 0 & 0 & 0 & 0 \\
                0 & 0 & 0 & 0 & 0 & 0 & 0 & 1 & 0 \\
                0 & 0 & 0 & 0 & i & 0 & 0 & 0 & 0}. $$
Let $Z_q= \Zm_{\Gtil_0}(t,q)$. The identity component of this group is a
1-dimensional torus $T_1$ consisting of the elements
$$T_1(s) = \diag(1,s,s^{-1},s,s^{-1},1,s^{-1},1,s).$$
The component group is of order 9 (in \cite{BGL2021} it is explicitly given).
We define the conjugation $\sigma \colon Z_q\to Z_q$ by $\sigma(z) = n_0 \ov x
n_0^{-1}$. A small calculation shows that $\sigma(T_1(s)) = T_1(\ov s^{-1})$.
This implies that $\Ho^1 (T_1,\sigma) = \{[1],[T_1(-1)]\}$ (see
\cite[Examples 3.1.7(3)]{BGL2021}). By \cite[Proposition 3.3.16]{BGL2021}
it follows that $\Ho^1 (Z_q,\sigma) = \{[1],[T_1(-1)]\}$.

Next we set $c=T_1(-1)$ and find an element $a\in \Zm_{\Gtil_0}(q)$ with
$ac=n_0\hs\ov a\hs n_0^{-1}$. It is
$$a=\diag(-1,-i,-i,-i,i,1,-i,1,i).$$
We have $g_0a\cdot e = 2e_{267}-\tfrac{1}{2}e_{349}-\tfrac{1}{4}e_{468}$.
So we get two real mixed orbits with representatives $q_{x,y}+e_1$ and
$q_{x,y}+g_0a\cdot e$. They are both $\SL(9,\C)$-conjugate to
$p+e$.
\end{example}

%\bibliographystyle{plain}
%\bibliography{article}

\end{document}